\documentclass[a4paper, reqno, draft]{amsart}
\usepackage{amssymb,latexsym}
\usepackage{amsmath}
\usepackage{amsthm}
\usepackage{graphicx}
\usepackage{hyperref}
\usepackage{titletoc}
\usepackage{palatino,mathpazo}
\usepackage{amsthm, amsmath, amssymb, amsfonts, amsopn, color, mdwlist, geometry}
\usepackage{enumerate}
\numberwithin{equation}{section}
\newtheorem{theorem}{Theorem}[section]
\newtheorem{proposition}{Proposition}[section]
\newtheorem{lemma}{Lemma}[section]
\newtheorem{corollary}{Corollary}[section]

\newtheorem{definition}{Definition}[section]
\theoremstyle{definition}

\newtheorem{remark}{Remark}

\def\N{\mathbb{N}}

\def\XXint#1#2#3{{\setbox0=\hbox{$#1{#2#3}{\int}$}
     \vcenter{\hbox{$#2#3$}}\kern-.5\wd0}}

\def\e{\varepsilon}

\def\R{{\mathbb R}}

\def\RN{\mathbb{R}^2}
\def\ulm{u_{\lambda,\mu}}

\def\lm0{\lambda,\mu>0}
\def\mlt{\lambda,\mu\to\infty,\ \frac{\lambda}{\mu}\to0}
\def\pl{\phi_{\lambda,\mu}}
\def\nl{N_{\lambda,\mu}}
\def\wlm{w_{\lambda,\mu}}

\begin{document}

\title[Periodic   Maxwell-Chern-Simons  vortices with concentrating property]{Periodic   Maxwell-Chern-Simons  vortices with concentrating property}

\author{Weiwei Ao}
\address[Weiwei Ao]
{Wuhan University,
Department of Mathematics and Statistics, Wuhan, 430072, PR China}
\email{wwao@whu.edu.cn}

\author{Ohsang Kwon}
\address[Ohsang Kwon]
{Department of Mathematics, Chungbuk National University, Chungdae-ro 1, Seowon-gu, Cheongju, Chungbuk 362-763, Korea}
\email{ohsangkwon@chungbuk.ac.kr}

\author{Youngae Lee}
\address[Youngae Lee]
{Department of Mathematics Education, Teachers College, Kyungpook National University, Daegu, South Korea}
\email{youngaelee@knu.ac.kr}

\begin{abstract}
In order to study electrically and magnetically charged vortices in  fractional quantum Hall effect and  anyonic superconductivity, the Maxwell-Chern-Simons (MCS) model was introduced by [Lee, Lee, Min (1990)] as a unified system of the classical Abelian-Higgs model (AH) and the Chern-Simons (CS) model.  In this article, the first  goal is to obtain  the uniform (CS) limit result of (MCS) model  with respect to  the Chern-Simons parameter without any restriction on either   a particular class of solutions or  the number of vortex points.   The most important step for this purpose is to derive the relation between  the Higgs field and the neutral scalar field. Our (CS) limit result also provides  the critical clue to answer the open problems raised by  [Ricciardi,Tarantello (2000)] and [Tarantello (2004)], and we succeed to establish the existence of periodic Maxwell-Chern-Simons vortices satisfying the concentrating
property of the density of superconductive electron pairs. Furthermore, we  expect that the (CS) limit analysis  in this paper would  help to study the stability, multiplicity, and bubbling phenomena for solutions of  the (MCS) model.
\end{abstract}

\date{\today}

\maketitle

\section{Introduction}\label{sec1}

As the pioneering work by Ginzburg and Landau, the classical Abelian-Higgs  (AH) model (or, Maxwell-Higgs)
was proposed in order to describe the superconductivity phenomena  at low temperature  (see \cite{BBH2,JT,La,PR}).  This model has been studied   in \cite{BGP,JT,Ta,WY} for various domains. However, (AH) model  can only describes electrically neutral  vortices, which  are static solutions of the corresponding Euler-Lagrange equation.    In order to  study  the  fractional quantum Hall effect and  high temperature superconductivity, we should investigate electrically and magnetically "charged" vortices. For this purpose,  one might attempt  to include Chern-Simons (CS) term into (AH) model. However, just adding (CS) term into (AH) model loses the  self-dual structure, which is characterized by a special class of static solution corresponding to a constrained energy minimizer. The self-dual equation has a benefit  in the gauge field theory since it is  a reduced first-order equation, so called "Bogomol'nyi  equation", for the more complicated second order equation of motion (see \cite{B, NO}). In order to obtain a  self-dual Chern-Simons theory,  Hong-Kim-Pac in  \cite{HKP} and Jackiw-Weinberg in \cite{JW} independently proposed a model for  charged vortices with electrodynamics governed only  by the (CS) term without Maxwell term, which was included in the (AH) model. This pure (CS) model was suggested from  the observation such that  the (CS) term is dominant over the Maxwell term in the large scale. During the last few decades,  the   (CS) model has been extensively studied in \cite{CI, CFL, C2, SY, SY1, W, Y} for entire solutions on a full space,  in \cite{CaY, C2, C3, DJLW, DJLW2, LJLPW, H, H2, H3, NT, ST, T} for the periodic case, and in \cite{HJ} for bounded domains (see also  \cite{ChC, CC, C1, H4, Kim, KK, Ku, S}).

As stated above,  a naive inclusion of both (AH) term and (CS) term in the Lagrangian   fails to make the system self-dual. However,   in \cite{LLM}, Lee, Lee, and  Min  succeeded in   restoring the self-duality in  Maxwell-Chern-Simons (MCS) model as a  unified self-dual system of (AH) and (CS), by  introducing a neutral scalar field.  Moreover, the authors in \cite{LLM} showed formally that  the self-dual equation of (MCS) owns both (AH) model and (CS) model as limiting problems according to the limit behavior of the electric charge and the Chern-Simons mass scale (see also \cite{Du}). This formal argument in \cite{LLM,Du}  could be supported with mathematically rigorous proof in \cite{ChK, ChK2,  R, RT}. In \cite{ChK}, Chae and Kim established the existence   of topological multivortex solution for (MCS) model in a full space $\mathbb{R}^2$.   Here, the topological entire solution in $\RN$ satisfies the specific boundary condition such that its first component vanishes at infinity. Moreover, the authors in \cite{ChK} showed the  convergence  of  topological multivortex solutions to the (CS) model and (AH) model. The convergence  depends on the asymptotic behavior of  the electric charge and the Chern-Simons mass scale. In \cite{ChK2}, they also obtained the corresponding result  for topological solutions on a flat two torus (see    \eqref{def_sol} for  the definition of topological solution on a flat two torus).    In \cite{RT}, Ricciardi and Tarantello showed that  there exist at least two gauge distinct periodic multivortices (topological solution and mountain pass solution), and analyzed their asymptotic behavior in terms of the (CS) limit and the (AH) limit.   Moreover, Ricciardi in \cite{R} obtained the  stronger convergence
result for an arbitrary sequence of  periodic multivortices while the Chern-Simons parameter, which is the ratio between  electric charge and the Chern-Simons mass scale, is fixed.

 In this article, one of main  goals is to improve the (CS) limit analysis for  (MCS) model without any restriction on either a particular class of solutions, the number of vortex points,  or the Chern-Simons parameter. Moreover, in view of our first result, we  could also
obtain the affirmative   answers for the open problems raised by Ricciardi and Tarantello in \cite{RT}, and Tarantello in \cite{T2}.

In order to introduce our results more precisely, let us recall   the  Lagrangian density $\mathcal{L}^{MCS}$ for the  (MCS)  model, which is defined in the $(2+1)$-dimensional Minkowski space $\mathbb{R}^{2,1}$ with the metric $\textrm{diag}(1,-1,-1)$:
\begin{equation}\begin{aligned}\label{lagrangian}
\mathcal{L}^{MCS}(A,\phi,\mathfrak{n})&=-\frac{1}{4q^2}F_{\alpha\beta}F^{\alpha\beta}-\frac{\mu}{4q^2}\epsilon^{\alpha\beta\gamma}A_{\alpha}F_{\beta\gamma}+D_{\alpha}\phi\overline{(D^{\alpha}\phi)}+\frac{1}{2q^2}\partial_{\alpha}\mathfrak{n}\partial ^{\alpha}\mathfrak{n}
\\&\quad-|\phi|^2\left(\mathfrak{n}-\frac{q^2}{\mu}\right)^2-\frac{q^2}{2}\left(|\phi|^2-\frac{\mu}{q^2}\mathfrak{n}\right)^2,
\end{aligned}\end{equation}
 where the metric is used to raise or lower indices, all the Greek indices run over $0,1,2$, and $\epsilon^{\alpha\beta\gamma}$ is the totally skew-symmetric tensor fixed so that $\epsilon^{012}=1$.    Here,  $\phi:\mathbb{R}^{1+2}\to\mathbb{C}$ is the complex valued Higgs field, $\mathfrak{n}:\mathbb{R}^{1+2}\to\mathbb{R}$ is the neutral scalar field,    $A_{\alpha}: \mathbb{R}^{1+2}\to\mathbb{R}$ is the gauge field,  $D_{\alpha}=\partial_{\alpha}-iA_{\alpha}$ is the gauge covariant derivative   associated with
$A_{\alpha}$
where $i=\sqrt{-1}$,  and   $F_{\alpha\beta}=\partial_{\alpha}A_{\beta}-\partial_{\beta}A_{\alpha}$ is the field strength.     The constant $q>0$ denotes the electric charge and $\mu>0$ is the Chern-Simons mass scale.
The gauge potential field $\mathcal{A}$ with a 1-form (connection) is identified  as
$\mathcal{A}=-i A_\alpha d x^{\alpha}$, and the Maxwell gauge field  $F_{\mathcal{A}}$ is expressed by $ F_{\mathcal{A}}= \mathrm{d} \mathcal{A}=-\frac{i}{2} F_{\alpha \beta} \mathrm{d} x^{\alpha} \wedge \mathrm{d} x^{\beta}$ is expressed by   the 2-form (curvature) .
Let us denote the self-dual potential by \begin{equation*}V(|\phi|,\mathfrak{n})=|\phi|^2\left(\mathfrak{n}-\frac{q^2}{\mu}\right)^2+
\frac{q^2}{2}\left(|\phi|^2-\frac{\mu}{q^2}\mathfrak{n}\right)^2.\end{equation*}
Note that in $\mathcal{L}^{MCS}$,   the Maxwell term for $\mathcal{A}$ is denoted by $F_{\alpha\beta}F^{\alpha\beta}$ and   the Chern-Simons term is represented by the quantity $\frac{\mu}{4q^2}\epsilon^{\alpha\beta\gamma}A_{\alpha}F_{\beta\gamma}$.
Indeed, the Lagrangian of the (AH) model and  the  (CS) model are given by \[\mathcal{L}^{{AH}}(A, \phi)=-\frac{1}{4 q^{2}} F_{\alpha \beta} F^{\alpha \beta}+D_{\alpha} \phi\overline{\left(D^{\alpha} \phi\right)}-\frac{q^{2}}{2}\left(|\phi|^{2}-1\right)^{2},\] and
  \[\mathcal{L}^{C S}(A, \phi)=-\frac{\mu}{4q^2} \varepsilon^{\alpha \beta \gamma} A_{\alpha} F_{\beta\gamma}+D_{\alpha} \phi\overline{\left(D^{\alpha} \phi\right)}-\frac{q^4}{\mu^2}|\phi|^{2}\left(|\phi|^{2}-1\right)^{2},\]
respectively.
If we  fix $q$, and  assume the identity $\mathfrak{n}=\frac{q^{2}} {\mu}$ in \eqref{lagrangian}, then as  $\mu\to0$, a limiting Lagrangian for  $\mathcal{L}^{MCS}$ formally would be  $\mathcal{L}^{ {AH}}$. On the other hand, if we fix $\frac{ q^{2}}{\mu}$, and  insert the identity $\mathfrak{n}=\frac{q^{2}} {\mu}|\phi|^{2}$ into the potential of  $\mathcal{L}^{MCS}$, then as  $\mu\to\infty$, a limiting Lagrangian for  $\mathcal{L}^{MCS}$ formally would be $\mathcal{L}^{ {CS}}$.

The periodic patterns of vortex configurations have been predicted and founded in the experiment for the study of superconductivity (see \cite{Ab}). Periodic vortices (or condensates) relative to \eqref{lagrangian} are defined as the static solutions, which is independent of the $x^0$-variable, for the following Euler-Lagrangian equations subject to  the 't Hooft type periodic boundary conditions  (see \cite{'tH}) :
 \begin{equation}\label{ELeq}
\left\{\begin{array}{l}
D_{\alpha}D^{\alpha}\phi=-\frac{\partial V}{\partial \bar{\phi}},\\
\frac{1}{q^2}\partial_{\alpha}\partial^{\alpha}\mathfrak{n}=-\frac{\partial V}{\partial \mathfrak{n}},\\ \frac{1}{q^2}\partial_{\beta}F^{\alpha\beta}+\frac{\mu}{2q^2}\epsilon^{\alpha\beta\gamma}F_{\beta\gamma}=J^{\alpha},
\end{array}
\right.
\end{equation}
where $J^{\alpha}=i(\bar{\phi}D^{\alpha}\phi-\phi\overline{(D^{\alpha}\phi)})$ is the conserved current for the system.
We say that $\left(\phi,  A_{\alpha}, \mathfrak{n}_1\right)$ is gauge equivalent to $\left(\psi, B_{\alpha}, \mathfrak{n}_2\right),$ if there exists a smooth function $\omega$ satisfying
\[\left(\psi,  B_{\alpha}, \mathfrak{n}_2\right)=\left(e^{i\omega}\phi,  A_{\alpha}+\partial_{\alpha}\omega, \mathfrak{n}_1\right), \ \alpha=0,1,2.\] The 't Hooft type periodic boundary conditions are required for  the invariance of \eqref{ELeq} with respect to the gauge transformation.
More precisely,      the periodic cell domain is given by
\[\Omega=\left\{x \in \mathbb{R}^{2}\  |\   x=t_{1} \mathbf{a}_{1}+t_{2} \mathbf{a}_{2},\  t_{1}, t_{2}\in (0,1)\right\},\] where   $\mathbf{a}_{1}$ and $\mathbf{a}_{2}$  are   linearly independent vectors    in $\mathbb{R}^{2}.$ Let
$\Gamma_{k}=\left\{x \in \mathbb{R}^{2} \ | \  x=t_{k} \mathbf{a}_{k}, \  t_{k}\in(0,1)\right\}, k=1,2,$  be a part of the boundary of $\Omega$.
 We assume that  $(A, \phi,\mathfrak{n})$ is a static (that is, independent of the $x^0$-variable)  solution of \eqref{ELeq}, and there exist smooth functions $\omega_{k}$, ($k=1,2$)  in a
neighborhood of $\Gamma_{1} \cup \Gamma_{2} \backslash \Gamma_{k}$,  satisfying
\begin{equation}\label{1.2}
\left\{\begin{array}{ll}{A_{j}\left(x+\mathbf{a}_{k}\right)=A_{j}(x)+\partial_{j} \omega_{k}(x),} \quad {j, k=1,2,} \\ {A_{0}\left(x+\mathbf{a}_{k}\right)=A_{0}(x)}, \\ {\phi\left(x+\mathbf{a}_{k}\right)=e^{-i \omega_{k}(x)} \phi(x)}, \\ {\mathfrak{n}\left(x+\mathbf{a}_{k}\right)=\mathfrak{n}(x), } \quad  k=1,2, \end{array}\right.
\end{equation}
for  $x \in  \Gamma_{1} \cup \Gamma_{2} \backslash \Gamma_{k}$, $k=1,2$.  We set
$\omega_{k}\left(s^{1}, s^{2}\right)=\omega_{k}\left(s^{1} \mathbf{a}_{1}, s^{2} \mathbf{a}_{2}\right),   k=1,2$ so that   $\phi$ is  single-valued in $\Omega.$
In view of the compatibility condition,  we have
\begin{equation}\label{1.3}\omega_{1}\left(0,0^{+}\right)-\omega_{1}\left(0,1^{-}\right)+\omega_{2}\left(1^{-}, 0\right)-\omega_{2}\left(0^{+}, 0\right)=2 \pi \mathfrak{M},\end{equation}
where  $\mathfrak{M}\in \mathbb{Z}_+$ is called the vortex number and coincides with the
total number of zeroes of $\phi$ in $\Omega$ counted according to their multiplicities.

 Since the Euler-Lagrangian equation \eqref{ELeq} is very complicated to study
even for stationary solution, we restrict to consider energy minimizers only.   It is well known from the arguments in \cite{B} that   a global minimizer of  static energy  on
suitable function spaces  is achieved by the following self-dual
equations:
 \begin{equation}\label{self}
\left\{\begin{array}{l}{\left(D_{1}+i D_{2}\right) \phi=0} \\ {F_{12}=q^{2}|\phi|^{2}-\mu \mathfrak{n}} \\ {-A_{0}=\mathfrak{n}-\frac{q^{2}}{\mu}} \\ {-\Delta A_{0}+\mu F_{12}=-2 q^{2} A_{0}|\phi|^{2}}\end{array}\right.
\end{equation} together with the boundary conditions \eqref{1.2}.  Due to
Jaffe-Taubes argument in \cite{JT, Ta},  the self-dual equation
\eqref{self}  is reduced to the following elliptic system (see \cite{ChK, Du, HK, RT, T} for the detail):
\begin{equation}\label{eq1}
\left\{\begin{array}{l}
\Delta u=\lambda\mu e^u-\mu N+4\pi\sum_{i=1}^n m_{i}\delta_{p_i},\\
\Delta N=\mu (\mu+\lambda e^u)N-\lambda \mu(\mu +\lambda)e^u,
\end{array}
\right.  \textrm{in} \ \ \Omega.
\end{equation}
where  $u=\ln|\phi|^2$, $\lambda=\frac{2q^2}{\mu}$, and  $N=2\mathfrak{n}$.  Here,  $\delta_{p_i}\in \Omega$ stands for the
Dirac measure concentrated at $p_i$, and $p_i\neq p_j$ if $i\neq j$.  Each $p_i$ is called a vortex point  and  $m_i\in\mathbb{N}$ is  the multiplicity of $p_i$.

In view of Remark \ref{rem1} below, the equation \eqref{eq1}  has  two different kinds of periodic solutions
satisfying one of the following asymptotic behaviors:\begin{equation}
\label{def_sol}
\begin{aligned}
  \left(u_{\lambda,\mu},\frac{N_{\lambda,\mu}}{\lambda}\right) \to (0,1) \quad\mbox{a.e. on }~ \Omega \quad\mbox{as }~ \lambda\to\infty,\  \mu\gg\lambda,
 \quad&\mbox{(topological solution)} \\
  \left(u_{\lambda,\mu},\frac{N_{\lambda,\mu}}{\lambda}\right) \to (-\infty,0) \quad\mbox{a.e. on }~ \Omega \quad\mbox{as }~ \lambda\to\infty,\  \mu\gg\lambda,
 \quad&\mbox{(nontopological solution)}
\end{aligned}
\end{equation}
Among the results obtained in \cite{ChK2, RT, R} for (MCS) model, let us review  the  (CS) limit   results for \eqref{eq1} on a flat two torus $\Omega$.  In \cite{ChK2},  Chae and Kim showed the existence of topological solution for \eqref{eq1}, and its (CS) convergence whenever $\mu\to\infty$ and $\lambda$ is fixed (see \cite{ChK} for the study  in $\mathbb{R}^2$).   In \cite{RT}, Ricciardi and Tarantello extended the (CS) limit  to other class of solutions. They  showed that there exists $\lambda_0>0$ sufficiently large such that  for any $\lambda>\lambda_0$, there is $\mu_{\lambda}>0$   satisfying that  if $\mu>\mu_{\lambda}$ , then \eqref{eq1} has at least two distinct solutions, topological solution and mountain pass solution, which converge to (CS) multivortices as $\mu\to\infty$.  Moreover, they derived the asymptotic behavior of these (CS) multivortices for  not only topological solution but also mountain pass solution provided $\mathfrak{M}=\sum_{i=1}^nm_i=1$ as $\lambda\to\infty$.  In \cite{R}, Ricciardi improved the results \cite{ChK2, RT} by obtaining  the (CS) limit for arbitrary sequence of solutions in $C^q$ norm for any $q\ge0$ whenever $\lambda=1$.

For given arbitrary configuration of vortex points,  our first goal is   to obtain the uniform (CS) limit result of (MCS) model for any class of solutions for \eqref{eq1} with large $\lambda, \mu>0$, and derive  the following Brezis-Merle type alternatives for (MCS) model.
\begin{theorem}\label{BrezisMerletypealternatives}
Let $Z\equiv\cup_{i}\{p_{i}\}$.
We assume that  $\{(u_{\lambda,\mu}, N_{\lambda,\mu})\}$ is a sequence of solutions of \eqref{eq1}.  Then
\begin{equation}\label{connection}\lim_{\lambda,\mu\to\infty,\ \frac{\lambda}{\mu}\to0}\left\|e^{u_{\lambda,\,\mu}}-\frac{N_{\lambda,\mu}}{\lambda}\right\|_{L^\infty(\Omega)}=0.\end{equation}
Moreover, as $\mlt$, up to subsequences, one of the following holds:

(i) $u_{\lambda,\mu}\to 0$ uniformly on any compact subset of $\Omega\setminus  Z$;

(ii)   $\ulm+2\ln\lambda-u_0\to\hat{w}$ in $C^{1}_{\textrm{loc}}(\Omega)$, where $\hat{w}$ satisfies  $\Delta \hat{w} +e^{\hat{w}+u_0}=4\pi\mathfrak{M}$;

(iii)  there exists a nonempty finite set $B=\{\hat{q}_1,\cdots,\hat{q}_k\}\subset\Omega$ and $k$-number of sequences  of points $q^{j}_{\lambda,\mu}\in\Omega$ such that $\lim_{\lambda,\mu\to\infty,\ \frac{\lambda}{\mu}\to0}q^{j}_{\lambda,\mu}=\hat{q}_j$, $\left(u_{\lambda,\mu}+2\ln\lambda\right)(q^{j}_{\lambda,\mu})\to+\infty$, and $u_{\lambda,\mu}+2\ln\lambda\to-\infty$ uniformly on any compact subset of $\Omega\setminus B$. Moreover,
\begin{equation*}\lambda^2 e^{u_{\lambda,\mu}}\left(1-\frac{N_{\lambda,\mu}}{\lambda}\right)\to\sum_j\alpha_j\delta_{\hat{q}_j},\ \ \alpha_j\ge8\pi,\end{equation*}
in the sense of measure.
\end{theorem}
 The most important step in the proof for Theorem \ref{BrezisMerletypealternatives} is to derive the relation \eqref{connection} between $u_{\lambda,\mu}$ and $N_{\lambda,\mu}$. In order to achieve this purpose, we apply the Green's representation formula for  the gradient estimation of $u_{\lambda,\mu}$, and use  the nondegeneracy of the operator $-\Delta +1$ in $\mathbb{R}^2$ after a suitable scaling.

We note that the elliptic system \eqref{eq1} is equivalent to
\begin{equation}\label{main_eq}
\left\{\begin{array}{l}
\Delta \left( u+\frac{N}{\mu}\right)=-\lambda^2 e^u\left(1-\frac{N}{\lambda}\right)+4\pi\sum_{i=1}^n m_{i}\delta_{p_i},\\
\Delta N=\mu^2 (1+\frac{\lambda}{\mu} e^u)N-\lambda \mu^2\left(1+\frac{\lambda}{\mu}\right)e^u
\end{array}
\right. \mbox{ in }\Omega.
\end{equation}
 To the best of our knowledge, the estimation \eqref{connection} in Theorem \ref{BrezisMerletypealternatives} has been known for a  fixed constant $\lambda>0$ as $\mu\to\infty$. We improve this result holds uniformly for large $\lambda>0$ satisfying $\lambda\ll\mu$.
Due to the estimation \eqref{connection}, \eqref{main_eq} would be regarded  as a perturbation of the following equation arising from (CS) model:
\begin{equation}\label{cs_eq}
 \Delta u=-\lambda^2 e^u\left(1- e^u\right)+4\pi\sum_{i=1}^n m_{i}\delta_{p_i}
\quad \mbox{ in }\Omega.
\end{equation}
The   corresponding result (i)-(iii) in  Theorem \ref{BrezisMerletypealternatives} for (CS) equation \eqref{cs_eq} has been proved in \cite{CK}  based on the arguments for Brezis-Merle type alternatives (see \cite{BCCT,BT,BM,CK, NT0,NT}).
However, since our case is the coupled  system problem, a major  obstacle arises from the interaction between two components $u_{\lambda,\mu}$ and $N_{\lambda,\mu}$. In order to overcome this difficulty, we should carry out a careful estimation for the gradient of $N_{\lambda,\mu}$ in the Pohozaev identity.

In \cite{RT},  the authors made a conjecture such that   the density of superconducting particles $e^{u_{\lambda,\mu}}$   of \eqref{eq1} converges   to $e^{u_\lambda}$  of \eqref{cs_eq} as $\mu\to\infty$  without the restriction $\mathfrak{M} = 1$, and it was proved in \cite{R} for fixed $\lambda=1$.   This result would be valid even uniformly for $\lambda>0$ since \eqref{main_eq} and \eqref{cs_eq} share   the similar asymptotic behavior   in  (i)-(iii) of Theorem \ref{BrezisMerletypealternatives} for  any sequence of solutions to \eqref{eq1} including even  mountain pass solution  and for any $\mathfrak{M}>0$. Moreover, we can improve the  (CS) convergence for blow up solutions, which are constructed below, in terms of not only $e^{u_{\lambda,\mu}}$ but also $u_{\lambda,\mu}$. We will continue to discuss the detail of uniform (CS) convergence for arbitrary solutions in forthcoming paper.

Now we consider the asymptotic behavior (iii)   in  Theorem \ref{BrezisMerletypealternatives}. The case (iii) is called blow up phenomena. More precisely, we define the blow up solutions as follows:
\begin{definition} Let $B= \{\hat{q}_j\}_{j=1}^k\subset\Omega$ be a set of finite points.  If  $\{(u_{\lambda,\mu}, N_{\lambda,\mu})\}$ is a family of solutions of \eqref{eq1} and there exist $k$-number of sequence of points $q^{j}_{\lambda,\mu}$, $j=1,\cdots,k$, satisfying

(i) $\lim_{\mlt}\left(u_{\lambda,\mu}+2\ln\lambda\right)(q^{j}_{\lambda,\mu})=+\infty,$ and

(ii) $\lim_{\mlt}q^{j}_{\lambda,\mu}=\hat{q}_j$,

 then $B$ is called a blow-up set and
$\{(u_{\lambda,\mu}, N_{\lambda,\mu})\}$ is called a family of bubbling solutions (or blow up solutions) of \eqref{eq1} at $B$
\end{definition}
In view of Theorem  \ref{BrezisMerletypealternatives}, we note that the blow up phenomena implies the concentration of density for the nonlinear terms in the first equation in \eqref{main_eq}.
We emphasize that this observation provides the affirmative answer for the open problem raised in \cite{T2}.  In other words, we would like to show  the existence of blow up solutions with the concentrating
property at the vortex points. It turns out that the construction of  solutions blowing up outside vortex points, that is, at the regular points, is more difficult than at the vortex points since the limit problem for the first one has nontrivial kernel. We first construct   solutions blowing up at a regular point, and continue to study solutions blowing up at a vortex point. 
\begin{theorem}\label{thm1}

Assume $\mathfrak{M}>2$.
Let $\hat{q}$ be  a non-degenerate critical point of $u_0$ defined in (\ref{u_0}). Assume that  $\lambda, \ \mu $  are large enough and satisfy  $(\ln\lambda)\lambda^2\ll\mu$.  Then  (\ref{eq1}) has a solution $(u_{\lambda,\mu}, N_{\lambda,\mu})$ satisfying

(i) $\lambda^2 e^{u_{\lambda,\mu}}\left(1-\frac{N_{\lambda,\mu}}{\lambda}\right)\to 4\pi \mathfrak{M}\delta_{\hat{q}}$ in the sense of measure as $\lambda, \mu  \to \infty, \frac{(\ln\lambda)\lambda^2}{\mu}\to0$,

(ii) $\max_{y\in \Omega}u_{\lambda,\mu}(y)\geq c$ for some constant $c\in\mathbb{R}$, and

(iii)  $\frac{N_{\lambda,\mu}}{\lambda}\to 0$ uniformly on any compact subset of  $\Omega\setminus \{\hat{q}\}$ as $\lambda, \mu  \to \infty, \frac{(\ln\lambda)\lambda^2}{\mu}\to0$.
\end{theorem}
\begin{remark}
By integrating  the first equation of 	(\ref{main_eq}), we have 
\begin{equation*}
\int_{\Omega}\lambda^2e^{u}(1-\frac{N}{\lambda})=4\pi\mathfrak{M}.	
\end{equation*}
Moreover, in view of  Lemma \ref{lemma2.1} below, one knows that the local mass of the Chern-Simons equation without vortex points is strictly greater than $8\pi$. So necessarily one has  $4\pi\mathfrak{M}>8\pi$, that is, $\mathfrak{M}>2$. This implies that when there is only one vortex point with multiplicity one, there should be no such kind of bubbling solutions considered in Theorem \ref{thm1}.
\end{remark}
Motivated by Theorem \ref{BrezisMerletypealternatives} and Theorem \ref{thm1}, we also could solve the open problem raised in \cite{T2}, and show  the existence of blow up solutions with the concentrating
property at the vortex point.

\begin{theorem}\label{5thm}
Assume $\mathfrak{M}>4$, $p_1\neq p_j$, $j=2,\cdots,n$, and $1\ll(\ln\lambda)^5\lambda^5\ll\mu$.
   Then  (\ref{eq1}) has a solution $(u_{\lambda,\mu}, N_{\lambda,\mu})$ satisfying

(i) $\lambda^2 e^{u_{\lambda,\mu}}\left(1-\frac{N_{\lambda,\mu}}{\lambda}\right)\to 4\pi \mathfrak{M}\delta_{p_1}$ in the sense of measure  as $\lambda, \mu  \to \infty, \frac{(\ln\lambda)^5\lambda^5}{\mu}\to0$,

(ii) $\max_{y\in \Omega}u_{\lambda,\mu}(y)\geq c$ for some constant $c\in\mathbb{R}$,

(iii) $\frac{N_{\lambda,\mu}}{\lambda}\to 0$  uniformly on any compact subset of $\Omega\setminus   \{p_1\}$  as $\lambda, \mu  \to \infty, \frac{(\ln\lambda)^5\lambda^5}{\mu}\to0$.
\end{theorem}
 \begin{remark}
If we consider the blow up solutions at the vortex point with the multiplicity one,  and assume that the maximum of the first component has a finite lower bound, then the limit equation becomes the  Chern-Simons equation containing the vortex point with the multiplicity one.  In this case, the local mass should be greater than $16\pi$, necessarily $4\pi\mathfrak{M}>16\pi$, and thus we need the condition $\mathfrak{M}>4$ in Theorem \ref{5thm}. 
 \end{remark}

We note that the    conditions for $\lambda,$ $\mu>0$ in Theorem \ref{thm1}-\ref{5thm} is stronger than  the condition $\lambda\ll\mu$ in Theorem   \ref{BrezisMerletypealternatives} because of technical reason, which occurs from the lower bound of $u_{\lambda,\mu}$.
The maximum of the first component for solutions in Theorem \ref{thm1} and Theorem \ref{5thm} has a  finite lower bound since the profile of approximate solutions comes from the entire solution of (CS) model.   In forthcoming paper,
we will study the  blow up solutions whose  first component has no  lower bound for the maximum value such that the limiting profile will be the Liouville equation.

The paper is organized as follows. In Section 2, we review some preliminaries in the gauge theory.
In Section 3, we analyze the asymptotic behavior of solutions and prove Theorem   \ref{BrezisMerletypealternatives}. In Section 4-5,  we study  the existence  of blow up solutions.

\section{Preliminaries}\label{sec2}
In this section, we review  some known results in the gauge theory. Firstly, we consider the following  problem
\begin{equation}
\begin{aligned} \label{limitingpro}
\Delta w + e^w(1-e^w)=4\pi m\delta_0\ \ \textrm{in}\ \mathbb{R}^2.
\end{aligned}
\end{equation}
We recall the following results.
\begin{lemma}\label{lemma2.1}
\cite{BM,CL} \cite[Lemma 3.2]{CK} Let $m$ be a nonnegative integer, and $w$ be a  solution of \eqref{limitingpro}.\\
If $e^w(1-e^w) \in L^1(\mathbb{R}^2)$, then either

(i) $w(x)\to0$ as $|x|\to\infty$, or

(ii) $w(x)=-\beta\ln |x| + O(1)$ near $\infty$, where
$\beta=-2m+{\frac{1}{2\pi}}\int_{\mathbb{R}^2} e^w(1-e^w)dx.$

\noindent
Assume that $w$ satisfies the boundary condition (ii). Then we have
\[\int_{\mathbb{R}^2}e^{2w} dx = \pi (\beta^2-4\beta-4m^2-8m), \  \textrm{and}\
\int_{\mathbb{R}^2}e^{w} dx = \pi (\beta^2-2\beta-4m^2-4m).\]
In particular,
$\int_{\mathbb{R}^2} e^w(1 -e^w)dx >8\pi(1 +m).$
\end{lemma}
Next we introduce the following result, which will help us to study the asymptotic behavior of solutions in $\Omega$.
\begin{lemma}\cite[Theorem 2.1]{CFL} \cite[Theorem 3.2]{CHMY}  \cite[Theorem 2.2]{SY}\label{propertyofentiresolution}
Let $m=0$, and $w$ be a  solution of \eqref{limitingpro} with $e^w(1-e^w)\in L^1(\mathbb{R}^2)$. Then, $w(x)$ is smooth,  radially symmetric with respect to some point $x_0$ in $\mathbb{R}^2$, and strictly decreasing function of $r=|x-x_0|$.

Assume $w(r;s)$ be the radially symmetric  solution with respect to $0$ of \eqref{limitingpro} such that \[\lim_{r\to0}w(r; s)=s,\ \  \textrm{and}\ \  \lim_{r\to0}w'(r; s)=0,\]
where $w'$ denotes $\frac{dw}{dr}(r; s)$, and let us set
\begin{equation}\label{defbetas}
\beta(s)\equiv\frac{1}{2\pi}\int_{\mathbb{R}^2} e^{w(r; s)}(1-e^{w(r; s)})dx=\int^\infty_0 e^{w(r; s)}(1-e^{w(r; s)})rdr.
\end{equation}Then one has

(i) $\beta(0)=0$ and $w(\cdot;0)\equiv0$;

(ii) $\beta: (-\infty, 0)\rightarrow(4,+\infty)$ is strictly increasing, bijective, and
$$\lim_{s\to-\infty}\beta(s)=4, \textrm{ and }\lim_{s\to 0_-}\beta(s)=+\infty.$$
\end{lemma}

\begin{lemma}(Lemma 2.1, \cite{RT})\label{lemma_RT} Let $(u_{\lambda,\mu},N_{\lambda,\mu})$ be solutions of \eqref{eq1} over $\Omega$. Then \[u_{\lambda,\mu}(x)<0, \quad 0<N_{\lambda,\mu}(x)<\lambda\ \ \ \ \textrm{for any}\ \  \ x\in\Omega.\]
\end{lemma}
In view of Lemma \ref{lemma_RT}, we can show that the nonlinear term of the first equation in \eqref{main_eq}  is uniformly bounded in $L^1(\Omega)$ with respect to $\lambda,\mu>0$ as  in the following corollary.
\begin{corollary}\label{cor1}Let $(u_{\lambda,\mu},N_{\lambda,\mu})$ satisfy  \eqref{eq1} over $\Omega$. Then we have
\[\int_{\Omega}\lambda^2 e^{u_{\lambda,\mu}}\left(1-\frac{N_{\lambda,\mu}}{\lambda}\right)dx=\int_{\Omega}\lambda^2 e^{u_{\lambda,\mu}}\left|1-\frac{N_{\lambda,\mu}}{\lambda}\right|dx=4\pi \mathfrak{M}.\]\end{corollary}
\begin{proof}By integrating \eqref{main_eq} over $\Omega$ and using Lemma \ref{lemma_RT}, we can obtain  Corollary \ref{cor1}.

\end{proof}

\begin{remark}\label{rem1}In view of  Corollary \ref{cor1}, we obtain   \begin{equation*}\lim_{\lambda\to\infty}\left[\int_{\Omega} e^{u_{\lambda,\mu}}\left|1-\frac{N_{\lambda,\mu}}{\lambda}\right|dx\right]=0,\end{equation*} which implies  \begin{equation}\label{limint}\textrm{either}\  \ u_{\lambda,\mu}\to-\infty \  \  \textrm{or}\  \  \frac{N_{\lambda,\mu}}{\lambda}\to1\  \  \textrm{a.e. in}\  \Omega\  \  \textrm{as}\  \  \lambda\to\infty.\end{equation}
Moreover, by integrating the second equation of \eqref{eq1} on $\Omega$, we also see that  \begin{equation}\label{limint1}\int_{\Omega} \left(1+\frac{\lambda}{\mu}e^{u_{\lambda,\mu}}\right)\frac{N_{\lambda,\mu}}{\lambda}dx=\int_{\Omega}\left(1+\frac{\lambda}{\mu}\right)e^{u_{\lambda,\mu}}dx.\end{equation}If $\mu>\lambda$, then it is reasonable to consider the class of solutions satisfying the asymptotic behavior in \eqref{def_sol}.
\end{remark}
Let us also recall the following form of the Harnack inequality.
\begin{lemma}(\cite{BT,GT})\label{harnarkineq}
Let $D\subseteq\mathbb{R}^2$ be a smooth bounded domain and $v$ satisfy:
$$-\Delta v=f\ \textrm{in}\ D,$$
with $f\in L^p(D)$, $p>1$. For any subdomain $D'\subset\subset D$, there exist two positive constants $\sigma\in(0,1)$ and $\tau>0$,
depending on $D'$ only such that:

$$(i)\ \textrm{if}\ \sup_{\partial D} v\le C,\ \textrm{then}\ \sup_{D'} v\le\sigma\inf_{D'}v+(1+\sigma)\tau\|f\|_{L^p}+(1-\sigma)C,$$

$$(ii)\ \textrm{if}\ \inf_{\partial D} v\ge -C,\ \textrm{then}\ \sigma\sup_{D'} v\le\inf_{D'}v+(1+\sigma)\tau\|f\|_{L^p}+(1-\sigma)C.$$
\end{lemma}

\section{Asymptotic behavior of solutions}\label{sec3}
In this section, we will study the asymptotic behavior of solutions to (\ref{main_eq}) and prove Theorem \ref{BrezisMerletypealternatives}. 
We firstly introduce some notations. 
Let  $G(x,y)$ be the Green's function  satisfying
\begin{equation*}
-\Delta_x G(x,y)=\delta_y-\frac{1}{|\Omega|},\quad\int_{\Omega}G(x,y)dy=0,
\end{equation*}
where $|\Omega|$ is the measure of $\Omega$, and we denote the regular part of $G(x,y)$ by \[\gamma(x,y)=G(x,y)+\frac{1}{2\pi}\ln|x-y|.\]
Let $\mathfrak{M}=\sum_{i=1}^n m_{i}$, and
\begin{equation}\label{u_0}
u_0(x)=-4\pi \sum_{i=1}^n m_{i}G(x,p_i).
\end{equation}
We set $u=v+u_0$, and assume $|\Omega|=1$. Then (\ref{main_eq}) is equivalent to
\begin{equation}\label{eq2}
\left\{\begin{array}{l}
\Delta \left( v+\frac{N}{\mu}\right)=-\lambda^2 e^{v+u_0}\left(1-\frac{N}{\lambda}\right) +4\pi \mathfrak{M},\\
\Delta N=\mu (\mu+\lambda e^{v+u_0})N-\lambda \mu(\lambda+\mu)e^{v+u_0}
\end{array}
\right. \mbox{ in }\Omega.
\end{equation}

\begin{lemma}\label{grad_sol} Let $(u_{\lambda,\mu},N_{\lambda,\mu})$ satisfy  \eqref{main_eq} over $\Omega$. Then there exists a constant $C>0$, independent of $\lambda>0$ and $\mu>0$, such that
\[\left\|\nabla\left(u_{\lambda,\mu}-u_0+\frac{N_{\lambda,\mu}}{\mu}\right)\right\|_{L^\infty(\Omega)}\le C\lambda.\]
\end{lemma}

\begin{proof}By applying the Green's representation formula for a solution $(u_{\lambda,\mu},N_{\lambda,\mu})$   of  \eqref{main_eq}, we see
\[\left(u_{\lambda,\mu}-u_0+\frac{N_{\lambda,\mu}}{\mu}\right)(x)-\int_\Omega \left(u_{\lambda,\mu}-u_0+\frac{N_{\lambda,\mu}}{\mu}\right)dy=\int_\Omega \lambda^2e^{u_{\lambda,\mu}(y)}\left(1-\frac{N_{\lambda,\mu}(y)}{\lambda}\right)G(x,y)dy.
\]
Together with  Lemma \ref{lemma_RT} and Corollary \ref{cor1}, we can obtain
\begin{equation}\begin{aligned}
&\left|\nabla_x\left(u_{\lambda,\mu}-u_0+\frac{N_{\lambda,\mu}}{\mu}\right)(x)\right|
\\&\le\left|\int_{B_d(x)}\lambda^2e^{u_{\lambda,\mu}(y)}\left(1-\frac{N_{\lambda,\mu}(y)}{\lambda}\right)\left(-\frac{x-y}{2\pi|x-y|^2}+\nabla \gamma(x,y)\right)dy\right|+c_0\\
&\leq \frac{ \lambda^2\left\|e^{u_{\lambda,\mu}}\left(1-\frac{N_{\lambda,\mu}(y)}{\lambda}\right)\right\|_{L^\infty(\Omega)}}{2\pi}\left[\int_{|x-y|\leq\frac{1}{\lambda}} \frac{1}{|x-y|}dy\right]\\
&\quad+\int_{\frac{1}{\lambda}\le|x-y|\leq d}\frac{\lambda^2e^{u_{\lambda,\mu}(y)}\left|1-\frac{N_{\lambda,\mu}(y)}{\lambda}\right|}{2\pi|x-y|}dy +c_1
 \le  C\lambda,
\end{aligned}\end{equation}
where $c_0,\ c_1,\ C>0$ are constants, independent of $\lambda, \mu>0$.

\end{proof}

Next we will have the key estimate which will reduce \eqref{main_eq} to an  almost decoupled system whose first equation is  a perturbation of a single Chern-Simons equation.

\begin{lemma}\label{lem3.2}Let $(u_{\lambda,\mu},N_{\lambda,\mu})$ satisfy  \eqref{main_eq} over $\Omega$. Then
\[\lim_{\lambda,\mu\to\infty,\ \frac{\lambda}{\mu}\to0}\left\|e^{u_{\lambda,\mu}}-\frac{N_{\lambda,\mu}}{\lambda}\right\|_{L^\infty(\Omega)}=0.\]
\end{lemma}

\begin{proof} Let $v_{\lambda,\mu}=u_{\lambda,\mu}-u_0$. Then $v_{\lambda,\mu}$ satisfies \eqref{eq2}.
We argue by   contradiction and suppose that there exists $x_{\lambda,\mu}\in\Omega$ such that
\begin{equation}\label{contra}
\left|e^{v_{\lambda,\mu}(x_{\lambda,\mu})+u_0(x_{\lambda,\mu})}-\frac{N_{\lambda,\mu}(x_{\lambda,\mu})}{\lambda}\right|\geq c>0.
\end{equation}
Let \begin{equation}\label{change_v}
y=\mu^{-1}x+x_{\lambda,\mu},\ {and}\ \tilde{N}_{\lambda,\mu}(x)=\frac{N_{\lambda,\mu}(\mu^{-1}x+x_{\lambda,\mu})}{\lambda}=\frac{N_{\lambda,\mu}(y)}{\lambda}.\end{equation}
Then we see that
\begin{equation}\begin{aligned}\label{eq_N}
 &\Delta_x \tilde{N}_{\lambda,\mu}(x)-(1+\frac{\lambda}{\mu}e^{(v_{\lambda,\mu}+u_0)(\mu^{-1}x+x_{\lambda,\mu})})\tilde{N}_{\lambda,\mu}(x)
 \\&=-\left(1+\frac{\lambda}{\mu}\right)e^{(v_{\lambda,\mu}+u_0+\frac{{N}_{\lambda,\mu}}{\mu})(\mu^{-1}x+x_{\lambda,\mu})}e^{-\frac{{N}_{\lambda,\mu}(\mu^{-1}x+x_{\lambda,\mu})}{\mu}}\\
&=-\left(1+\frac{\lambda}{\mu}\right)\left(1+O\left(\frac{\|{N}_{\lambda,\mu}\|_{L^\infty(\Omega)}}{\mu}\right)\right)e^{(v_{\lambda,\mu}+u_0+\frac{{N}_{\lambda,\mu}}{\mu})(\mu^{-1}x+x_{\lambda,\mu})}\\
&=-(1+o(1))e^{(v_{\lambda,\mu}+u_0+\frac{{N}_{\lambda,\mu}}{\mu})(\mu^{-1}x+x_{\lambda,\mu})}\ \ \textrm{as} \ \ \lambda,\mu\to\infty,\ \frac{\lambda}{\mu}\to0.
\end{aligned}\end{equation} Here, the last equality is obtained from Lemma \ref{lemma_RT}.

Fix a constant $R>0$, independent of $\lambda, \mu>0$.
The mean value theorem and  $\left\|\nabla\left(v_{\lambda,\mu}+\frac{N_{\lambda,\mu}}{\mu}\right)\right\|_{L^\infty(\Omega)}=O(\lambda)$ in Lemma \ref{grad_sol}  yield   some $\theta\in [0,1]$ satisfying
\begin{equation}\begin{aligned}\label{2.4}
&\left(v_{\lambda,\mu}+\frac{N_{\lambda,\mu}}{\mu}\right)(\mu^{-1}x+x_{\lambda,\mu})
\\&=\left(v_{\lambda,\mu}+\frac{N_{\lambda,\mu}}{\mu}\right)(x_{\lambda,\mu})+\nabla_z\left(v_{\lambda,\mu}(z)+\frac{N_{\lambda,\mu}(z)}{\mu}\right)\Big|_{z=\mu^{-1}\theta x+x_{\lambda,\mu}}\cdot (\mu^{-1}x)\\
&=\left(v_{\lambda,\mu}+\frac{N_{\lambda,\mu}}{\mu}\right)(x_{\lambda,\mu})+O\left(\frac{\lambda}{\mu}|x|\right)\\
&=\left(v_{\lambda,\mu}+\frac{N_{\lambda,\mu}}{\mu}\right)(x_{\lambda,\mu})+o(1) \quad \textrm{for} \  |x|\leq R \ \textrm{as}\ \ \lambda,\mu\to\infty,\ \frac{\lambda}{\mu}\to0.
\end{aligned}\end{equation}

We are going to consider the following  cases according to the location of limit point for $x_{\lambda,\mu}$, up to subsequence.

\medskip

\noindent\textbf{Case 1.}
$\lim_{\lambda,\mu\to\infty,\frac{\lambda}{\mu}\to0}x_{\lambda,\mu}=x_0\notin \cup^{n}_{i=1}\{p_i\}$.

\medskip
\noindent
Since $B_d(x_0)\cap\cup^{n}_{i=1}\{p_i\}=\emptyset$ for a sufficiently small constant $d>0$, $u_0$ is smooth in $B_d(x_0)$. Together with \eqref{2.4}, we see that
\[
(v_{\lambda,\mu}+u_0+\frac{{N}_{\lambda,\mu}}{\mu})(\mu^{-1}x+x_{\lambda,\mu}) = (v_{\lambda,\mu}+u_0+\frac{{N}_{\lambda,\mu}}{\mu})(x_{\lambda,\mu})+o(1) \quad \textrm{for }  |x|\leq R,
\]here we used that if $|x|\le R$, then $\mu^{-1}x+x_{\lambda,\mu}\in B_d(x_0)$. \\
In view of Lemma \ref{lemma_RT}, we have  $|\tilde{N}_{\lambda,\mu}|\le1$, and thus there exists a function $N_0$ satisfying
$\tilde{N}_{\lambda,\mu}\to N_0$ in $C^{1}_{\textrm{loc}}(\RN)$ as $\lambda,\mu\to\infty,\ \frac{\lambda}{\mu}\to0$, where $N_0$ is a solution of
\[
\Delta N_0-N_0=-c_0\ \ \textrm{in}\ \ \RN, \  \ \textrm{and}\ \ \|N_0\|_{L^\infty(\RN)}\le1,
\]
and $c_0=\lim_{\lambda,\mu\rightarrow\infty}e^{(v_{\lambda,\mu}+u_0+\frac{{N}_{\lambda,\mu}}{\mu})(x_{\lambda,\mu})}$. Then we have $N_0\equiv c_0$ in $\RN$ (for example, see \cite[Proposition 2.3]{FLL}).\\
We  also note that
\begin{equation}\label{c1}
\lim_{\lambda,\mu\to\infty,\ \frac{\lambda}{\mu}\to0} \frac{N_{\lambda,\mu}(x_{\lambda,\mu})}{\lambda}=\lim_{\lambda,\mu\to\infty,\ \frac{\lambda}{\mu}\to0} \tilde{N}_{\lambda,\mu}(0)= N_0(0)= c_0,  \end{equation} and 
\begin{equation}\label{c2}\begin{aligned}
 c_0 &=\lim_{\lambda,\mu\to\infty,\ \frac{\lambda}{\mu}\to0} e^{(v_{\lambda,\mu}+u_0+\frac{{N}_{\lambda,\mu}}{\mu})(x_{\lambda,\mu})}
=\lim_{\lambda,\mu\to\infty,\ \frac{\lambda}{\mu}\to0} e^{(v_{\lambda,\mu}+u_0)(x_{\lambda,\mu})}\left(1+O\left(\frac{\|N_{\lambda,\mu}\|_{L^\infty(\Omega)}}{\mu}\right)\right)
\\
&=\lim_{\lambda,\mu\to\infty,\ \frac{\lambda}{\mu}\to0} e^{(v_{\lambda,\mu}+u_0)(x_{\lambda,\mu})}\left(1+o(1)\right)= \lim_{\lambda,\mu\to\infty,\ \frac{\lambda}{\mu}\to0} e^{(v_{\lambda,\mu}+u_0)(x_{\lambda,\mu})},
\end{aligned}\end{equation}
here, we used Lemma \ref{lemma_RT} and the assumption $1\ll\lambda\ll\mu$ in the third equality.
However,  \eqref{c1} and \eqref{c2} contradict  the assumption  \eqref{contra}.

\medskip

\noindent\textbf{Case 2.}  $x_{\lambda,\mu}\rightarrow x_0=p_i$ for some $i$.

\medskip\noindent
Define $\hat{u}_0(x)\in C^\infty(B_d(p_i))$ such that
\begin{equation} \hat{u}_0(x)=u_0(x) - 2m_i\ln|x-p_i|.\label{uhat}\end{equation} Together with \eqref{2.4}, we have
\[
(v_{\lambda,\mu}+\hat{u}_0+\frac{{N}_{\lambda,\mu}}{\mu})(\mu^{-1}x+x_{\lambda,\mu}) = (v_{\lambda,\mu}+\hat{u}_0+\frac{{N}_{\lambda,\mu}}{\mu})(x_{\lambda,\mu})+o(1), \quad \textrm{if }  |x|\leq R.
\]

There are two cases according to  the behavior  of $|x_{\lambda,\mu}-p_i|\mu$.

\medskip
\noindent
\textbf{Case 2-(1).}    $\lim_{\lambda,\mu\to\infty,\frac{\lambda}{\mu}\to0}|x_{\lambda,\mu}-p_i|\mu=\infty$.

\medskip
\noindent
In this case, the equation \eqref{eq_N} and \eqref{uhat} imply that  if  $|x|\le R$, then
\begin{equation}\begin{aligned}\label{case2-1}
&\Delta_x\tilde{N}_{\lambda,\mu}(x) - (1+o(1))\tilde{N}_{\lambda,\mu}(x)\\&=-(1+o(1))|\mu^{-1}x+x_{\lambda,\mu}-p_i|^{2m_i}e^{(v_{\lambda,\mu}+\hat{u}_0+\frac{{N}_{\lambda,\mu}}{\mu}(\mu^{-1}x+x_{\lambda,\mu}))}
\\
&=-(1+o(1))\left|\frac{x}{\mu|x_{\lambda,\mu}-p_i|}+\frac{x_{\lambda,\mu}-p_i}{|x_{\lambda,\mu}-p_i|}\right|^{2m_i}|x_{\lambda,\mu}-p_i|^{2m_i}e^{(v_{\lambda,\mu}+\hat{u}_0+\frac{{N}_{\lambda,\mu}}{\mu})(x_{\lambda,\mu})}\\
&=-(1+o(1))e^{(v_{\lambda,\mu}+u_0+\frac{{N}_{\lambda,\mu}}{\mu})(x_{\lambda,\mu})}.
\end{aligned}\end{equation}
Then the same arguments in  {Case 1} implies  a contradiction again.

\medskip
\noindent
\textbf{Case 2-(2).} $\lim_{\lambda,\mu\to\infty,\frac{\lambda}{\mu}\to0}|x_{\lambda,\mu}-p_i|\mu\leq c$ for some constant $c>0$.

\medskip
\noindent
In view of \eqref{eq2}, Lemma \ref{lemma_RT}, and the condition $\lambda,\mu\to\infty,\ \frac{\lambda}{\mu}\to0$,  we  see that
\begin{equation}\begin{aligned}
 \Delta \left(v_{\lambda,\mu}+\frac{N_{\lambda,\mu}}{\mu}\right) &
=-\lambda^2e^{(v_{\lambda,\mu}+u_0+\frac{{N}_{\lambda,\mu}}{\mu})}e^{-\frac{{N}_{\lambda,\mu}}{\mu}}\left(1-\frac{N_{\lambda,\mu}}{\lambda}\right)+ 4\pi \mathfrak{M}
\\&
=-\lambda^2e^{(v_{\lambda,\mu}+u_0+\frac{{N}_{\lambda,\mu}}{\mu})}\left(1+\frac{O(\|N_{\lambda,\mu}\|_{L^\infty(\Omega)}}{\mu}\right)\left(1-\frac{N_{\lambda,\mu}}{\lambda}\right)+ 4\pi \mathfrak{M}\\
&= - (1+o(1))\lambda^2|x-p_i|^{2m_i}e^{(v_{\lambda,\mu}+\hat{u}_0+\frac{{N}_{\lambda,\mu}}{\mu})}\left(1-\frac{N_{\lambda,\mu}}{\lambda}\right)+ 4\pi \mathfrak{M}.
\end{aligned}\end{equation}
Let $$\hat{v}_{\lambda,\mu}(x)=\left(v_{\lambda,\mu}+\frac{N_{\lambda,\mu}}{\mu}\right)(\mu^{-1}x+p_i)-2m_i \ln\mu$$ and $z=\mu^{-1}x+p_i$.\\ Then, we have 
\begin{equation}\label{eqvj}\begin{aligned}
&\Delta_x \hat{v}_{\lambda,\mu}(x) + \lambda^2\mu^{-2}|x|^{2m_i}e^{\hat{v}_{\lambda,\mu}(x)+\hat{u}_0(\mu^{-1}x+p_i)}(1+o(1))\left(1-\frac{N_{\lambda,\mu}(\mu^{-1}x+p_i)}{\lambda}\right)\\&=4\pi\mathfrak{M}\mu^{-2}=o(1) \ \ \textrm{as} \ \ \lambda,\mu\to\infty,\ \frac{\lambda}{\mu}\to0.\end{aligned}
\end{equation}
In view of Lemma \ref{lemma_RT} and $\hat{u}_0\in C^\infty(B_d(p_i))$, we  see that
\begin{equation}\begin{aligned}\label{uppb1}
\hat{v}_{\lambda,\mu}(x)+2m_i\ln|x|
&=(v_{\lambda,\mu}+u_0+\frac{{N}_{\lambda,\mu}}{\mu})(\mu^{-1}x+p_i)-\hat{u}_0(\mu^{-1}x+p_i)
\\&\leq \left\|\frac{{N}_{\lambda,\mu}}{\mu}\right\|_{L^\infty(\Omega)}+\|\hat{u}_0\|_{L^\infty(B_d(p_i))}\le c_0\ \ \ \textrm{in}\ \ B_{d\mu}(0),
\end{aligned}\end{equation}for some constant $c_0>0$, independent of $\lambda, \mu>0$.\\
By Lemma \ref{grad_sol}, we also see that
\begin{equation}\begin{aligned}\label{gradb1}
|\nabla_x \hat{v}_{\lambda,\mu}(x)|
 &=\mu^{-1}\left|\nabla_z\left(v_{\lambda,\mu}+\frac{N_{\lambda,\mu}}{\mu}\right)(z)\Big|_{z=\mu^{-1}x+p_i}\right|\\&=O(\frac{\lambda}{\mu})=o(1)\ \ \ \textrm{in}\ \ B_{d\mu}(0)\ \ \  \textrm{as}\ \ \ \lambda,\mu\to\infty,\ \frac{\lambda}{\mu}\to0.
\end{aligned}\end{equation}
From \eqref{uppb1} and \eqref{gradb1}, we note that there are two possibilities as follows:

\medskip
\noindent
{\bf(i).}  $\sup_{\partial B_1(0)}|\hat{v}_{\lambda,\mu}|\le C$. \\
In this case, \eqref{gradb1} implies that $|\hat{v}_{\lambda,\mu}|$ is uniformly bounded in $C^0_{\textrm{loc}}\left(B_{d\mu}(0)\right)$ for $\lambda,\mu>0$.
Then there exists a function $v_0$ such  that $\hat{v}_{\lambda,\mu}\to v_0$ in $C^{1}_{\textrm{loc}}\left(B_{d\mu}(0)\right)$ and $\nabla v_0\equiv0$ in $\RN$. It implies that $v_0\equiv c$ for some constant $c\in\R$ and  $\Delta v_0=0$ in $\RN$. From the mean value theorem for harmonic function and \eqref{uppb1}, we see that  for any constant $R>0$,
\begin{equation}\label{contras}
\begin{split}
-\infty<c&=v_0(0)= \frac{1}{|\partial B_R(0)|}\int_{\partial B_R(0)}v_0(y) dS_y\\
&\le \frac{1}{|\partial B_R(0)|}\int_{\partial B_R(0)}(c_0-2m_i\ln|y|)dS_y=c_0-2m_i\ln R.
\end{split}
\end{equation}
We get a contradiction as $R\to\infty$ in \eqref{contras}.

\medskip
\noindent
{\bf (ii).}  $\sup_{\partial B_1(0)}\hat{v}_{\lambda,\mu}\rightarrow-\infty$.\\
In this case, \eqref{gradb1} implies that $\hat{v}_{\lambda,\mu}\to-\infty$ is uniformly in $C^0_{\textrm{loc}}\left(B_{d\mu}(0)\right)$ for $\lambda,\mu>0$.
By \eqref{eq_N}, we have
\begin{equation}\begin{aligned}
&\Delta_x\tilde{N}_{\lambda,\mu}(x)- \left(1+\frac{\lambda}{\mu}e^{(v_{\lambda,\mu}+u_0)(\mu^{-1}x+x_{\lambda,\mu})}\right)\tilde{N}_{\lambda,\mu}(x)\\
&=- \left(1+\frac{\lambda}{\mu}\right)e^{(v_{\lambda,\mu}+u_0)(\mu^{-1}x+x_{\lambda,\mu})}
\\
&=- (1+o(1))\left|x+\mu(x_{\lambda,\mu}-p_i)\right|^{2m_i}e^{\hat{v}_{\lambda,\mu}(x+\mu(x_{\lambda,\mu}-p_i))+\hat{u}_0(\mu^{-1}x+x_{\lambda,\mu})}=o(1).
\end{aligned}\end{equation}
Since $|\tilde{N}_{\lambda,\mu}(x)|\le 1$ for all $x\in B_{{d\mu}}(0)$, there is a function $\tilde{N}_0$ satisfying $\tilde{N}_{\lambda,\mu}\to\tilde{N}_0$ in $C^1_{\textrm{loc}}(\RN)$, and  \[\Delta \tilde{N}_0- \tilde{N}_0=0\ \ \ \ \textrm{in}\ \ \ \RN,\ \ \|\tilde{N}_0\|_{L^\infty(\RN)}\le1,\] which implies $\tilde{N}_0\equiv0 $ in $\RN$ (for example, see \cite[Proposition 2.3]{FLL}).
We note that \begin{equation}\label{N}\lim_{\lambda,\mu\to\infty,\ \frac{\lambda}{\mu}\to0}\left(\frac{N_{\lambda,\mu}(x_{\lambda,\mu})}{\lambda}\right)=\lim_{\lambda,\mu\to\infty,\ \frac{\lambda}{\mu}\to0}\tilde{N}_{\lambda,\mu}(0)=\tilde{N}_0(0)=0,\end{equation} and
\begin{equation}\label{Ne}
\lim_{\lambda,\mu\to\infty,\ \frac{\lambda}{\mu}\to0}e^{(v_{\lambda,\mu}+u_0)(x_{\lambda,\mu})}=\lim_{\lambda,\mu\to\infty,\ \frac{\lambda}{\mu}\to0}\left((1+o(1))\left|\mu(x_{\lambda,\mu}-p_i)\right|^{2m_i}e^{\hat{v}_{\lambda,\mu}(\mu(x_{\lambda,\mu}-p_i))+\hat{u}_0(x_{\lambda,\mu})}\right)=0.
\end{equation}
However, \eqref{N} and \eqref{Ne} contradict the assumption \eqref{contra}.

\end{proof}


In view of Lemma \ref{lem3.2}, the first equation of \eqref{main_eq} can be regarded  as a perturbation of a single Chern-Simons equation \eqref{cs_eq}.  By applying the arguments in  \cite[Lemma 4.1]{CK}, we can obtain the following result. 
\begin{lemma}\label{uniformestimates}
 Suppose that
there exists a sequence of solutions $(u_{\lambda,\mu}, N_{\lambda,\mu})$ of \eqref{main_eq} such that
$$\lim_{\lambda,\mu\to\infty,\ \frac{\lambda}{\mu}\to0}\Big(\inf_\Omega|u_{\lambda,\mu}|\Big)=0.$$
Then, we have \begin{equation}\label{K}\lim_{\mlt}\|u_{\lambda,\mu}\|_{L^\infty(K)}=0\ \ \ \textrm{for any compact set }\ \ K\subset\Omega\setminus Z.\end{equation}
\end{lemma}
\begin{proof}
Choose a sequence of points $\{x_{\lambda,\mu}\}\subseteq \Omega$ such that
\begin{equation}|u_{\lambda,\mu}(x_{\lambda,\mu})|=\inf_\Omega|u_{\lambda,\mu}|, \ \ \textrm{and} \ \lim_{\mlt}u_{\lambda,\mu}(x_{\lambda,\mu})=0.\label{to0}\end{equation}
Passing to a subsequence (still denoted by $u_{\lambda,\mu}$),
we may assume that $\lim_{\lambda,\mu\to\infty,\ \frac{\lambda}{\mu}\to0}x_{\lambda,\mu}=x_0\in \Omega$. We consider the following two cases according to the location of $x_0$.

\medskip
\noindent
\textbf{Case 1.} $x_0\notin Z$.\\
Let $d>0$ be a small constant satisfying $B_d(x_0)\cap Z=\emptyset$.
We argue by contradiction and suppose that there exist a compact set $K\subset\Omega\setminus Z$, a positive constant $c_K>0$, and a sequence $\{z_{\lambda,\mu}\}\subset K$  such that
$\sup_K|u_{\lambda,\mu}|=|u_{\lambda,\mu}(z_{\lambda,\mu})|\ge c_K>0$ for  large $\lm0$. We choose a connected compact set $K_1\subset\Omega\setminus Z$ satisfying $B_d(x_0)\cup K \subset K_1$.
Since $u_{\lambda,\mu}(z_{\lambda,\mu})\le-c_K<0$,  Lemma \ref{propertyofentiresolution} implies there is a constant  $s_1<0$ such that
$$\beta(s_1)>4\mathfrak{M}\ \textrm{and}\ -c_K<s_1<0.$$
We can also choose $y_{\lambda,\mu}\in K_1$ such that $u_{\lambda,\mu}(y_{\lambda,\mu})=s_1$ by the intermediate value theorem.\\
Let $\bar{u}_{\lambda,\mu}(x)=\left(u_{\lambda,\mu}+\frac{\nl}{\mu}\right)(\lambda^{-1} x+y_{\lambda,\mu})$ for $x\in\Omega_{y_{\lambda,\mu}}\equiv\{\ x\in\mathbb{R}^2\ |\ \lambda^{-1} x+y_{\lambda,\mu}\in K_1\ \}$.\\
By Corollary \ref{cor1} and $u_{\lambda,\mu}(y_{\lambda,\mu})=s_1$, we see that  $\bar{u}_{\lambda,\mu}$ satisfies
\begin{equation}
\begin{aligned}
\left\{
 \begin{array}{ll}
 \Delta \bar{u}_{\lambda,\mu}+ e^{\bar{u}_{\lambda,\mu}(x)-\frac{\nl}{\mu}(\lambda^{-1} x+y_{\lambda,\mu})}\left(1-\frac{N_{\lambda,\mu}(\lambda^{-1} x+y_{\lambda,\mu})}{\lambda}\right) =0\ \ \  \textrm{in}\ \ \Omega_{y_{\lambda,\mu}},
 \\ \bar{u}_{\lambda,\mu}(0)=s_1+\frac{\nl(y_{\lambda,\mu})}{\mu},
 \\ \int_{\Omega_{y_{\lambda,\mu}}} \left|e^{\bar{u}_{\lambda,\mu}(x)-\frac{\nl}{\mu}(\lambda^{-1} x+y_{\lambda,\mu})}\left(1-\frac{N_{\lambda,\mu}(\lambda^{-1} x+y_{\lambda,\mu})}{\lambda}\right)\right| dx\le {4\pi\mathfrak{M}}.
 \end{array}\right.
\end{aligned}
\end{equation}
By using Lemma \ref{grad_sol} and $W^{2,p}$ estimation, we see that $\bar{u}_{\lambda,\mu}$ is bounded in $C^{1,\sigma}_{\textrm{loc}}(\Omega_{y_{\lambda,\mu}})$ for some $\alpha\in(0,1)$.
In view of Lemma \ref{lemma_RT} and Lemma \ref{lem3.2}, we see that  if $0<\lambda\ll\mu$, then
\begin{equation}\begin{aligned}\label{ee}&e^{\bar{u}_{\lambda,\mu}(x)-\frac{\nl}{\mu}(\lambda^{-1} x+y_{\lambda,\mu})}\left(1-\frac{N_{\lambda,\mu}(\lambda^{-1} x+y_{\lambda,\mu})}{\lambda}\right)\\&= e^{\bar{u}_{\lambda,\mu}(x)}\left(1+O\left(\frac{\|N_{\lambda,\mu}\|_{L^\infty(\Omega)}}{\mu}\right)\right)\left(1-e^{\bar{u}_{\lambda,\mu}(x)}+o(1)+O\left(\frac{\|N_{\lambda,\mu}\|_{L^\infty(\Omega)}}{\mu}\right)\right)
\\&=e^{\bar{u}_{\lambda,\mu}(x)}(1+o(1))\left(1-e^{\bar{u}_{\lambda,\mu}(x)}+o(1)\right),\end{aligned}\end{equation} and  $\bar{u}_{\lambda,\mu}(0)=s_1+o(1).$ Passing to a subsequence,  $\bar{u}_{\lambda,\mu}$ converges in
$C^1_{\textrm{loc}}(\mathbb{R}^2)$ to a function $u_*$, which is a solution of
\begin{equation}
\begin{aligned}
\left\{
 \begin{array}{ll}
 \Delta u_*+ e^{u_*}(1-e^{u_*}) =0\ \textrm{in}\ \mathbb{R}^2,
 \\ u_*(0)=s_1,
 \\ \int_{\mathbb{R}^2}|e^{u_*}(1-e^{u_*})|dx\le {4\pi\mathfrak{M}}.
 \end{array}\right.
\end{aligned}
\end{equation}
By using   Lemma \ref{propertyofentiresolution}, we see that $u_*$ is
radially symmetric with respect to a point $\bar{p}$ in $\mathbb{R}^2$.\\
In view of Lemma \ref{propertyofentiresolution},  we have 
\begin{equation}
\begin{aligned}
{4\pi\mathfrak{M}}
\ge \Big|\int_{\mathbb{R}^2} e^{u_*}(1-e^{u_*}) dx\Big|
=2\pi|\beta(u_*(\bar{p}))|\ge2\pi|\beta(s_1)|
>{8\pi\mathfrak{M}},
\end{aligned}
\end{equation}
which implies a contradiction. Thus    \eqref{K} holds true in Case 1.

\medskip
\noindent
\textbf{Case 2.} $x_0=p_i\in Z$ for some $i$.
\\ Fix a small constant $r_0>0$ such that $B_{r_0}(x_0)\cap Z=\{x_0\}$.
For simplicity, we assume that $x_0=0$.  We are going to show that
\begin{equation}\lim_{\mlt}\left(\inf_{|x|=r_0}|\ulm(x)|\right)=0.\label{cmlt}\end{equation}
Once we have \eqref{cmlt}, the argument in Case 1 implies \eqref{K}. In order to prove \eqref{cmlt}, we argue by contradiction again and suppose that, up to a subsequence, $\lim_{\mlt}\left(\inf_{|x|=r_0}|u_{\lambda,\mu}(x)|\right)\ge\tau_0$ for some constant $\tau_0>0$. Since $u_{\lambda,\mu}<0$, we have
\begin{equation}\label{gam}\lim_{\mlt}\left(\sup_{|x|=r_0}\ulm(x)\right)<-\tau_0.\end{equation} We divide our discussion into the following two cases.

\medskip

\noindent
(i). $\lim_{\mlt}\left( \lambda|x_{\lambda,\mu}|\right)<+\infty$.
\\ Note that $\ulm(x)= 2m_i\ln|x|+v_{\lambda,\mu}(x)$ near $x=0$ for some $1\le i\le n$, where $v_{\lambda,\mu}$ is a smooth function in $B_d(0)$. Let
\[
\hat{v}_{\lambda,\mu}(x)= v_{\lambda,\mu}(|x_{\lambda,\mu}|x)+2m_i\ln|x_{\lambda,\mu}|+\frac{\nl(|x_{\lambda,\mu}|x)}{\mu}\ \ \textrm{for} \ \ |x|\le\frac{r_0}{|x_{\lambda,\mu}|}.
\]
Then $\hat{v}_{\lambda,\mu}$ satisfies
\begin{equation}\label{vmain_eq}
\left\{\begin{array}{l}
\Delta \hat{v}_{\lambda,\mu}=-\lambda^2|x_{\lambda,\mu}|^2|x|^{2m_i} e^{\hat{v}_{\lambda,\mu}(x)-\frac{\nl(|x_{\lambda,\mu}|x)}{\mu}}\left(1-\frac{\nl(|x_{\lambda,\mu}|x)}{\lambda}\right)\ \ \textrm{in}\ \ B_{r_1}(0),\\
\int_{B_{\frac{r_0}{|x_{\lambda,\mu}|}}(0)}\lambda^2|x_{\lambda,\mu}|^2|x|^{2m_i} e^{\hat{v}_{\lambda,\mu}(x)-\frac{\nl(|x_{\lambda,\mu}|x)}{\mu}}\left(1-\frac{\nl(|x_{\lambda,\mu}|x)}{\lambda}\right)dx\le 4\pi\mathfrak{M}
\end{array}
\right. \mbox{ in }B_{\frac{r_0}{|x_{\lambda,\mu}|}}(0).
\end{equation}
By \eqref{to0}, we note that
\begin{equation}\label{limv}\lim_{\mlt}\hat{v}_{\lambda,\mu}
\left(\frac{x_{\lambda,\mu}}{|x_{\lambda,\mu}|}\right)
=\lim_{\mlt}\left(\ulm(x_{\lambda,\mu})
+\frac{\nl(x_{\lambda,\mu})}{\mu}\right)=0.\end{equation}  Together with Lemma \ref{grad_sol} and the assumption $\liminf_{\mlt}\left( \lambda|x_{\lambda,\mu}|\right)<+\infty$, we see that $\hat{v}_{\lambda,\mu}$ is bounded in $C^0_{\textrm{loc}}\left(B_{r_1}(0)\right)$.  Passing to a subsequence, we may assume that \[\lim_{\mlt}\frac{x_{\lambda,\mu}}{|x_{\lambda,\mu}|}=y_0\in \mathbb{S}^1,\ \lim_{\mlt}\left(\lambda|x_{\lambda,\mu}|\right)=c_0\ge0, \ \textrm{and}\ \hat{v}_{\lambda,\mu}\to\hat{v}\ \textrm{ in}\ C_{\textrm{loc}}^1\left(B_{r_1}(0)\right),\] for some function $\hat{v}$. By using  Lemma \ref{lemma_RT} and Lemma \ref{lem3.2} as in \eqref{ee}, we see that $\hat{u}(x)=\hat{v}(x)+2m_i\ln|x|$ satisfies
\[
\Delta\hat{u}+c_0^2e^{\hat{u}}(1-e^{\hat{u}})=4\pi m_i\delta_0\ \ \ \textrm{in}\ \ \RN.
\]
In view of Lemma \ref{lemma_RT}, we have  $ 0\le \frac{\nl(|x_{\lambda,\mu}|x)}{\lambda}\le 1 $. Moreover, Lemma \ref{lem3.2} implies    $\lim_{\mlt}\frac{\nl(|x_{\lambda,\mu}|x)}{\lambda}=e^{\hat{u}(x)}\in [0,1]$. Since $\hat{u}\le 0$ in $\RN$ and $\hat{u}(y_0)= 0$, we have $\hat{u}\equiv0$ by Hopf Lemma, which implies a contradiction.

\medskip
\noindent
(ii).
$\lim_{\mlt}\left( \lambda|x_{\lambda,\mu}|\right)= +\infty$.\\
Lemma \ref{propertyofentiresolution} implies there is a constant  $s_2<0$ such that
\[
\beta(s_2)>4\mathfrak{M}\ \textrm{and}\ -\tau_0<s_2<0,
\]
where $\tau_0$ is the constant in \eqref{gam}.
We can also choose $\hat{y}_{\lambda,\mu}$ on the line segment  joining $x_{\lambda,\mu}$ and $\frac{r_0 x_{\lambda,\mu}}{|x_{\lambda,\mu}|}$ such that $u_{\lambda,\mu}(\hat{y}_{\lambda,\mu})=s_2$ and $|\hat{y}_{\lambda,\mu}|\ge |x_{\lambda,\mu}|$ by the intermediate value theorem.
Let $\hat{u}_{\lambda,\mu}(x)=\left(u_{\lambda,\mu}+\frac{\nl}{\mu}\right)(\lambda^{-1} x+\hat{y}_{\lambda,\mu})$ for $x\in\hat{\Omega}_{\hat{y}_{\lambda,\mu}}\equiv\{\ x\in\mathbb{R}^2\ |\ \lambda^{-1} x+\hat{y}_{\lambda,\mu}\in B_{\frac{|x_{\lambda,\mu}|}{2}}(\hat{y}_{\lambda,\mu}) \ \}$.  Here we note that $0\notin B_{\frac{|x_{\lambda,\mu}|}{2}}(\hat{y}_{\lambda,\mu}) $.
Then $\hat{u}_{\lambda,\mu}$ satisfies
\begin{equation}
\begin{aligned}
\left\{
 \begin{array}{ll}
 \Delta \hat{u}_{\lambda,\mu}+ e^{\hat{u}_{\lambda,\mu}-\frac{\nl}{\mu}(\lambda^{-1} x+\hat{y}_{\lambda,\mu})}\left(1-\frac{N_{\lambda,\mu}(\lambda^{-1} x+\hat{y}_{\lambda,\mu})}{\lambda}\right) =0\ \ \  \textrm{in}\ \ \hat{\Omega}_{\hat{y}_{\lambda,\mu}},
 \\ \hat{u}_{\lambda,\mu}(0)=s_2+\frac{\nl(\hat{y}_{\lambda,\mu})}{\mu},
 \\ \int_{\hat{\Omega}_{\hat{y}_{\lambda,\mu}}} |e^{\hat{u}_{\lambda,\mu}}(1-e^{\hat{u}_{\lambda,\mu}})| dx\le {4\pi\mathfrak{M}}.
 \end{array}\right.\label{nonl}
\end{aligned}
\end{equation} Using the same argument as in Case 1, we get a  contradiction by comparing the upper/lower bound of $L^1$ norm of the nonlinear term in \eqref{nonl}.  Thus the claim \eqref{cmlt} holds true.  Then we can again apply the arguments in  Case 1  and prove \eqref{K} holds true.

\end{proof}

As a corollary of Lemma \ref{uniformestimates}, we get the following proposition.
\begin{proposition}\label{alternatives}
Let $(u_{\lambda,\mu}, \nl)$ be a sequence of solutions of \eqref{main_eq}. Then, up to subsequences, one of the following holds true:

(i) $u_{\lambda,\mu}\to0$ uniformly on any compact subset of $\Omega\setminus Z$ as $\mlt$, or

(ii) there exists a constant $\nu_0>0$ such that $\sup_{\lambda,\mu\to\infty,\ \frac{\lambda}{\mu}\to0}\Big(\sup_\Omega u_{\lambda,\mu}\Big)\le-\nu_0$.
\end{proposition}

\medskip
\noindent
\textbf{Completion of the proof of Theorem \ref{BrezisMerletypealternatives}.}
Note that Proposition \ref{alternatives}-(i) corresponds to Theorem \ref{BrezisMerletypealternatives}-(i).
In order to complete the proof of Theorem \ref{BrezisMerletypealternatives}, from now on, we will study the asymptotic behavior for the solution $(\ulm,\nl)$ of \eqref{main_eq} satisfying Proposition \ref{alternatives}-(ii).
Let us denote
\[
w_{\lambda,\mu}= u_{\lambda,\mu}+\frac{\nl}{\mu} +2\ln\lambda \quad\mbox{in }~ \Omega.
\]
Clearly $w_{\lambda,\mu}$ satisfies
\begin{equation}
  \label{wmaineq}\begin{aligned}
\left\{
 \begin{array}{ll}
\Delta w_{\lambda,\mu} +  e^{w_{\lambda,\mu}-\frac{\nl}{\mu}}\left(1-\frac{\nl}{\lambda}\right)
 = 4\pi\sum_{i=1}^n m_{i}\delta_{p_i} \quad \mbox{in }~ \Omega,\\
 \\ \sup_{\lambda,\mu\to\infty,\ \frac{\lambda}{\mu}\to0}\sup_\Omega\left(w_{\lambda,\mu}-2\ln\lambda-\frac{\nl}{\mu}\right)\le-\nu_0<0, \\ \int_{\Omega} |e^{w_{\lambda,\mu}-\frac{\nl}{\mu}}\left(1-\frac{\nl}{\lambda}\right)| dx=\int_{\Omega}e^{w_{\lambda,\mu}-\frac{\nl}{\mu}}\left(1-\frac{\nl}{\lambda}\right)dx= {4\pi\mathfrak{M}}.
 \end{array}\right.
\end{aligned}
\end{equation}
By Lemma \ref{lemma_RT} and Lemma \ref{lem3.2}, as $\mlt$, we have  (for example, see \eqref{ee})
\begin{equation}\label{weqp}
\Delta w_{\lambda,\mu} +e^{\wlm}(1+o(1)) \left(1-\lambda^{-2}e^{\wlm}+o(1)\right)
 = 4\pi\sum_{i=1}^n m_{i}\delta_{p_i} \quad \mbox{in }~ \Omega,
\end{equation}
and
\begin{equation*}
\int_{\Omega}e^{\wlm}(1+o(1))\left|\left(1-\lambda^{-2}e^{\wlm}+o(1)\right)\right|dx=4\pi\mathfrak{M}.
\end{equation*}
Then we have
\begin{equation}\label{w1l}
0<2\pi\mathfrak{M}\le \|e^{\wlm}\|_{L^1(\Omega)}\le\frac{8\pi\mathfrak{M}}{1-e^{-{\nu_0}}}.
\end{equation}
We consider the following two cases:

\medskip\noindent
\textbf{Case 1.} $\sup_{\Omega}\wlm\le C$ for some constant $C>0$.

 In this case, we note that the Harnack inequality (i.e. Lemma \ref{harnarkineq}) and  \eqref{w1l}  imply
$\wlm-u_0$ is uniformly bounded in $L^\infty(\Omega)$. Moreover,  in view of $W^{2,2}$ estimation, $\wlm-u_0$ is uniformly bounded in $C^{1,\alpha}(\Omega)$ for some $\alpha\in(0,1)$. Then  $\wlm-u_0\to\hat{w}$ in $C^{1}_{\textrm{loc}}(\Omega)$, where $\hat{w}$ satisfies
\[\Delta \hat{w} +e^{\hat{w}+u_0}=4\pi\mathfrak{M}.\]
We note that Case 1 implies Theorem \ref{BrezisMerletypealternatives}-(ii).

\medskip\noindent
\textbf{Case 2.} $\lim_{\mlt}\sup_{\Omega}\wlm=+\infty.$

Following \cite{BM}, we say that a point $q\in \Omega$ is a blow-up point for $\{w_{\lambda,\mu}\}$
if there exists a sequence $\{x_{\lambda,\mu,q}\}$ such that
\[ x_{\lambda,\mu,q} \to q \quad\mbox{and}\quad w_{\lambda,\mu} (x_{\lambda,\mu,q}) \to \infty
 \quad\mbox{as }~ \lambda,\mu\to\infty,\ \frac{\lambda}{\mu}\to0. \]
The set $S\subset\Omega $, which consists   of blow-up points for $\{w_{\lambda,\mu}\}$, is called the blow-up set for $\{w_{\lambda,\mu}\}$.

\medskip\noindent
\textbf{Step 1.} Let $p\in\Omega$ be a blow-up point for $\{w_{\lambda,\mu}\}$.
Then we have  the following ``minimal mass" result.
\begin{equation}\label{massofsingularpoint} \liminf_{\lambda,\mu\to\infty,\ \frac{\lambda}{\mu}\to0}\int_{B_d(p)}{e}^{w_{\lambda,\mu}-\frac{\nl}{\mu}}\left(1-\frac{\nl}{\lambda}\right) dx \ge 8\pi\ \ \textrm{for any}\ \ d>0.\end{equation}
Indeed, we note that  the equation \eqref{weqp} is a perturbation of  \begin{equation}\label{eqw0}
\Delta w_{\lambda} +e^{w_{\lambda}} \left(1-\lambda^{-2}e^{\wlm} \right)
 = 4\pi\sum_{i=1}^n m_{i}\delta_{p_i} \quad \mbox{in }~ \Omega,
\end{equation}
and the minimal mass result for \eqref{eqw0} was obtained in \cite[Lemma 4.2]{CK}.  By the similar arguments in \cite{CK}, we can also get \eqref{massofsingularpoint} for the solution $\wlm$ of \eqref{weqp}. Here we skip the detail and refer to \cite{CK}.

The estimation \eqref{massofsingularpoint} shows that if Case 2 happens, then
$\{w_{\lambda,\mu}\}$ has a nonempty finite blow-up set $B\subset \Omega $,
and $|B| \le \frac{{4\pi\mathfrak{M}}}{8\pi}=\frac{\mathfrak{M}}{2}$. We also see that
for any compact set $K\subseteq\Omega\setminus B$,
there exists a constant $C_{K}>0$ such that
\begin{equation}
  \label{locbdd}
 \sup_{K}w_{\lambda,\mu}\le C_{K}.
\end{equation}

\medskip\noindent
\textbf{Step 2.}
In this step, we are going to prove that the blow up phenomena implies the concentration of mass as in \cite{BT,  BM,  CK}. However, our case is a coupled system problem, we should carry out a delicate analysis in order to prove the concentration of mass.
Firstly, we claim that
\begin{equation}\label{lem:outside} w_{\lambda,\mu}-u_0\to-\infty\ \ \ \textrm{ uniformly on any compact set} \ K\subset\Omega\setminus B.\end{equation}
Choose a small constant $d>0$ satisfying for any $q\in B$, $B_{2d}(q)\cap\left[B\cup Z\right]=\{q\}$.
For each $q\in B$, we let $\{x_{\lambda,\mu,q}\}$ be a sequence of points such that
\[x_{\lambda,\mu,q}\to q\in B\ \ \textrm{ and}\ \ w_{\lambda,\mu}(x_{\lambda,\mu,q})=\sup_{B_d(q)} w_{\lambda,\mu}\to\infty\ \ \textrm{ as}\ \ \lambda,\mu\to\infty,\ \frac{\lambda}{\mu}\to0.\]
We shall prove that
\begin{equation}
  \label{eq:-infty}
 \lim_{\lambda,\mu\to\infty,\ \frac{\lambda}{\mu}\to0}\Big(\inf_{\partial B_r(q)}(w_{\lambda,\mu}-u_0)\Big)=-\infty
\end{equation}
for any $r\in(0,d]$ and all $q\in B$.
Then by (\ref{locbdd}) and Harnack's inequality, we get that
\[ \lim_{\lambda,\mu\to\infty,\ \frac{\lambda}{\mu}\to0} \big( \sup_{\Omega\setminus (\cup_{q_j\in B}B_r(q_j))} (w_{\lambda,\mu}-u_0) \big)
 =-\infty \quad \mbox{for any }~ r\in(0,d]. \]
To prove \eqref{eq:-infty}, we argue by contradiction and
suppose that there exist $r\in(0,d]$ and $q\in B$ such that
\begin{equation*}
 \begin{aligned}
\lim_{\lambda,\mu\to\infty,\ \frac{\lambda}{\mu}\to0} \Big(\inf_{\partial B_r(q)}(w_{\lambda,\mu}-u_0)\Big)\ge c,
 \end{aligned}
\end{equation*}
for some constant $c\in\mathbb{R}$. For simplicity, we assume that $q=0$.
By using (\ref{locbdd}) and Harnack's inequality, we can verify that $\{w_{\lambda,\mu} -u_0\}$ is bounded
in $C_{loc}^0 (B_{2d}(0) \setminus \{0\})$. Then elliptic estimates imply that
there exists a function $\xi\in C^{1,\sigma}_{loc}(B_{2d}(0)\setminus \{0\})$ such that
along a subsequence $w_{\lambda,\mu} -u_0\to\xi$ in $C^1_{loc}(B_{2d}(0) \setminus \{0\})$.
Let \[\alpha_q=\lim_{d\to0}\liminf_{\lambda,\mu\to\infty,\ \frac{\lambda}{\mu}\to0}\int_{B_d(p)}{e}^{w_{\lambda,\mu}-\frac{\nl}{\mu}}\left(1-\frac{\nl}{\lambda}\right) dx\ge 8\pi.\]
In view of Lemma \ref{lemma_RT} and Lemma \ref{lem3.2} as in \eqref{ee}, we see that
\begin{equation*}
 \begin{aligned}
{e}^{w_{\lambda,\mu}-\frac{\nl}{\mu}}\left(1-\frac{\nl}{\lambda}\right)
\to  e^{\xi+u_0} +\alpha_q\delta_{0},
 \end{aligned}
\end{equation*}
in the sense of measure on $B_{2d}(0)$.
By Corollary \ref{cor1} and Fatou's lemma, we have that $e^{\xi+u_0}\in L^1(B_{2d}(0))$.
Moreover, Green's representation formula implies that
\begin{equation*}
\begin{aligned}
\xi(x)=-\frac{\alpha_q}{2\pi} \ln|x| +\phi(x) +\eta(x),
\end{aligned}
\end{equation*}
where $\eta\in C^1(B_{r}(0))$ for every $r\in(0,d)$, and we let
\begin{equation}
  \label{definitionofphi}
\begin{aligned}
\phi(x)=\frac{1}{2\pi} \int_{B_{d}(0)} \ln\Big(\frac{1}{|x-y|}\Big) e^{(\xi+u_0)(y)}dy.
\end{aligned}
\end{equation}
We note that
\[ \phi\ge\frac{1}{2\pi}\ln\Big(\frac{1}{2d}\Big)\|e^{\xi +u_0}\|_{L^1(B_{d}(0))}
 \quad\mbox{on }~ B_{d}(0). \]
Then we see that $e^{\xi(x)}=|x|^{-\alpha_q/2\pi} e^{\phi+\eta}\ge c|x|^{-\alpha_q/2\pi}$
for $0<|x|\le d$ and some constant $c>0$.
Then the integrability of $e^{\xi+u_0}$ implies that
\begin{equation}
  \label{contradtioneq}
4\pi(1+m)>\alpha_q,
\end{equation}
where $m=m_{i}$ if $q=p_{i}\in B\cap {Z}$, and $m=0$ if $q\in B\setminus {Z}$.

Let $\phi_{\lambda,\mu}(x)=w_{\lambda,\mu}(x)-2m\ln|x|$. Then $\phi_{\lambda,\mu}$ satisfies
\begin{equation}
  \label{wisingentirernsoleq}
\Delta\pl+|x|^{2m}e^{\pl-\frac{\nl}{\mu}}(1-\frac{\nl}{\lambda})=0
 \quad\mbox{in }~ B_{d}(0).
\end{equation}
Multiplying (\ref{wisingentirernsoleq}) by $x\cdot\nabla \phi_{\lambda,\mu}$
and integrating over $B_r(0)$ for $r\in(0,d)$,  we get that
\begin{equation}
\begin{aligned}\label{eqp1}
&\int_{\partial B_r(0)} \left[ \frac{ (x\cdot\nabla \phi_{\lambda,\mu})^2}{|x|}
 -\frac{|x||\nabla \phi_{\lambda,\mu}|^2}{2} +|x|^{2m+1}e^{\pl-\frac{\nl}{\mu}}\left(1-\frac{\nl}{\lambda}\right)\right] d\sigma \\
&=  \int_{B_r(0)} \left(2+2m\right)e^{\pl+2m\ln|x|-\frac{\nl}{\mu}}\left(1-\frac{\nl}{\lambda}\right)dx
\\&\quad+  \int_{B_r(0)} e^{\pl+2m\ln|x|-\frac{\nl}{\mu}}\left[\left(\frac{\nl}{\lambda}-1\right)\frac{\nabla \nl}{\mu}\cdot x-\frac{\nabla \nl}{\lambda}\cdot x\right]dx.
\end{aligned}
\end{equation}
We recall the second equation in \eqref{main_eq}, which can be written into
\begin{equation}\label{eqn}
\Delta \nl=\mu^2 \left(1+\frac{|x|^{2m}e^{\pl-\frac{\nl}{\mu}}}{\lambda\mu}\right)\nl-\frac{ \mu}{\lambda}\left({\lambda}+{\mu}\right)|x|^{2m}e^{\pl-\frac{\nl}{\mu}}\ \ \textrm{in}\ \ \Omega.
\end{equation}
Multiplying (\ref{eqn}) by $x\cdot\nabla N_{\lambda,\mu}$
and integrating over $B_r(0)$ for $r\in(0,d)$,  we have
\begin{equation}
\begin{aligned}\label{eqp2}
&  \int_{B_r(0)} e^{\pl+2m\ln|x|-\frac{\nl}{\mu}}\left[\left(\frac{\nl}{\lambda}-1\right)\frac{\nabla \nl}{\mu}\cdot x-\frac{\nabla \nl}{\lambda}\cdot x\right]dx
\\&=\frac{1}{\mu^2}\int_{\partial B_r(0)} \left[ \frac{ (x\cdot\nabla N_{\lambda,\mu})^2}{|x|}
 -\frac{|x||\nabla N_{\lambda,\mu}|^2}{2}  \right] d\sigma -\int_{\partial B_r(0)} \frac{\nl^2|x|}{2}d\sigma+ \int_{B_r(0)}  \nl^2dx
\\&\ge\frac{1}{\mu^2}\int_{\partial B_r(0)} \left[ \frac{ r\left(\frac{\partial \nl}{\partial r}\right)^2}{2}
 -\frac{ \left(\frac{\partial \nl}{\partial \theta}\right)^2}{2r}- \frac{r\mu^2\nl^2}{2}  \right] d\sigma\\
 &\ge-\frac{1}{\mu^2}\int_{\partial B_r(0)} r\left[
 \frac{ \left(\frac{\partial \nl}{\partial \theta}\right)^2}{2r^2}+ \frac{\mu^2\nl^2}{2}  \right] d\sigma,
\end{aligned}
\end{equation}
here we used $\nabla\nl\cdot x=\frac{\partial\nl}{\partial r} r$ and $|\nabla \nl|^2=\left(\frac{\partial\nl}{\partial r}\right)^2+\frac{\left(\frac{\partial\nl}{\partial \theta}\right)^2}{r^2}$.

 We claim that  there is a sequence $\{r_l\}_{l\in \N}$ satisfying
\begin{equation}\label{rn}
\lim_{l\to\infty}r_l=0\ \ \ \textrm{and}\ \ \
\lim_{l\to\infty}\left(\frac{1}{\mu^2}\int_{\partial B_{r_l}(0)}r_l \left[
 \frac{ \left(\frac{\partial \nl}{\partial \theta}\right)^2}{r_l^2}+ \mu^2\nl^2 \right] d\sigma\right)=0.
\end{equation}
In order to prove the claim \eqref{rn}, we multiply the second equation in (\ref{main_eq}) by $ N_{\lambda,\mu}$
and integrate over $\Omega$.  Then we have $$\int_{\Omega}|\nabla \nl|^2+\mu^2 \left(1+\frac{\lambda}{\mu} e^{\ulm}\right)\nl^2dx
=\int_{\Omega}\lambda \mu^2\left(1+\frac{\lambda}{\mu}\right)e^{\ulm}\nl dx.$$ Together with Lemma \ref{lemma_RT} and \eqref{w1l}, we see that there is a constant $C>0$ such that
\begin{equation}\label{eqn11}
\begin{aligned}\frac{1}{\mu^2} \int_0^{d}\int_{\partial B_r(0)} \left[\frac{\left(\frac{\partial\nl}{\partial \theta}\right)^2}{r^2}+\mu^2 \nl^2 \right]&d\sigma  dr
 \le \frac{1}{\mu^2}\left( \int_{\Omega}|\nabla \nl|^2+\mu^2 \left(1+\frac{\lambda}{\mu} e^{\ulm}\right)\nl^2dx\right)
\\&=\int_{\Omega}\lambda \left(1+\frac{\lambda}{\mu}\right)e^{\ulm}\nl dx
 \le  \int_{\Omega}\lambda ^2\left(1+\frac{\lambda}{\mu}\right)e^{\ulm}dx\le C. \end{aligned}
\end{equation}
In view of \eqref{eqn11}, we note that there exists a sequence $r_l$ satisfying \eqref{rn}.
Otherwise, there exist constants $\e>0$ and $\bar{r}>0$ satisfying
\begin{equation*}
\frac{1}{\mu^2}\int_{\partial B_{r}(0)} r\left[
 \frac{ \left(\frac{\partial \nl}{\partial \theta}\right)^2}{r^2}+ \mu^2\nl^2 \right] d\sigma\ge \e  \ \ \textrm{for any}\ \ r\in(0,\bar{r}).
\end{equation*}
From \eqref{eqn11}, we have
\begin{equation*}
C\ge  \frac{1}{\mu^2}\int_0^{\bar{r}}\int_{\partial B_{r}(0)} \left[
 \frac{ \left(\frac{\partial \nl}{\partial \theta}\right)^2}{r^2}+ \mu^2\nl^2 \right] d\sigma dr\ge \int_0^{\bar{r}}\frac{\e}{r} dr =+\infty,
 \end{equation*}
which is a contradiction.

At this point, in view of   \eqref{eqp1}, \eqref{eqp2}, and \eqref{rn}, we see that for any $\e>0$, there is $l_{\e}$ such that if $l\ge l_\e$, then
\begin{equation}
\begin{aligned}
&\int_{\partial B_{r_l}(0)} \left[ \frac{ (x\cdot\nabla \phi_{\lambda,\mu})^2}{|x|}
 -\frac{|x||\nabla \phi_{\lambda,\mu}|^2}{2} +|x|^{2m+1}e^{\pl-\frac{\nl}{\mu}}\left(1-\frac{\nl}{\lambda}\right)\right] d\sigma \\
&=  \int_{B_{r_l}(0)} \left(2+2m\right)e^{\pl+2m\ln|x|-\frac{\nl}{\mu}}\left(1-\frac{\nl}{\lambda}\right)dx\\
&\quad+  \int_{B_{r_l}(0)} e^{\pl+2m\ln|x|-\frac{\nl}{\mu}}\left[\left(\frac{\nl}{\lambda}-1\right)\frac{\nabla \nl}{\mu}\cdot x
-\frac{\nabla \nl}{\lambda}\cdot x\right]dx\\
&\ge \int_{B_{r_l}(0)} \left(2+2m\right)e^{\pl+2m\ln|x|-\frac{\nl}{\mu}}\left(1-\frac{\nl}{\lambda}\right)dx
 -  \frac{1}{\mu^2}\int_{\partial B_{r_l}(0)} r_l\left[
 \frac{ \left(\frac{\partial \nl}{\partial \theta}\right)^2}{2r_l^2}
 + \frac{\mu^2\nl^2}{2}  \right] d\sigma\\
 &\ge \int_{B_{r_l}(0)} \left(2+2m\right)e^{\pl+2m\ln|x|-\frac{\nl}{\mu}}\left(1-\frac{\nl}{\lambda}\right)dx -\e.
\end{aligned} \label{tozeroasrzero}
\end{equation}
Let $\varphi_0(x)=(\xi+u_0)(x)-2m\ln|x|$.
Letting $\lambda,\mu\to\infty,\ \frac{\lambda}{\mu}\to0$ in (\ref{tozeroasrzero}), Lemma \ref{lemma_RT} and Lemma \ref{lem3.2} as in \eqref{ee} imply
\begin{equation}
  \label{finaltozeroasrzero}
\begin{aligned}
\int_{\partial B_{r_l}(0)} \Big[ \frac{(x\cdot\nabla \varphi_0)^2}{|x|}
 -\frac{|x||\nabla \varphi_0|^2}{2} &\quad+  |x|^{2m+1}e^{\varphi_0} \Big] d\sigma\\
& \ge (2m+2)\Big(\alpha_q +\int_{B_{r_l}(0)}e^{\xi+u_0}dx\Big)-\e.
\end{aligned}
\end{equation}
There exists a constant $c>0$ such that
$|x|^{2m}e^{\varphi_0} =e^{\xi+u_0}\le c|x|^{-\tau}e^\phi$ in $B_r(0)$ for any $r\in(0,d)$,
where $\tau\equiv\max\{0,\frac{\alpha_q}{2\pi}-2m\}$.
We note that $\tau\in[0,2)$ from (\ref{contradtioneq}).
In view of (\ref{definitionofphi}) and Corollary 1 in \cite{BM},
we see that $e^{|\phi|}\in L_{loc}^k(B_{d}(0))$ for any $k\in [1,\infty)$.
Since $\xi\in C^2_{loc}(B_{2d}(0)\setminus \{0\})$, we have
$|x|^{2m}e^{\varphi_0}\in L^t(B_{d}(0))$ for any $t\in(1,\frac{2}{\tau})$.
Then H\"{o}lder's inequality implies that $\phi\in L^\infty(B_{d}(0))$ and
\begin{equation}
  \label{s}
|x|^{2m}e^{\varphi_0} =e^{\xi+u_0}\le C|x|^{-\tau}\ \ \textrm{ for some constant}\ \ C>0.
\end{equation}
We note that for $|x|=r<d$,
\begin{equation*}
\begin{aligned}
 |\nabla\phi(x)|&\le\frac{1}{2\pi} \int_{B_{d}(0)}\frac{1}{|x-y|}e^{(\xi+u_0)(y)}dy \\
&= \frac{1}{2\pi} \Big[\int_{B_{d}(0)\setminus B_{r/2}(x)} \frac{e^{(\xi+u_0)(y)}}{|x-y|} dy
 +\int_{B_{d}(0)\cap B_{r/2}(x)} \frac{e^{(\xi+u_0)(y)}}{|x-y|} dy \Big].
\end{aligned}
\end{equation*}
Fix $t\in(1,\frac{2}{\tau})$ and choose a constant $a\in (0,\min\{1,2-\tau\})$ such that $\frac{at}{t-1} <2$.
H\"{o}lder's inequality implies that
\[ \int_{B_{d}(0)\setminus B_{r/2}(x)} \frac{e^{(\xi+u_0)(y)}}{|x-y|} dy
 \le  \int_{B_{d}(0)\setminus B_{r/2}(x)} \frac{Cr^{a-1}}{|y-x|^a}
 e^{(\xi +u_0)(y)} dy \le Cr^{a-1}. \]
Since $|x|=r$, we have $B_{r/2}(x)\subseteq\Omega\setminus B_{r/2}(0)$.
It follows from (\ref{s}) that
\[ \int_{B_{d}(0)\cap B_{r/2}(x)} \frac{e^{(\xi+u_0)(y)}}{|x-y|} dy
 \le \int_{|y-x|\le r/2} \frac{Cr^{-\tau}}{|y-x|} dy =O(r^{1-\tau}). \]
Since $a\in(0,2-\tau)$, we see that $|\nabla \phi(x)| =O(|x|^{a-1}+1)$ as $|x|\to 0$.
Consequently $\nabla\varphi_0(x)=-\frac{\alpha_q x}{2\pi|x|^2}+\nabla h(x)$
with $|\nabla h(x)| =O(|x|^{a-1}+1)$ as $|x|\to 0$.
Letting  $\e\to0$ and $r_l\to 0$ in (\ref{finaltozeroasrzero}),
we obtain that $(2m+2)\alpha_q\le\frac{\alpha_q^2}{4\pi}$, which contradicts (\ref{contradtioneq}).

Therefore, it follows from Harnack's inequality that $w_{\lambda,\mu}-u_0\to-\infty$ and
$w_{\lambda,\mu}\to-\infty$ uniformly on any compact subset of $\Omega\setminus B$.

\textbf{Step 3.}
In view of Lemma \ref{lemma_RT} and Corollary \ref{cor1}, along a subsequence,
${e}^{w_{\lambda,\mu}-\frac{\nl}{\mu}}\left(1-\frac{\nl}{\lambda}\right)$ converges to a nonnegative measure. However, this measure must be supported in $B$ since $w_{\lambda,\mu}\to-\infty$
uniformly on any compact set $K\subset \Omega\setminus B$.
Then we see that
 as $\lambda,\mu\to\infty,\ \frac{\lambda}{\mu}\to0$,
\[{e}^{w_{\lambda,\mu}-\frac{\nl}{\mu}}\left(1-\frac{\nl}{\lambda}\right)
 ~\to~ \sum_{q\in B} \alpha_q\delta_{q} ~\quad (\alpha_q\ge 8\pi) \]
in the sense of measure.
In view of the above arguments, we conclude that Case 2 implies Theorem \ref{BrezisMerletypealternatives}-(iii).  Now we complete the proof of Theorem \ref{BrezisMerletypealternatives}.

\qed

\section{Proof of Theorem \ref{thm1}}\label{sec4}

In this section, we are going to construct blow up solutions of (\ref{eq2}) such that
$$\sup u_{\lambda,\mu}\geq -c_0>-\infty.$$

Based on Theorem \ref{BrezisMerletypealternatives} above, our construction in this section was inspired by    the construction in \cite{LY2} where the authors constructed blow up solutions for the $SU(3)$ Chern-Simons system on torus using an entire regular solution for the single Chern-Simons equation as the building blocks.
\subsection{The approximate solution and the reduction}
Without loss of generality, we assume $|\Omega|=1$. We recall the equation (\ref{eq2}) as follows:
\begin{equation}\label{equation4.1}
\left\{\begin{array}{l}
\Delta (u+\frac{N}{\mu})=-\lambda^2 e^{u+u_0}\left(1-\frac{N}{\lambda}\right)+4\pi\mathfrak{M},\\
\Delta \frac{N}{\lambda}=\mu (\mu+\lambda e^{u+u_0})\frac{N}{\lambda}- \mu(\lambda+\mu)e^{u+u_0}
\end{array}
\right. \mbox{ in }\Omega.
\end{equation}

We are going to define the approximate solutions for (\ref{eq2}). Let  $w$ be  the radially symmetric solution of
\begin{equation}\label{decay}
\left\{\begin{array}{l}
\Delta w+e^{w}(1-e^w)=0 \ \textrm{in } \R^2,\\
w'(|x|)\rightarrow -\frac{2\mathfrak{M}}{|x|}+\frac{a_1(2\mathfrak{M}-2)}{|x|^{2\mathfrak{M}-1}}+O(\frac{1}{t^{2\mathfrak{M}+1}}), \ |x|\gg1,              \\
w(|x|)=-2\mathfrak{M}\ln t +I_1-\frac{a_1}{|x|^{2\mathfrak{M}-2}}+O(\frac{1}{|x|^{2\mathfrak{M}}}),   \ |x|\gg1,
\end{array}
\right.
\end{equation}where
$a_1$ and $I_1$ are constants (see \cite[Theorem 2.1, Lemma 2.6]{CFL} for the existence of $w$ satisfying \eqref{decay}).
We set
\begin{equation}\label{Uq}
U_{\lambda,q}(y)=\left\{\begin{array}{l}
w(\lambda|y-q|)-u_0(q)+4\pi \mathfrak{M}(\gamma(y,q)-\gamma(q,q))(1-\theta), \ y\in B_{d}(q),\\
w(d\lambda)-u_0(q)+4\pi \mathfrak{M}(G(y,q)-\gamma(q,q)+\frac{1}{2\pi}\ln d)(1-\theta), \ y \in \Omega /B_{d}(q),
\end{array}
\right.
\end{equation}
where \[\theta=\frac{1}{2\mathfrak{M}\lambda^{2\mathfrak{M}-2}}\{\frac{a_1(2\mathfrak{M}-2)}{d^{2\mathfrak{M}-2}}+O(\frac{1}{\lambda^2})\},\] which makes $U_{\lambda,q}\in C^1(\Omega)$, We would find a solution of \eqref{eq2} with the following form:
\begin{equation}\label{error}
u+\frac{N}{\mu}=U_{\lambda,q}+\varphi\ \ \textrm{and}\ \ \frac{N}{\lambda}=  e^{U_{\lambda,q}+u_0}(1+\varphi)+S,
\end{equation}
here $(\varphi,S)$ would be regard as an error term.
For the convenience, we also denote
\begin{equation}\label{f}
\begin{split}
w_{\lambda,q}(y)&=w(\lambda|y-q|),\ \ h(\varphi,S)=e^{U_{\lambda,q}+u_0}(1+\varphi)+S, \\
F(t)&=e^t(1-e^t)\ \ \textrm{and}\ \ f(t)=F'(t)=e^t(1-2e^t).
\end{split}
\end{equation}
The notation $1_{B_{2d}(q)}$ means that $1_{B_{2d}(q)}(x)=1$ if $x\in B_{2d}(q)$ and $1_{B_{2d}(q)}(x)=0$ if $x\notin{B_{2d}(q)}$.

Then the  equation (\ref{eq2}) is reduced to a system for $(\varphi,S)$:
\begin{equation}\label{eq_4}
\left\{\begin{array}{l}
\Delta \varphi+\lambda^2 f(w_{\lambda,q}(y))\cdot1_{B_{2d}(q)}\varphi=g_{1,\lambda,\mu}(\varphi,S),\\
\Delta S-\mu^2S=g_{2,\lambda,\mu}(\varphi,S),
\end{array}
\right.
\end{equation}
where
\begin{equation}\begin{aligned}\label{g1g2}
g_{1,\lambda,\mu}(\varphi,S):=&\quad-\Delta U_{\lambda,q} + \lambda^2f(w_{\lambda,q}(y))\cdot1_{B_{2d}(q)}\varphi -\lambda^2F(U_{\lambda,q}+u_0)+ 4\pi\mathfrak{M}\\
&\quad -\lambda^2f(U_{\lambda,q}+u_0)\varphi+\lambda^2F(U_{\lambda,q}+u_0) (1+\varphi-e^{\varphi-\frac{\lambda}{\mu}h(\varphi,S)})\\
&\quad + \lambda^2e^{U_{\lambda,q}+u_0}S +  \lambda^2e^{U_{\lambda,q}+u_0}(e^{\varphi-\frac{\lambda}{\mu}h(\varphi,S)}-1) (e^{U_{\lambda,q}+u_0}\varphi+S),\\
g_{2,\lambda,\mu}(\varphi,S):=&\quad-\Delta \left\{e^{U_{\lambda,q}+u_0}(1+\varphi)\right\} + \mu^2e^{U_{\lambda,q}+u_0}\left\{1+\varphi-e^{\varphi-\frac{\lambda}{\mu}h(\varphi,S)}\right\}\\
&\quad +\lambda\mu\left\{e^{2U_{\lambda,q}+2u_0+\varphi-\frac{\lambda}{\mu}h(\varphi,S)}(1+\varphi) +  (S -1) e^{U_{\lambda,q}+u_0+\varphi-\frac{\lambda}{\mu}h(\varphi,S)}\right\}.
\end{aligned}\end{equation}

\subsection{The linear and nonlinear problem}
In this subsection, we are going to study the linear and nonlinear problem. First, let us introduce the space we are going to work in. Fix a small constant $0<\alpha<\frac{1}{2}$. Let us introduce two function spaces $X_{\alpha,q}$ and $Y_{\alpha,q}$. Define
\begin{equation}
\rho(z)=(1+|z|)^{1+\frac{\alpha}{2}}, \quad \textrm{and} \quad \bar{\rho}(z)=\frac{1}{(1+|z|)(\ln(2+|z|))^{1+\frac{\alpha}{2}}}.
\end{equation}

We say that $\psi \in X_{\alpha,q}$ if
\begin{equation}
\|\psi\|_{X_{\alpha,q}}^2=\|(\Delta \tilde{\psi})\rho\|^2_{L^2(B_{2d\lambda}(0))}+\|\tilde{\psi}\bar{\rho}\|_{L^2(B_{2d\lambda}(0))}^2+\||\Delta \psi|^2+\psi^2\|_{L^1(\Omega /B_d(q))}<+\infty
\end{equation}
where $\tilde{\psi}(z)=\psi(\lambda^{-1}z+q)$.  We say $\psi \in Y_{\alpha, q}$ if
\begin{equation*}
\|\psi\|_{Y_{\alpha,q}}^2= \frac{1}{\lambda^4}\|\tilde{\psi}\rho\|^2_{L^2(B_{2d\lambda}(0))}+\|\psi\|^2_{L^2(\Omega \setminus B_d(q))}<+\infty.
\end{equation*}

Let $\chi(y)$ be a smooth cut off function such that $\chi=1$ in $B_{d}(0)$, $\chi=0$ in $\Omega / B_{2d}(0)$ and $0\leq \chi\leq 1$. Define
\begin{equation}\label{eq13}
W_{q,j}=\chi(y-q)\frac{\partial w(\lambda|y-q|)}{\partial q_j}, \ j=1,2.
\end{equation}

Let
\begin{equation}\label{eq15}
Z_{q,j}=-\Delta W_{q,j}+\lambda^2 e^{w(\lambda|y-q|)}W_{q,j}, \ j=1,2.
\end{equation}
We define two subspace of $X_{\alpha, q}$ and $Y_{\alpha}$ as
\begin{equation}\label{eq16}
E_{q}=\{u: u\in X_{\alpha,q}, \int_{\Omega}Z_{q,j}u=0, \ j=1,2 \}
\end{equation}
and
\begin{equation}\label{eq17}
F_{q}=\{u: u\in Y_{\alpha,q}, \int_{\Omega}W_{q,j}u=0,\ j=1,2 \}.
\end{equation}

Define the projection operator to $F_{q}$ as
\begin{equation}\label{eq18}
Q_{q}u=u-\sum_{j=1}^2 c_{j}Z_{q,j}
\end{equation}
where $c_{j}$ are chosen so that $Q_{q}u \in F_{q}$. We have the following estimates:
\begin{lemma}\label{lemma2}
There holds $\|Q_{q}u\|_{Y_{\alpha}}\leq C\|u\|_{Y_{\alpha}}$ for some positive constant $C$ independent of $q$.
\end{lemma}

First we need the preliminary results for the linear operators $L_{1,q},\ L_2$ in \cite{FLL, LY2}, where
\begin{equation}\label{eq10}
\left\{\begin{array}{l}
L_{1,q}(\varphi):=\Delta \varphi+\lambda^2f(w(\lambda|y-q|))\cdot1_{B_{2d}(q)}\varphi, \\
L_2(S):=\Delta S-\mu^2 S.
\end{array}
\right.
\end{equation}

\begin{theorem}[Theorem B.3 in \cite{LY2}]\label{theorema}
The operator $Q_{q}L_{1,q}$ is an isomorphism from $E_{q}$ to $F_{q}$. Moreover if $\varphi \in E_{q}$ and $g_1\in F_{q}$ satisfies $Q_{q}L_{1,q}(\varphi)=g_1$, then there exists a constant $C>0$ independent of $q$ such that
\begin{equation*}
\|\varphi\|_{L^\infty(\Omega)}+\|\varphi\|_{X_{\alpha,q}}\leq C\ln \lambda \|g_1\|_{Y_{\alpha,q}}.
\end{equation*}
\end{theorem}
\begin{theorem}\label{theoremb}
The operator
\begin{equation*}
L_2: W^{2,2}(\Omega)\to L^2(\Omega)
\end{equation*}
is an isomorphism. Moreover, for any $S\in W^{2,2}(\Omega)$ and $g_2\in L^2(\Omega)$ satisfies $L_2(S)=g_2$, there exists a positive constant $C$ independent of $\mu$ such that
\begin{equation}\label{eq19} \left\{\begin{array}{l}
\mu^2\|S\|_{L^2(\Omega)}+\mu\|S\|_{L^\infty(\Omega)}+\mu\|\nabla S\|_{L^2(\Omega)}+\|\partial_i\partial_j S\|_{L^2(\Omega)}\leq C\|g_2\|_{L^2(\Omega)},
\\  \mu^2\|S\|_{L^\infty(\Omega)}\leq C\|g_2\|_{L^\infty(\Omega)}\ \ \textrm{if}\ \ g_2\in L^\infty(\Omega).
\end{array}
\right.
\end{equation}
\end{theorem}
\begin{proof}
It has been shown in  \cite[Theorem 2.4 ]{FLL} that $L_2$ is an isomorphism,  $\mu\| S\|_{L^\infty(\Omega)}\leq C\|g_2\|_{L^2(\Omega)}$, and   $\mu^2\| S\|_{L^\infty(\Omega)}\leq C\|g_2\|_{L^\infty(\Omega)}$  if $g_2\in L^\infty(\Omega)$. In order to complete the proof of Theorem \ref{theorema}, if is enough to prove
\[
\mu^2\|S\|_{L^2(\Omega)}+\mu\|\nabla S\|_{L^2(\Omega)}+\|\partial_i\partial_j S\|_{L^2(\Omega)}\leq C\|g_2\|_{L^2(\Omega)}.
\]
Here we prove the estimate (\ref{eq19}).

Multiply the equation $L_2(S)=g_2$ by $S$ and integrate over $\Omega$, one has
\begin{equation}
\int_{\Omega}|\nabla S|^2+\mu^2\int_{\Omega}S^2=-\int_{\Omega}g_2S.
\end{equation}
By Holder's inequality,
\begin{equation*}
\int_{\Omega}|\nabla S|^2+\mu^2\int_{\Omega}S^2\leq \mu^{-2}\int_{\Omega}g_2^2+\frac{\mu^2}{4}\int_{\Omega}S^2,
\end{equation*}
which implies that
\begin{equation*}
\mu^2\|S\|_{L^2(\Omega)}+\mu\|\nabla S\|_{L^2(\Omega)}\leq C\|g_2\|_{L^2(\Omega)}.
\end{equation*}
Since $\Delta S-\mu^2 S=g_2$, one can get that
\begin{equation}
\|\Delta S\|_{L^2(\Omega)}\leq \mu^2\|S\|_{L^2(\Omega)}+\|g_2\|_{L^2(\Omega)},
\end{equation}
so
\begin{equation*}
\mu^2\|S\|_{L^2(\Omega)}+\mu\|\nabla S\|_{L^2(\Omega)}+\|\partial_i\partial_j S\|_{L^2(\Omega)}\leq C\|g_2\|_{L^2(\Omega)}.
\end{equation*}
Now, we complete the proof of Theorem \ref{theoremb}.

\end{proof}

Next, let us consider the corresponding nonlinear problem. We define an operator $\Psi$ by
\[
\Psi(\varphi,S) = \Big((Q_{q}L_{1,q})^{-1}(Q_{q}g_{1,\lambda,\mu}(\varphi,\hat{S})), \hat{S}\Big),
\]
where $\hat{S}=L^{-1}_2(g_{2,\lambda,\mu}(\varphi,S))$, and a subset $M_{\lambda,\mu}$ of $E_q\times W^{2,2}(\Omega)$ by
\begin{equation*}\begin{aligned}
M_{\lambda,\mu}=\left\{(\varphi,S)\in E_q\times W^{2,2}(\Omega)\ \ \Big| \ \   \|(\varphi,S)\|\leq \frac{(\ln\lambda)^2}{\lambda}  \right\}.
\end{aligned}\end{equation*}
where
\begin{equation*}
\|(\varphi,S)\|:= \|\varphi\|_{L^\infty(\Omega)}+\|\varphi\|_{X_{\alpha,q}}+\frac{(\ln\lambda)^2}{\mu^2}(\mu^2\|S\|_{L^2(\Omega)}+\mu\|S\|_{L^\infty(\Omega)}+\|S\|_{W^{2,2}(\Omega)}).
\end{equation*}
We note that if $(\varphi,S)\in M_{\lambda,\mu}$, then \begin{equation*} \|\varphi\|_{L^\infty(\Omega)}+\|\varphi\|_{X_{\alpha,q}}\leq \frac{(\ln\lambda)^2}{\lambda}, \ \textrm{and}\
 \mu^2\|S\|_{L^2(\Omega)}+\mu\|S\|_{L^\infty(\Omega)}+\|S\|_{W^{2,2}(\Omega)} \leq \frac{\mu^2}{\lambda}.\end{equation*}

To apply contraction argument we need some estimations for the right hand side of (\ref{eq_4}).

\begin{lemma}\label{deltaeUqu}
There exists a constant $C$ such that
\[
\|\Delta \{e^{U_{\lambda,q}+u_0}(1+\varphi)\}\|_{L^2(\Omega)}\leq C\lambda(1+\|\varphi\|_{L^\infty(\Omega)}+\|\varphi\|_{X_{\alpha,q}}).
\]
for any $(\varphi,S)\in M_{\lambda,\mu}$.
\end{lemma}

\begin{proof}
We have
\begin{equation}\begin{aligned}
\Delta \{e^{U_{\lambda,q}+u_0}(1+\varphi)\}=&e^{U_{\lambda,q}+u_0}(1+\varphi)\Delta (U_{\lambda,q}+u_0) + e^{U_{\lambda,q}+u_0}\Delta\varphi\\
&\quad +2 e^{U_{\lambda,q}+u_0}\nabla (U_{\lambda,q}+u_0) \cdot\nabla\varphi + e^{U_{\lambda,q}+u_0}(1+\varphi)|\nabla(U_{\lambda,q}+u_0)|^2.
\end{aligned}\end{equation}

First, we consider the $L^2$ norm of $\Delta \{e^{U_{\lambda,q}+u_0}(1+\varphi)\}$ in $B_d(q)$.

In $B_d(q)$, we get that
\begin{equation}\begin{aligned}\label{laplaceUq}
\Delta U_{\lambda,q}=&\Delta w_{\lambda,q}(y) +4\pi \mathfrak{M}(1-\theta)\Delta\gamma(y,q)\\
=&\quad-\lambda^2 F(w_{\lambda,q}(y)) +4\pi \mathfrak{M}(1-\theta)
\end{aligned}\end{equation}
and
\[
\Delta u_0=-4\pi \mathfrak{M}.
\]

Moreover, since $\gamma(y,q)$ and $u_0$ are smooth functions in $B_d(q)$, we get that
\[
\frac{\partial U_{\lambda,q}}{\partial y_j}=\lambda w_{\lambda,q}'(y)\frac{y_j-q_j}{|y-q|} +4\pi \mathfrak{M}(1-\theta)\frac{\partial \gamma(y,q)}{\partial y_j}, \quad j=1,2,
\]
and
\[
\frac{\partial u_0}{\partial y_j}=O(1), \quad j=1,2.
\]

By the definitions of $U_{\lambda,q}$ and $u_0$, we obtain that
\[
e^{U_{\lambda,q}}=e^{w_{\lambda,q}(y)+O(1)} \quad \textrm{and} \quad e^{u_0}=O(1) \quad \textrm{in} \ B_d(q).
\]

Then,
\begin{equation}\begin{aligned}
&\|\Delta \{e^{U_{\lambda,q}+u_0}(1+\varphi)\}\|_{L^2(B_d(q))}\\
&\leq  C\{\|e^{w_{\lambda,q}(y)}(1+\varphi)\{\lambda^2F(w_{\lambda,q}(y))+4\pi \mathfrak{M} \theta\}\|_{L^2(B_d(q))}+\|e^{w_{\lambda,q}(y)}\Delta \varphi\|_{L^2(B_d(q))}\\
&\quad +\|e^{w_{\lambda,q}(y)}(\lambda w_{\lambda,q}'(y)+1)|\nabla \varphi|\|_{L^2(B_d(q))}+\|e^{w_{\lambda,q}(y)}(1+\varphi)(\lambda^2(w_{\lambda,q}')^2+1)\|_{L^2(B_d(q))}\}\\
&=: C\{ I + II + III + IV\}.
\end{aligned}\end{equation}

Let
\begin{equation}\label{scaling}
z:=\lambda(y-q) \quad \textrm{and} \quad \tilde{u}(z):=u(\lambda^{-1}z+q)
\end{equation}
for any function $u$. Then,
\begin{equation}
\begin{aligned}\label{BdL2}
I&=\|e^{w_{\lambda,q}(y)}(1+\varphi)\{\lambda^2F(w_{\lambda,q}(y))+4\pi \mathfrak{M}\theta\}\|_{L^2(B_d(q))}\\
&=\lambda^{-1}\|e^{w(z)}(1+\tilde{\varphi})\{\lambda^2F(w(z))+4\pi \mathfrak{M} \theta\}\|_{L^2(B_{d\lambda}(0))}\\
&\leq C\lambda\|e^{w(z)}(1+\tilde{\varphi})\|_{L^2(B_{d\lambda}(0))}\leq C\lambda(1+\|\varphi\|_{L^\infty(\Omega)}),
\end{aligned}\end{equation}
where $\tilde{\varphi}(z):=\varphi(\lambda^{-1}z+q)$ as defined in (\ref{scaling}). In the last inequality in (\ref{BdL2}), we used the decay rate of $w(z)$ in (\ref{decay}).
From (\ref{decay}), we also have
\[
II=\|e^{w_{\lambda,q}(y)}\Delta \varphi\|_{L^2(B_d(q))}\leq \lambda \|e^{w(z)}\Delta_z \tilde{\varphi}(z)\|_{L^2(B_{d\lambda}(0))}
\leq C\lambda \|(\Delta \tilde{\varphi})\rho\|_{L^2(B_{d\lambda}(0))}.
\]
We can rewrite $III$ by
\begin{equation}
\begin{aligned}
III&=\|e^{w_{\lambda,q}(y)}(\lambda w_{\lambda,q}'+1)|\nabla \varphi|\|_{L^2(B_d(q))}\\
&\leq \lambda \|e^{w(z)} w'(z)|\nabla_z \tilde{\varphi}|\|_{L^2(B_{d\lambda}(0))}+ \|e^{w(z)}|\nabla_z \tilde{\varphi}|\|_{L^2(B_{d\lambda}(0))}.
\end{aligned}\end{equation}
Since $|e^{w(z)}w'(z)|\leq C(1+|z|)^{-2\mathfrak{M}-1}$,
\[
\lambda \|e^{w(z)} w'(z)|\nabla_z \tilde{\varphi}|\|_{L^2(B_{d\lambda}(0))}
\leq C\lambda\|(1+|z|)^{-2\mathfrak{M}-1}|\nabla_z \tilde{\varphi}|\|_{L^2(B_{d\lambda}(0))}.
\]
There exist finite number of points $\{z_i\}$ such that $B_{d\lambda}(0)\subseteq \cup_{i}B_1(z_i)\subseteq \cup_{i}B_2(z_i)\subseteq B_{2d\lambda}(0)$ and for any index $i$, $|\{z_j\mid B_1(z_i)\cap B_1(z_j)\neq\emptyset\}|\leq c_1,$ where a positive constant $c_1$ is independent of $\lambda$. Then, we have
\begin{equation}\begin{aligned}\label{gradw2}
\lambda\|(1+|z|)^{-2\mathfrak{M}-1}|\nabla_z \tilde{\varphi}|\|_{L^2(B_1(z_i))}
\leq
\lambda 2^{2\mathfrak{M}+1}
\|(1+|z_i|)^{-2\mathfrak{M}-1}|\nabla_z \tilde{\varphi}|\|_{L^2(B_1(z_i))}.
\end{aligned}\end{equation}
Moreover, there exists some constant $c_2$ from $W^{2,2}$ estimate satisfying
\begin{equation}\begin{aligned}\label{gradw20}
&\|\nabla_z \tilde{\varphi}\|_{L^2(B_1(z_i))}\\
&\leq c_2(\|\Delta \tilde{\varphi}\|_{L^2(B_2(z_i))}+\|\tilde{\varphi}\|_{L^2(B_2(z_i))})\\
&\leq 3^{1+\frac{\alpha}{2}}c_2\Big((1+|z_i|)^{-1-\frac{\alpha}{2}}\|\rho\Delta \tilde{\varphi}\|_{L^2(B_2(z_i))}+(1+|z_i|)\{\ln (2+|z_i|)\}^{1+\frac{\alpha}{2}}\|\tilde{\varphi}\bar{\rho}\|_{L^2(B_2(z_i))}\Big).
\end{aligned}\end{equation}
Therefore, it follows that
\begin{equation}\begin{aligned}\label{gradw21}
&\|e^{w_{\lambda,q}(y)}\lambda w_{\lambda,q}'|\nabla \varphi|\|_{L^2(B_d(q))}\\
&\leq 2^{2\mathfrak{M}+1}3^{1+\frac{\alpha}{2}}c_2\lambda \sum_{i}
\Big((1+|z_i|)^{-2\mathfrak{M}-2-\frac{\alpha}{2}}\|\rho\Delta \tilde{\varphi}\|_{L^2(B_2(z_i))}+(1+|z_i|)^{-2\mathfrak{M}}\{\ln (2+|z_i|)\}^{1+\frac{\alpha}{2}}\|\tilde{\varphi}\bar{\rho}\|_{L^2(B_2(z_i))}\Big)\\
&\leq 2^{2\mathfrak{M}+1}3^{1+\frac{\alpha}{2}}c_1c_2\lambda(\|\rho\Delta \tilde{\varphi}\|_{L^2(B_{2d\lambda}(0))}+\|\tilde{\varphi}\bar{\rho}\|_{L^2(B_{2d\lambda}(0))})\\
&\leq C\lambda \|\varphi\|_{X_{\alpha,q}}.
\end{aligned}\end{equation}
Similarly, we have
\[
\|e^{w(z)}|\nabla_z \tilde{\varphi}|\|_{L^2(B_{d\lambda}(0))}\leq C \|\varphi\|_{X_{\alpha,q}}.
\]
From the decay rate of $e^{w(z)}w'(z)$, we have
\begin{equation}\begin{aligned}
IV&=\|e^{w_{\lambda,q}(y)}(1+\varphi)(\lambda^2(w'_{\lambda,q})^2+1)\|_{L^2(B_d(q))}\\
&\leq \lambda \|e^{w(z)}(1+\tilde{\varphi})(w'(z))^2\|_{L^2(B_{d\lambda}(0))}+\lambda^{-1} \|e^{w(z)}(1+\tilde{\varphi})\|_{L^2(B_{d\lambda}(0))}\\
&\leq C\lambda (1+\|\varphi\|_{L^\infty(B_{d}(q))}).
\end{aligned}\end{equation}

\medskip
\noindent
Next, we consider $\|\Delta \{e^{U_{\lambda,q}+u_0}(1+\varphi)\}\|_{L^2(\Omega\setminus B_d(q))}$.
In $\Omega\setminus B_d(q)$, we have
\[\Delta U_{\lambda,q}=4\pi \mathfrak{M}(1-\theta), \ \ \Delta u_0=-4\pi \mathfrak{M}+4\pi\sum^n_{i=1}m_i\delta_{p_i},
\]
and
\[\nabla U_{\lambda,q}=4\pi \mathfrak{M}(1-\theta)\nabla G(y,q),\ \ \nabla u_0=-4\pi\sum^n_{i=1} m_i\nabla G(y,p_i).\]

We see that  $G(y,q)$ is a smooth function in $\Omega\setminus B_d(q)$ and thus
\begin{equation}\begin{aligned}\label{Bd1}
&| \Delta U_{\lambda,q}|+|\nabla U_{\lambda,q}|+|e^{U_{\lambda,q}}|\\&=| 4\pi \mathfrak{M}(1-\theta)|+|4\pi \mathfrak{M}(1-\theta)\nabla G(y,q)|+ e^{U_{\lambda,q}}=O(1) \ \ \textrm{in}\ \Omega\setminus B_d(q).
\end{aligned}\end{equation}
We  also see that
\[
u_0(y)=2m_i\ln|y-p_i|-4\pi m_i\gamma(y,p_i)-4\pi \sum_{i\neq j} m_j G(y,p_j), \quad \textrm{in} \ B_d(p_i),
\]
where $m_i$ is the multiplicity of $p_i$.
Since $\gamma(y,p_i)$ and $G(y,p_j)$, $j\neq i$,  are smooth functions in $B_d(p_i)$, we have   $e^{u_0}(y)=O(|y-p_i|^{2m_i})\in C^2(B_d(p_i))$.  Obviously, we also have  $   e^{u_0} \in C^2(\Omega\setminus \cup_{i=1}^n B_d(p_i))$. From this observation, we see that
\begin{equation}\label{Bd2}
|\Delta e^{u_0}|+|\nabla e^{u_0}|+|e^{u_0}|=O(1)\ \ \textrm{in}\ \Omega\setminus B_d(q).
\end{equation}
We also see that
\begin{equation}\begin{aligned}\label{detail_exp}
&\Delta \{e^{U_{\lambda,q}+u_0}(1+\varphi)\}\\
&=\left\{e^{U_{\lambda,q}}\left(\Delta U_{\lambda,q}+|\nabla U_{\lambda,q}|^2\right)e^{u_0}+2e^{U_{\lambda,q}}\nabla U_{\lambda,q}\cdot\nabla e^{u_0}+e^{U_{\lambda,q}}\Delta e^{u_0}\right\}(1+\varphi)\\
&\quad + 2\left(e^{U_{\lambda,q}+u_0}\nabla U_{\lambda,q}+ e^{U_{\lambda,q}}\nabla e^{u_0}\right)\cdot\nabla\varphi + e^{U_{\lambda,q}+u_0}\Delta \varphi.
\end{aligned}\end{equation}
We are going to estimate  $\|\Delta \{e^{U_{\lambda,q}+u_0}(1+\varphi)\}\|_{L^2(\Omega\setminus B_d(q))}$ by dividing the region  $\Omega\setminus B_d(q)$ into  $\Omega\setminus  B_{\frac{3}{2}d}(q)$ and $ B_{\frac{3}{2}d}(q)\setminus B_d(q)$.
\\
Firstly, the estimations \eqref{Bd1}-\eqref{detail_exp} and $W^{2,2}$-estimation imply
\begin{equation}\begin{aligned}
&\|\Delta \{e^{U_{\lambda,q}+u_0}(1+\varphi)\}\|_{L^2(\Omega\setminus  B_{\frac{3}{2}d}(q))}
\\
&\leq C(1+\|\varphi\|_{L^2(\Omega\setminus  B_{\frac{3}{2}d}(q))}+\|\nabla\varphi\|_{L^2(\Omega\setminus  B_{\frac{3}{2}d}(q))}+\|\Delta\varphi\|_{L^2(\Omega\setminus  B_{\frac{3}{2}d}(q))})\\
&\leq C(1+\|\varphi\|_{L^2(\Omega\setminus  B_{d}(q))}+\|\Delta\varphi\|_{L^2(\Omega\setminus  B_{d}(q))})\\&\le C(1+\|\varphi\|_{X_{\alpha,q}}).
\end{aligned}\end{equation}
Secondly, in $ B_{\frac{3}{2}d}(q)\setminus B_{d}(q)$,  the estimations \eqref{Bd1}-\eqref{detail_exp} again imply
\begin{equation*}\begin{aligned}
&\|\Delta \{e^{U_{\lambda,q}+u_0}(1+\varphi)\}\|_{L^2( B_{\frac{3}{2}d}(q)\setminus B_{d}(q))}\\&\le \|\left(e^{U_{\lambda,q}}\left(\Delta U_{\lambda,q}+|\nabla U_{\lambda,q}|^2\right)e^{u_0}+2e^{U_{\lambda,q}}\nabla U_{\lambda,q}\cdot\nabla e^{u_0}+e^{U_{\lambda,q}}\Delta e^{u_0}\right)(1+\varphi)\|_{L^2( B_{\frac{3}{2}d}(q)\setminus B_{d}(q))}\\
&\quad + \|2\left(e^{U_{\lambda,q}+u_0}\nabla U_{\lambda,q}+ e^{U_{\lambda,q}}\nabla e^{u_0}\right)\cdot\nabla\varphi\| _{L^2( B_{\frac{3}{2}d}(q)\setminus B_{d}(q))}
+\| e^{U_{\lambda,q}+u_0}\Delta \varphi\|_{L^2( B_{\frac{3}{2}d}(q)\setminus B_{d}(q))}
\\&\le C(1+\|\varphi\|_{L^2(\Omega\setminus  B_{ d}(q))}+ \|\Delta\varphi\|_{L^2(\Omega\setminus  B_{ d}(q))}) +2 \|e^{U_{\lambda,q}+u_0}\nabla\left( U_{\lambda,q}+   {u_0}\right)\cdot\nabla\varphi\| _{L^2( B_{\frac{3}{2}d}(q)\setminus B_{d}(q))}.
\end{aligned}\end{equation*}
In order to estimate $\| e^{U_{\lambda,q}+u_0}\nabla\left( U_{\lambda,q}+   {u_0}\right)\cdot\nabla\varphi\| _{L^2( B_{\frac{3}{2}d}(q)\setminus B_{d}(q))}$, we note that  the estimation \eqref{decay} yields
\[
e^{U_{\lambda,q}+u_0}\leq C\lambda^{-2\mathfrak{M}}, \quad \textrm{in} \  B_{\frac{3}{2}d}(q)\setminus B_{d}(q).
\]
From the definition of $U_{\lambda,q}$ and $u_0$, it is clear that $\nabla\left( U_{\lambda,q}+   {u_0}\right)$ are uniformly bounded in $ B_{\frac{3}{2}d}(q)\setminus B_{d}(q)$.
Therefore, we obtain that
\begin{equation}\begin{aligned}
&\|e^{U_{\lambda,q}+u_0}\nabla(U_{\lambda,q}+u_0)\cdot\nabla\varphi\|_{L^2(B_{\frac{3}{2}d}(q)\setminus B_{d}(q))}\\
&\leq C\lambda^{-2\mathfrak{M}}\|\nabla_y\varphi(y)\|_{L^2(B_{\frac{3}{2}d}(q)\setminus B_{d}(q))}\leq C\lambda^{-2\mathfrak{M}}\|\nabla\tilde{\varphi}\|_{L^2(B_{\frac{3}{2}d\lambda}(0)\setminus B_{d\lambda}(0))}\\
&\leq C\lambda^{-2\mathfrak{M}}\sum_{i}\|\nabla\tilde{\varphi}\|_{L^2(B_1(z_i))}\\
&\leq C\lambda^{-2\mathfrak{M}}\sum_{i}(\|\Delta\tilde{\varphi}\|_{L^2(B_2(z_i))} +\|\tilde{\varphi}\|_{L^2(B_2(z_i))})
\end{aligned}\end{equation}
where $\{B_1(z_i)\}$ is finite covering of $B_{\frac{3}{2}d\lambda}(0)\setminus B_{d\lambda}(0)$ as in the calculus in $III$.
Since
\[
\|\Delta\tilde{\varphi}\|_{L^2(B_2(z_i))}\leq (1+|z_i|)^{-1-\frac{\alpha}{2}}\|\Delta\tilde{\varphi}\rho\|_{L^2(B_2(z_i))}\leq C\lambda^{-1-\frac{\alpha}{2}}\|\Delta\tilde{\varphi}\rho\|_{L^2(B_2(z_i))}
\]
and
\[
\|\tilde{\varphi}\|_{L^2(B_2(z_i))}\leq (1+|z_i|)\{\ln(2+|z_i|)\}^{1+\frac{\alpha}{2}}\|\tilde{\varphi}\bar{\rho}\|_{L^2(B_2(z_i))}\leq C\lambda^{1+\frac{\alpha}{2}}\|\tilde{\varphi}\bar{\rho}\|_{L^2(B_2(z_i))},
\]
we obtain
\begin{equation}\begin{aligned}
&\|e^{U_{\lambda,q}+u_0}\nabla(U_{\lambda,q}+u_0)\cdot\nabla\varphi\|_{L^2(B_{\frac{3}{2}d}(q)\setminus B_{d}(q))}\\
&\leq C\lambda^{-2\mathfrak{M}+1+\frac{\alpha}{2}}(\|\Delta\tilde{\varphi}\rho\|_{L^2(B_{2d\lambda}(0))} +\|\tilde{\varphi}\bar{\rho}\|_{L^2(B_{2d\lambda}(0))})\\
&\leq C\lambda^{-2\mathfrak{M}+1+\frac{\alpha}{2}}\|\varphi\|_{X_{\alpha,q}}.
\end{aligned}\end{equation}
Therefore, the proof is complete.

\end{proof}


\begin{proposition}\label{contractionmap}
There exists a fixed point $({\varphi}_q, {S}_q)\in M_{\lambda,\mu}$ of the operator $\Psi$.
\end{proposition}

\begin{proof}
In order to prove Proposition \ref{contractionmap}, it is enough to show that   $\Psi$ is a contraction map
from $M_{\lambda,\mu}$ to $M_{\lambda,\mu}$ due to the  contraction mapping
theorem.
We are going to prove that $\Psi$ is a contraction map
from $M_{\lambda,\mu}$ to $M_{\lambda,\mu}$ with the following two steps.

\medskip
\noindent
\textbf{Step 1.} We claim that $\Psi(\varphi,S)\in M_{\lambda,\mu}$ for any $(\varphi,S)\in M_{\lambda,\mu}$. First, we consider $\|g_{2,\lambda,\mu}(\varphi,S)\|_{L^2(\Omega)}$.\\
From the definition of (\ref{f}), we note that
\begin{equation}\begin{aligned}
\|\frac{\lambda}{\mu}h(\varphi,S)\|_{L^\infty(\Omega)}\leq
&\frac{\lambda}{\mu}\|e^{U_{\lambda,q}+u_0}(1+\varphi)\|_{L^\infty(\Omega)}+\frac{\lambda}{\mu}\|S\|_{L^\infty(\Omega)}\\
\leq& \frac{C\lambda}{\mu}(1+\|\varphi\|_{L^\infty(\Omega)})+\frac{\lambda}{\mu}\|S\|_{L^\infty(\Omega)}.
\end{aligned}\end{equation}
This implies that for any $(\varphi,S)\in M_{\lambda,\mu}$, $\|\frac{\lambda}{\mu}h(\varphi,S)\|_{L^\infty(\Omega)}=O(1)$.

From the definition of $U_{\lambda,q}$ and $u_0$, $e^{U_{\lambda,q}}=e^{w_{\lambda,q}(y)+O(1)}$ and $u_0=O(1)$ in $B_d(q)$. This implies that
\[
\|e^{U_{\lambda,q}+u_0}\|_{L^2(B_d(q))}=C\lambda^{-1}\|e^{w(z)}\|_{L^2(B_{d\lambda}(0))}\leq C\lambda^{-1}.
\]
In $\Omega\setminus B_d(q)$, $e^{U_{\lambda,q}}=e^{-2\mathfrak{M}\ln(d\lambda)+O(1)}$ and $e^{u_0}=O(1)$. This implies that
\[
\|e^{U_{\lambda,q}+u_0}\|_{L^2(\Omega\setminus B_d(q))}=O(\lambda^{-2\mathfrak{M}}).
\]
It follows that $\|e^{U_{\lambda,q}+u_0}\|_{L^2(\Omega)}\leq C\lambda^{-1}$ and $\|e^{U_{\lambda,q}+u_0}\|_{L^\infty(\Omega)}=O(1)$.

From (\ref{g1g2}), we get that
\begin{equation}\begin{aligned}
\|g_{2,\lambda,\mu}\|_{L^2(\Omega)}\leq &\|\Delta\{e^{U_{\lambda,q}+u_0}(1+\varphi)\}\|_{L^2(\Omega)}+\|\mu^2e^{U_{\lambda,q}+u_0}(1+\varphi-e^{\varphi-\frac{\lambda}{\mu}h(\varphi,S)})\|_{L^2(\Omega)}\\
&\quad +\lambda\mu\|e^{U_{\lambda,q}+u_0+\varphi
-\frac{\lambda}{\mu}h(\varphi,S)}
\{e^{U_{\lambda,q}+u_0}(1+\varphi)+S-1\}\|_{L^2(\Omega)}.
\end{aligned}\end{equation}

Using Taylor's Theorem, we see that for some $0\leq\sigma\leq1$,
\begin{equation}\begin{aligned}
&\|\mu^2e^{U_{\lambda,q}+u_0}(1+\varphi-e^{\varphi-\frac{\lambda}{\mu}h(\varphi,S)})\|_{L^2(\Omega)}\\
&\leq\mu^2\|e^{U_{\lambda,q}+u_0}\{\frac{\lambda}{\mu}S+\frac{\lambda}{\mu}e^{U_{\lambda,q}+u_0}(1+\varphi)
-\frac{1}{2}e^{\sigma\varphi-\frac{\sigma\lambda}{\mu}h(\varphi,S)}(\varphi-\frac{\lambda}{\mu}h(\varphi,S))^2\}\|_{L^2(\Omega)}\\
&\leq\lambda\mu\|e^{U_{\lambda,q}+u_0}S\|_{L^2(\Omega)}+\lambda\mu\|e^{U_{\lambda,q}+u_0}\|_{L^2(\Omega)}(1+\|\varphi\|_{L^\infty(\Omega)})+
C\mu^2\|e^{U_{\lambda,q}+u_0}(\varphi-\frac{\lambda}{\mu}h(\varphi,S))^2\|_{L^2(\Omega)}\\
&\leq\lambda\mu\|S\|_{L^2(\Omega)}+C\mu(1+\|\varphi\|_{L^\infty(\Omega)})+
C\mu^2\|\varphi\|_{L^\infty(\Omega)}^2\|e^{U_{\lambda,q}+u_0}\|_{L^2(\Omega)}+C\lambda^2\|S\|_{L^\infty(\Omega)}\|S\|_{L^2(\Omega)}.
\end{aligned}\end{equation}

We also obtain
\begin{equation}\begin{aligned}
&\lambda\mu\|e^{U_{\lambda,q}+u_0+\varphi-\frac{\lambda}{\mu}h(\varphi,S)}\{e^{U_{\lambda,q}+u_0}(1+\varphi)+S-1\}\|_{L^2(\Omega)}\\
&\leq\lambda\mu\Big(\|e^{U_{\lambda,q}+u_0+\varphi-\frac{\lambda}{\mu}h(\varphi,S)}S\|_{L^2(\Omega)}+
\|e^{U_{\lambda,q}+u_0+\varphi-\frac{\lambda}{\mu}h(\varphi,S)}\{e^{U_{\lambda,q}+u_0}(1+\varphi)-1\}\|_{L^2(\Omega)}\Big)\\
&\leq\lambda\mu\|S\|_{L^2(\Omega)}+
\mu(1+\|\varphi\|_{L^\infty(\Omega)}).
\end{aligned}\end{equation}

From Lemma \ref{deltaeUqu}, we obtain that
\begin{equation}\begin{aligned}
\|g_{2,\lambda,\mu}\|_{L^2(\Omega)}
\leq \mu(1+\|\varphi\|_{L^\infty(\Omega)})+\lambda\|\varphi\|_{X_{\alpha,q}}+
\frac{\mu^2}{\lambda}\|\varphi\|_{L^\infty(\Omega)}^2+\lambda\|S\|_{L^2(\Omega)}(\mu+\lambda\|S\|_{L^\infty(\Omega)}).
\end{aligned}\end{equation}

Therefore,
\begin{equation}\label{estimateg2}
\begin{aligned}
\|g_{2,\lambda,\mu}(\varphi,S)\|_{L^2(\Omega)}\leq C\mu(1+\frac{\mu(\ln\lambda)^4}{\lambda^3}), \quad \textrm{for any} \ (\varphi,S)\in M_{\lambda,\mu}.
\end{aligned}\end{equation}

Next, we consider $\|g_{1,\lambda,\mu}(\varphi,\hat{S})\|_{Y_{\alpha}}$, where $\hat{S}=L^{-1}_2(g_{2,\lambda,\mu}(\varphi,S))$.

From (\ref{laplaceUq}), we can rewrite (\ref{g1g2}) in $B_{2d}(q)$ as
\begin{equation}\begin{aligned}
g_{1,\lambda,\mu}(\varphi,\hat{S})=&[\lambda^2F(w_{\lambda,q}(y))\cdot1_{B_{2d}(q)}-\lambda^2F(U_{\lambda,q}+u_0)+\lambda^2f(w_{\lambda,q}(y))\varphi-\lambda^2f(U_{\lambda,q}+u_0)\varphi]\\
&\quad +[\lambda^2F(U_{\lambda,q}+u_0)(1+\varphi-e^{\varphi-\frac{\lambda}{\mu}h(\varphi,\hat{S})})]+4\pi \mathfrak{M}\theta\\
&\quad +[\lambda^2e^{U_{\lambda,q}+u_0}\hat{S}+\lambda^2e^{U_{\lambda,q}+u_0}(e^{\varphi-\frac{\lambda}{\mu}h(\varphi,\hat{S})}-1)(e^{U_{\lambda,q}+u_0}\varphi+\hat{S})]\\
=:&I+II+III+IV
\end{aligned}\end{equation}

From the decay of $w$ and mean value theorem, there exists some $0\leq\sigma,\sigma'\leq1$ such that
\begin{equation}\begin{aligned}
I=&\lambda^2F(w_{\lambda,q}(y))\cdot1_{B_{2d}(q)}-\lambda^2F(U_{\lambda,q}+u_0)+\lambda^2f(w_{\lambda,q}(y))\varphi-\lambda^2f(U_{\lambda,q}+u_0)\varphi\\
=&-\lambda^2f\Big(\sigma w_{\lambda,q}+(1-\sigma)(U_{\lambda,q}+u_0)\Big)\Big(u_0(y)-u_0(q)+4\pi \mathfrak{M}(\gamma(y,q)-\gamma(q,q))(1-\theta)\Big)\cdot1_{B_{d}(q)}\\
&\quad -\lambda^2f'\Big(\sigma' w_{\lambda,q}+(1-\sigma')(U_{\lambda,q}+u_0)\Big)\Big(u_0(y)-u_0(q)+4\pi \mathfrak{M}(\gamma(y,q)-\gamma(q,q))(1-\theta)\Big)\varphi\cdot1_{B_{2d}(q)}\\&\quad+O(\lambda^{-2\mathfrak{M}+2}).
\end{aligned}\end{equation}
This implies that
\begin{equation}\begin{aligned}
&\|\tilde{I}(1+|z|)^{1+\alpha/2}\|_{L^2(B_{2d\lambda}(0))}\\
&\leq C\lambda^2(1+\|\varphi\|_{L^\infty(\Omega)})\|e^{w(z)}(1+|z|)^{1+\alpha/2}\Big(u_0(\lambda^{-1}z+q)-u_0(q)\Big)\|_{L^2(B_{2d\lambda}(0))}\\
&\quad+C\lambda^2(1+\|\varphi\|_{L^\infty(\Omega)})\|e^{w(z)}(1+|z|)^{1+\alpha/2}\Big(\gamma(\lambda^{-1}z+q,q)-\gamma(q,q)\Big)\|_{L^2(B_{2d\lambda}(0))}
+O(\lambda^{-2\mathfrak{M}+4+\frac{\alpha}{2}}).
\end{aligned}\end{equation}
We again apply mean value theorem to $u_0(\lambda^{-1}z+q)-u_0(q)$ and $\gamma(\lambda^{-1}z+q,q)-\gamma(q,q)$, then for some $0\leq\sigma,\sigma'\leq1$ we get that
\[
|u_0(\lambda^{-1}z+q)-u_0(q)|= \lambda^{-1}|\nabla_y u_0(y)\Big|_{y=\sigma\lambda^{-1}z+q}||z|
\]
and
\[
|\gamma(\lambda^{-1}z+q,q)-\gamma(q,q)|= \lambda^{-1}|\nabla_y \gamma(y)\Big|_{y=\sigma\lambda^{-1}z+q}||z|.
\]
Since $u_0$ and $\gamma$ are regular and $e^{w(z)}\leq C(1+|z|)^{-2\mathfrak{M}}$ in $B_{2d\lambda}(0)$, we get that
\[
\|\tilde{I}(1+|z|)^{1+\alpha/2}\|_{L^2(B_{2d\lambda}(0))}\leq C\lambda.
\]

From Taylor's Theorem, there exists some $0\leq\sigma\leq1$ such that
\begin{equation}\label{taylorargue}
\begin{aligned}
II&=\lambda^2F(U_{\lambda,q}+u_0)\Big(1+\varphi-1-\varphi+\frac{\lambda}{\mu}h(\varphi,\hat{S})
-\frac{1}{2}e^{\sigma(\varphi-\frac{\lambda}{\mu}h(\varphi,\hat{S}))}(\varphi-\frac{\lambda}{\mu}h(\varphi,\hat{S}))^2\Big)\\
&=\lambda^2F(U_{\lambda,q}+u_0)\Big(\frac{\lambda}{\mu}h(\varphi,\hat{S})-\frac{1}{2}e^{\sigma(\varphi-\frac{\lambda}{\mu}h(\varphi,\hat{S}))}
(\varphi-\frac{\lambda}{\mu}h(\varphi,\hat{S}))^2\Big).
\end{aligned}\end{equation}
We have
\begin{equation}\begin{aligned}
&\|e^{w(z)}(1+|z|)^{1+\alpha/2}\hat{S}(\lambda^{-1}z+q)\|_{L^2(B_{2d\lambda}(0))}\\
&= \Big(\int_{B_{2d\lambda}(0)} e^{2w(z)}(1+|z|)^{2+\alpha}\hat{S}^2(\lambda^{-1}z+q)dz \Big)^{1/2}\\
&\leq C\Big(\lambda^2\int_{B_{2d\lambda}(0)} \hat{S}^2(\lambda^{-1}z+q)\lambda^{-2}dz \Big)^{1/2}\\
&= C\lambda\Big(\int_{B_{2d}(q)} \hat{S}^2(y)dy \Big)^{1/2}=C\lambda\|\hat{S}\|_{L^2(B_{2d(q)})},
\end{aligned}\end{equation}
where $y:=\lambda^{-1}z+q$. From Theorem \ref{theoremb} and $\|g_{2,\lambda,\mu}(\varphi,\hat{S})\|_{L^2(\Omega)}$, we note that
\begin{equation}\begin{aligned}
\|\frac{\lambda}{\mu}h(\varphi,\hat{S})\|_{L^\infty(\Omega)}\leq
&\frac{\lambda}{\mu}\|e^{U_{\lambda,q}+u_0}(1+\varphi)\|_{L^\infty(\Omega)}+\frac{\lambda}{\mu}\|\hat{S}\|_{L^\infty(\Omega)}\\
\leq& \frac{C\lambda}{\mu}(1+\|\varphi\|_{L^\infty(\Omega)})+\frac{\lambda}{\mu}\|\hat{S}\|_{L^\infty(\Omega)}\leq 1.
\end{aligned}\end{equation}
This implies that
\begin{equation}\begin{aligned}
&\lambda^{-2}\|\tilde{II}(1+|z|)^{1+\frac{\alpha}{2}}\|_{L^2(B_{2d\lambda}(0))}\\
&=\|F(\tilde{U}_q+\tilde{u}_0)\Big(\frac{\lambda}{\mu}h(\tilde{\varphi},\tilde{\hat{S}})
-\frac{1}{2}e^{\sigma(\tilde{\varphi}-\frac{\lambda}{\mu}h(\tilde{\varphi},\tilde{\hat{S}}))}
(\tilde{\varphi}-\frac{\lambda}{\mu}h(\tilde{\varphi},\tilde{\hat{S}}))^2\Big)(1+|z|)^{1+\frac{\alpha}{2}}\|_{L^2(B_{2d\lambda}(0))}\\
&\leq\frac{C\lambda}{\mu}\|e^{w(z)}(1+|z|)^{1+\alpha/2}(1+\tilde{\varphi})\|_{L^2(B_{2d\lambda}(0))}
+\frac{C\lambda}{\mu}\|e^{w(z)}(1+|z|)^{1+\alpha/2}\tilde{\hat{S}}\|_{L^2(B_{2d\lambda}(0))}\\
&\quad +C\|e^{w(z)}(1+|z|)^{1+\alpha/2}\Big(\tilde{\varphi}^2+\frac{\lambda^2}{\mu^2}(1+\tilde{\varphi})^2+\frac{\lambda^2}{\mu^2}\tilde{\hat{S}}^2\Big)\|_{L^2(B_{2d\lambda}(0))}\\
&\leq \frac{C\lambda}{\mu}(1+\|\varphi\|_{L^\infty(\Omega)})+C\|\varphi\|_{L^\infty(\Omega)}^2 +\frac{C\lambda^2}{\mu} \|\hat{S}\|_{L^2(B_{2d}(q))}(1+\frac{\lambda}{\mu} \|\hat{S}\|_{L^\infty(\Omega)}).
\end{aligned}\end{equation}

From the definition of $\theta$, it follows that
\begin{equation}\begin{aligned}
\|\tilde{III}(1+|z|)^{1+\frac{\alpha}{2}}\|_{L^2(B_{2d\lambda}(0))}&=\|4\pi\mathfrak{M}\theta (1+|z|)^{1+\frac{\alpha}{2}}\|_{L^2(B_{2d\lambda}(0))}\\
&\leq C\lambda^{-2\mathfrak{M}+2}\|(1+|z|)^{1+\frac{\alpha}{2}}\|_{L^2(B_{2d\lambda}(0))}\leq C\lambda^{-2\mathfrak{M}+4+\frac{\alpha}{2}}.
\end{aligned}\end{equation}

Finally,
\begin{equation}\begin{aligned}\label{IV}
\lambda^{-2}\|\tilde{IV}(1+|z|)^{1+\frac{\alpha}{2}}\|_{L^2(B_{2d\lambda}(0))}
&\leq C\|e^{w(z)}(1+|z|)^{1+\frac{\alpha}{2}}\varphi(e^{\varphi-\frac{\lambda}{\mu}h}-1)\|_{L^2(B_{2d\lambda}(0))}\\
&\quad +C\|e^{w(z)}(1+|z|)^{1+\frac{\alpha}{2}}\hat{S}(\lambda^{-1}z+q)\|_{L^2(B_{2d\lambda}(0))}.
\end{aligned}\end{equation}

Since
\[
|e^{\varphi-\frac{\lambda}{\mu}h}-1|\leq \frac{C\lambda}{\mu} + C|\varphi|+ \frac{C\lambda}{\mu}|\hat{S}|,
\]
this implies that
\begin{equation}\begin{aligned}
\lambda^{-2}\|\tilde{IV}(1+|z|)^{1+\frac{\alpha}{2}}\|_{L^2(B_{2d\lambda}(0))}\leq \frac{C\lambda}{\mu}\|\varphi\|_{L^\infty(\Omega)}+C\|\varphi\|_{L^\infty(\Omega)}^2 + C\lambda \|\hat{S}\|_{L^2(B_{2d}(q))}.
\end{aligned}\end{equation}

In $\Omega\setminus B_d(q)$, we have
\begin{equation}\begin{aligned}
&g_{1,\lambda,\mu}(\varphi,\hat{S})\\
&=\lambda^2F(U_{\lambda,q}+u_0)(\varphi-e^{\varphi-\frac{\lambda}{\mu}h(\varphi,\hat{S})})+\lambda^2f(w_{\lambda,q}(y))\cdot1_{B_{2d}(q)}\varphi+4\pi\mathfrak{M}\theta\\
&\quad -\lambda^2f(U_{\lambda,q}+u_0)\varphi+\lambda^2e^{U_{\lambda,q}+u_0}
(e^{U_{\lambda,q}+u_0+\varphi-\frac{\lambda}{\mu}h(\varphi,\hat{S})}\varphi-e^{U_{\lambda,q}+u_0}\varphi+e^{\varphi-\frac{\lambda}{\mu}h(\varphi,\hat{S})}\hat{S}).
\end{aligned}\end{equation}
Since $e^{U_{\lambda,q}}\leq c\lambda^{-2\mathfrak{M}}$ in $\Omega\setminus B_d(q)$, it follows that
\[
\|g_{1,\lambda,\mu}(\varphi,\hat{S})\|_{L^2(\Omega\setminus B_d(q))}\leq
C\lambda^{-2\mathfrak{M}+2}(1+\|\varphi\|_{L^\infty(\Omega)}+ \|\hat{S}\|_{L^2(\Omega\setminus B_d(q))}).
\]

Therefore, Theorem \ref{theoremb} and the assumption $\lambda^2\ln\lambda<\mu$ yield for any $(\varphi,S)\in M_{\lambda,\mu}$
\begin{equation}\label{estimateg1}
\begin{aligned}
\|g_{1,\lambda,\mu}(\varphi,\hat{S})\|_{Y_{\alpha}}
&\leq C\frac{1}{\lambda}+C(\frac{\lambda}{\mu}+\lambda^{-2\mathfrak{M}+2})\|\varphi\|_{L^\infty(\Omega)}
+C\|\varphi\|_{L^\infty(\Omega)}^2+C\lambda\|\hat{S}\|_{L^2(\Omega)}(1+\frac{\lambda^2}{\mu^2} \|\hat{S}\|_{L^\infty(\Omega)})\\
&\leq \frac{(\ln\lambda)^{1/2}}{\lambda}.
\end{aligned}\end{equation}

From Theorem \ref{theorema} and \ref{theoremb}, the inequalities (\ref{estimateg2}) and (\ref{estimateg1}) yield that $\Psi(\varphi,S)\in M_{\lambda,\mu}$ for any $(\varphi,S)\in M_{\lambda,\mu}$.

\medskip
\noindent
\textbf{Step 2.} We claim that for any $(\varphi_1,S_1)$ and $(\varphi_2,S_2)$ in $M_{\lambda,\mu}$, there exists some constant $0<\tau<1$ such that
\begin{equation}\begin{aligned}\label{claims2}
\|\Psi(\varphi_1,S_1)- \Psi(\varphi_2,S_2)\|<\tau\|(\varphi_1,S_1)- (\varphi_2,S_2)\|.
\end{aligned}\end{equation}
Firstly, we see that
\begin{equation*}\begin{aligned}
&\|g_{2,\lambda,\mu}(\varphi_1,S_1)- g_{2,\lambda,\mu}(\varphi_2,S_2)\|_{L^2(\Omega)}\\
&\leq \|\Delta\{e^{U_{\lambda,q}+u_0}(\varphi_1-\varphi_2)\}\|_{L^2(\Omega)}
+\|\mu^2e^{U_{\lambda,q}+u_0}(\varphi_1-\varphi_2+e^{\varphi_2-\frac{\lambda}{\mu}h(\varphi_2,S_2)}
-e^{\varphi_1-\frac{\lambda}{\mu}h(\varphi_1,S_1)})\|_{L^2(\Omega)}\\
&\quad +\lambda\mu\|e^{U_{\lambda,q}+u_0+\varphi_1-\frac{\lambda}{\mu}h(\varphi_1,S_1)}\{e^{U_{\lambda,q}+u_0}(1+\varphi_1)+S_1-1\}
\\&\quad
-e^{U_{\lambda,q}+u_0+\varphi_2-\frac{\lambda}{\mu}h(\varphi_2,S_2)}\{e^{U_{\lambda,q}+u_0}(1+\varphi_2)+S_2-1\}\|_{L^2(\Omega)}
\end{aligned}\end{equation*}
By the similar way in Step 1, we can get that
\begin{equation}\begin{aligned}\label{diffg2}
&\|g_{2,\lambda,\mu}(\varphi_1,S_1)- g_{2,\lambda,\mu}(\varphi_2,S_2)\|_{L^2(\Omega)}
\\&=O(\mu(1+\frac{\mu(\ln\lambda)^2}{\lambda^2})\|\varphi_1-\varphi_2\|_{L^\infty(\Omega)}+\lambda\|\varphi_1-\varphi_2\|_{X_{\alpha,q}} +\lambda\mu\|S_1-S_2\|_{L^2(\Omega)}).
\end{aligned}\end{equation}
Next, we consider $\|g_{1,\lambda,\mu}(\varphi_1,\hat{S}_1)-g_{1,\lambda,\mu}(\varphi_2,\hat{S}_2)\|_{Y_{\alpha}}$, where $\hat{S}_i=L^{-1}_2(g_{2,\lambda,\mu}(\varphi_i,S_i))$, $i=1,2$.
We see that
\begin{equation}\begin{aligned}
&g_{1,\lambda,\mu}(\varphi_1,\hat{S}_1)-g_{1,\lambda,\mu}(\varphi_2,\hat{S}_2)\\&=[\lambda^2f(w_{\lambda,q}(y))\cdot1_{B_{2d}(q)}-\lambda^2f(U_{\lambda,q}+u_0)](\varphi_1-\varphi_2)\\
&\quad +[\lambda^2F(U_{\lambda,q}+u_0)(\varphi_1-\varphi_2+e^{\varphi_2-\frac{\lambda}{\mu}h(\varphi_2,\hat{S}_2)}-e^{\varphi_1-\frac{\lambda}{\mu}h(\varphi_1,\hat{S}_1)})]\\
&\quad +\lambda^2e^{U_{\lambda,q}+u_0}[e^{U_{\lambda,q}+u_0}
(e^{\varphi_1-\frac{\lambda}{\mu}h(\varphi_1,\hat{S}_1)}\varphi_1
-e^{\varphi_2-\frac{\lambda}{\mu}h(\varphi_2,\hat{S}_2)}\varphi_2)\\
&\quad \quad +e^{U_{\lambda,q}+u_0}(\varphi_2-\varphi_1)+\hat{S}_1e^{\varphi_1-\frac{\lambda}{\mu}h(\varphi_1,\hat{S}_1)}-\hat{S}_2e^{\varphi_2-\frac{\lambda}{\mu}h(\varphi_2,\hat{S}_2)}].
\end{aligned}
\end{equation}
By the similar way in Step 1, we can get that
\begin{equation}\label{diffg1}
\begin{aligned}
&\|g_{1,\lambda,\mu}(\varphi_1,\hat{S}_1)-g_{1,\lambda,\mu}(\varphi_2,\hat{S}_2)\|_{Y_{\alpha}}
\\
&=O(\frac{(\ln\lambda)^2}{\lambda}\|\varphi_1-\varphi_2\|_{L^\infty(\Omega)}
+\lambda\|\hat{S}_1-\hat{S}_2\|_{L^2(\Omega)}).
\end{aligned}\end{equation}
In view of \eqref{diffg2}-\eqref{diffg1}, Theorem \ref{theorema}, and Theorem \ref{theoremb}, we can prove the claim \eqref{claims2}.

\end{proof}

\noindent
{\bf Completion of the proof of Theorem \ref{thm1}}.
By Proposition \ref{contractionmap}, we get that for any large $\lambda, \mu>0$ and any $q$ close to $\hat{q}$, where ${\hat{q}}$ is a non-degenerate critical point of $u_0$,
there are $(\varphi_q,S_q)\in M_{\lambda,\mu}$ and constants $c_{q,j}$ such that
\begin{equation}\label{final}
\left\{\begin{array}{l}
\Delta \varphi_q+\lambda^2 f(w(\lambda|y-q|))\cdot1_{B_{2d}(q)}\varphi_q=g_{1,\lambda,\mu}(\varphi_q,S_q)+\sum_{j=1}^{2}
c_{q,j}Z_{q,j},\\
\Delta S_q-\mu^2S_q=g_{2,\lambda,\mu}(\varphi_q,S_q).
\end{array}
\right.
\end{equation}

In the following, we will choose $q$ suitably (depending on $\lambda,\mu>0$) such that the corresponding constants $c_{q,j}$ are zero and thus $(u_{\lambda,\mu}, N_{\lambda,\mu})$ is a solution to \eqref{main_eq}, where \[u_{\lambda,\mu}+\frac{N_{\lambda,\mu}}{\mu}=U_{\lambda,q}+\varphi_q\ \ \textrm{and}\ \ \frac{N_{\lambda,\mu}}{\lambda}=  e^{U_{\lambda,q}+u_0}(1+\varphi_q)+S_q.\]
It is standard to prove the following.

\begin{lemma}\label{lemmac3}
If
\begin{equation}\label{eqw41}
\int_{\Omega} \left(\Delta \varphi_q+\lambda^2 f(w(\lambda|y-q|))\cdot1_{B_{2d}(q)}\varphi_q-g_{1,\lambda,\mu}(\varphi_q,S_q)\right)W_{q,j}dx=0, \  j=1,2,
\end{equation}then $c_{q,j}=0$ for $j=1,2$.
\end{lemma}
\medskip
\noindent

Next we have the following reduced problem:

\begin{lemma}\label{reducedproblem}
\begin{equation}
\begin{aligned}\label{finalclaim}
&\int_{\Omega} \left(\Delta \varphi_q+\lambda^2 f(w(\lambda|y-q|))\cdot1_{B_{2d}(q)}\varphi_q-
g_{1,\lambda,\mu}(\varphi_q,S_q)\right)W_{q,j}dy\\
&=a_0D_j u_0(q)+o(1) \ \ \textrm{as}\ \ \lambda, \mu\to \infty, \frac{(\ln\lambda)\lambda^2}{\mu}\to 0, \ j=1,2,
\end{aligned}
\end{equation}
for some $a_0\neq 0$.
\end{lemma}
\begin{proof}
Since $W_{q,j}=\chi(y-q)\frac{\partial w(\lambda|y-q|)}{\partial q_j}=-\lambda\chi(y-q)\frac{\partial w(z)}{\partial z_j}\Big|_{z=\lambda(y-q)}$,  we see that
\begin{equation}
\begin{aligned}\label{re1}
&\int_{\Omega} \left(\Delta_y \varphi_q+\lambda^2 f(w(\lambda|y-q|))\cdot1_{B_{2d}(q)}\varphi_q-g_{1,\lambda,\mu}(\varphi_q,S_q)\right)W_{q,j}dy\\
&=-\lambda\int_{B_{2d\lambda}(0)} \left(\Delta_z \tilde{\varphi_q} +f(w(z))\tilde{\varphi}_q\right)
\chi(\lambda^{-1}z)\frac{\partial w(z)}{\partial z_j}dz\\
&\quad +\lambda\int_{B_{2d\lambda}(0)}[F(w(z))\cdot1_{B_{d\lambda}(0)}-F((U_{\lambda,q}+u_0)(\lambda^{-1}z+q))]
\chi(\lambda^{-1}z)\frac{\partial w(z)}{\partial z_j}dz\\
&\quad +\lambda\int_{B_{2d\lambda}(0)}[f(w(z))-f((U_{\lambda,q}+u_0)(\lambda^{-1}z+q))]\tilde{\varphi}_q(z)\chi(\lambda^{-1}z)
\frac{\partial w(z)}{\partial z_j}dz\\
&\quad +\lambda\int_{B_{2d\lambda}(0)}[F((U_{\lambda,q}+u_0)(\lambda^{-1}z+q))(1+\tilde{\varphi}_q-e^{\tilde{\varphi}_q
-\frac{\lambda}{\mu}h(\tilde{\varphi}_q,\tilde{S}_q)})
+4\pi \mathfrak{M}\theta\lambda^{-2}]\chi(\lambda^{-1}z)\frac{\partial w(z)}{\partial z_j}dz\\
&\quad +\lambda\int_{B_{2d\lambda}(0)}e^{(U_{\lambda,q}+u_0)(\lambda^{-1}z+q)}[e^{(U_{\lambda,q}+u_0)(\lambda^{-1}z+q)}(e^{\tilde{\varphi}_q-\frac{\lambda}{\mu}h(\tilde{\varphi}_q,\tilde{S}_q)}-1)\tilde{\varphi}_q
+e^{\tilde{\varphi}_q-\frac{\lambda}{\mu}h(\tilde{\varphi}_q,\tilde{S}_q)}\tilde{S}_q]\chi(\lambda^{-1}z)\frac{\partial w(z)}{\partial z_j}dz\\
&=: I+II+III+IV+V.
 \end{aligned}
\end{equation}

We will estimate the above term by term.

\medskip
\noindent
\textbf{Step 1.} We claim that $I=o(1)$. \\
Note that
\begin{equation}\label{der}
\Delta \left(\frac{\partial w}{\partial z_j}\right)+ f(w)\frac{\partial w}{\partial z_j}=0\ \ \textrm{in}\ \ \RN.
\end{equation}
Together with the integration by parts, we have
\begin{equation}
\begin{aligned}\label{der1}
&\lambda\int_{B_{2d\lambda}(0)} \left(\Delta_z \tilde{\varphi}_q+f(w(z))\tilde{\varphi}_q\right)\chi(\lambda^{-1}z)\frac{\partial w}{\partial z_j}dz\\
&= \lambda\int_{B_{2d\lambda}(0)}  \left(\Delta_z \left(\chi(\lambda^{-1}z)\frac{\partial w}{\partial z_j}\right)
+f(w(z))\chi(\lambda^{-1}z)\frac{\partial w}{\partial z_j}\right)\tilde{\varphi}_qdz\\
&=\lambda\int_{B_{2d\lambda}(0)}  \Bigg(\Delta_z \left(\chi(\lambda^{-1}z)\right)\frac{\partial w}{\partial z_j}
+2\nabla_z(\chi(\lambda^{-1}z))\cdot\nabla_z\left(\frac{\partial w}{\partial z_j}\right)\Bigg)\tilde{\varphi}_qdz\\
&\quad +\lambda\int_{B_{2d\lambda}(0)}  \Bigg(\Delta_z\left(\frac{\partial w}{\partial z_j}\right)+f(w(z))
\frac{\partial w(z)}{\partial z_j}\Bigg)\chi(\lambda^{-1}z)\tilde{\varphi}_qdz\\
&=\lambda\int_{B_{2d\lambda}(0)}  \Bigg(\Delta_z \left(\chi(\lambda^{-1}z)\right)\frac{\partial w}{\partial z_j}
+2\nabla_z(\chi(\lambda^{-1}z))\cdot\nabla_z\left(\frac{\partial w}{\partial z_j}\right)\Bigg)\tilde{\varphi}_qdz=O(\|\varphi_q\|_{L^\infty(\Omega)})=o(1).
\end{aligned}
\end{equation}

\noindent
\textbf{Step 2.} We claim that  $II=a_0D_j u_0(q)+o(1)$ for some $a_0\neq0$.\\
Recall  the definition of $U_{\lambda,q}$ from \eqref{Uq}, and let  $\Gamma(y)= u_0(y)+4\pi \mathfrak{M}(1-\theta)\gamma(y,q)$.
From the radial symmetry and decay rate of $w(z)$, we see that for some $0\leq\tau\leq1$
\begin{equation}\begin{aligned}\label{der2}&\lambda\int_{B_{2d\lambda}(0)}[F(w(z))\cdot1_{B_{d\lambda}(0)}-F((U_{\lambda,q}+u_0)(\lambda^{-1}z+q))]\chi(\lambda^{-1}z)\frac{\partial w(z)}{\partial z_j}dz
\\&=\lambda\int_{B_{d\lambda}(0)}\left[F(w(z))-F\left(w(z)+ \Gamma(\lambda^{-1}z+q)-\Gamma(q)\right)\right]\chi(\lambda^{-1}z)\frac{\partial w(z)}{\partial z_j}dz+O(\lambda^{-2\mathfrak{M}+2})
\\&=\lambda\int_{B_{d\lambda}(0)}f(w(z))\left(\Gamma(q)-\Gamma(\lambda^{-1}z+q)\right) \frac{\partial w(z)}{\partial z_j}dz\\
&\quad -\lambda\int_{B_{d\lambda}(0)}\frac{f'\left(w(z)+ \tau\left(\Gamma(\lambda^{-1}z+q)-\Gamma(q)\right)\right)}{2}\left(\Gamma(q)-\Gamma(\lambda^{-1}z+q)\right)^2 \frac{\partial w(z)}{\partial z_j}dz +O(\lambda^{-2\mathfrak{M}+2})
\\&=-\left(\int_{B_{d\lambda}(0)}f(w(z))\nabla\Gamma(q) \cdot z\frac{\partial w(z)}{\partial z_j}dz\right)+O(\lambda^{-1})
\\&=-\left(\int_{B_{d\lambda}(0)}f(w(z))D_j\Gamma(q) w'(z) \frac{z_j^2}{|z|}dz\right)+O(\lambda^{-1})
\\&=-\pi D_j\Gamma(q) \left(\int_{0}^{\infty}f(w(r))\frac{dw(r)}{dr}r^2dr\right)+O(\lambda^{-1})
\\&=-\pi D_j\Gamma(q) \left(F(w(r)) r^2\Big|_0^{\infty}-2\int_{0}^{\infty}F(w(r))rdr\right)+O(\lambda^{-1})
\\&= D_j \Gamma(q) \int_{\RN}F(w(x))dx +O(\lambda^{-1}).
\end{aligned}\end{equation}
It has been known that  $D_j \gamma(q,q)=0$ for any $q\in\Omega$ and $j=1,2$, which implies $D_j\Gamma(q) =D_j u_0(q)$. Together with $F(w(x))>0$ for all $x\in\RN$, we prove  the claim.

\medskip
\noindent
\textbf{Step 3.}We claim that  $III+IV+V+VI=o(1)$.
For some $0\leq\tau\leq1$ we note that
\begin{equation}
\begin{aligned}
\Big|III\Big|\leq&\lambda\int_{B_{2d\lambda}(0)}|f'\Big(\tau w+(1-\tau)(U_{\lambda,q}+u_0)\Big)\Big(U_{\lambda,q}(\lambda^{-1}z+q)+u_0(\lambda^{-1}z+q)-w(z)\Big)|\tilde{\varphi}_q(z)
\chi(\lambda^{-1}z)\frac{\partial w(z)}{\partial z_j}dz\\
\leq&C\lambda\int_{B_{2d\lambda}(0)}e^{w(z)}|u_0(\lambda^{-1}z+q))-u_0(q)+4\pi \mathfrak{M}(1-\theta)(\gamma(\lambda^{-1}z+q,q)-\gamma(q,q))|
\tilde{\varphi}_q(z)\chi(\lambda^{-1}z)\frac{\partial w(z)}{\partial z_j}dz\\
\leq&C\lambda\int_{B_{2d\lambda}(0)}e^{w(z)}\lambda^{-1}|z|\tilde{\varphi}_q(z)\chi(\lambda^{-1}z)\frac{\partial w(z)}{\partial z_j}dz\\
=&O(\|\varphi_q\|_{L^\infty(\Omega)})=o(1).
\end{aligned}
\end{equation}
Moreover,  since  $(\varphi_q,S_q)$ is a fixed point of $\Psi$,  we have
\begin{equation}\label{rv0}
S_q=L^{-1}_2(g_{2,\lambda,\mu}(\varphi_q,S_q)).
\end{equation}
From the proof of Proposition \ref{contractionmap}, we know that  $\|g_{2,\lambda,\mu}(\varphi_q,S_q)\|_{L^2(\Omega)}=O(\mu+\frac{\mu^2(\ln\lambda)^4}{\lambda^3})$ since $(\varphi_q,S_q)\in M_{\lambda,\mu}$. From Theorem \ref{theoremb}, it follows $\|S_q\|_{L^2(\Omega)}\leq O(\frac{1}{\mu}+\frac{(\ln\lambda)^4}{\lambda^3})$.
Then, by the assumption $\lambda^2\ln\lambda<\mu$ and the similar way in (\ref{taylorargue}), we get that
\[
\Big| IV\Big|= O\Big(\frac{\lambda^2}{\mu}(1+\|\varphi_q\|_{L^\infty(\Omega)}+\|S_q\|^2_{L^2(\Omega)})+\lambda\|\varphi_q\|^2_{L^\infty(\Omega)}
+\frac{\lambda^3}{\mu^2}\|S_q\|_{L^\infty(\Omega)}\|S_q\|^2_{L^2(\Omega)}\Big)+O(\lambda\theta)=o(1).
\]
We recall $h(\varphi_q,S_q) = e^{U_{\lambda,q}+u_0}(1+\varphi_q)+S_q$. In the estimation in \eqref{IV}, the assumption $\lambda^2\ln\lambda<\mu$ yields that
\begin{equation}
\begin{aligned}\label{re3}
\Big|V\Big|&= O(\lambda\|{\varphi}_q\|_{L^\infty(\Omega)}(\frac{\lambda}{\mu}+\|{\varphi}_q\|_{L^\infty(\Omega)})
+\lambda^2\|S_q\|_{L^2(\Omega)})\\
&=o(1).
 \end{aligned}
\end{equation}
From the above estimates, we can derive (\ref{finalclaim}).
\end{proof}
\medskip

 From Lemma \ref{reducedproblem}, we can derive \eqref{finalclaim}.  Since we assume that  $D u_0(\hat{q})=0$ and  $D^2(u_0)(\hat{q})$ is nondegenerate,  from Lemma \ref{reducedproblem}, we can find a point $q$ near $\hat{q}$ such that the right hand side of \eqref{finalclaim} is equal to zero. Together with Lemma \ref{lemmac3}, we can find $q$ satisfying $c_{q,j}=0$ for $j=1,2$. At this point, we   complete the proof of Theorem \ref{thm1}.

  \qed

\section{Proof of Theorem \ref{5thm}}\label{5sec5}
In this section, we are going to construct blow up solutions of (\ref{eq2}) at the vortex point satisfying
$\sup u_{\lambda,\mu}\geq -c_0>-\infty.$
Based on Theorem \ref{BrezisMerletypealternatives}, our construction in this section was inspired by the arguments in section \ref{sec4} and the construction in \cite{LY1} where the authors construct blow up solutions at the vortex point using an entire solution for the  Chern-Simons equation, which has a singularity, as the building blocks.

 We recall the equation (\ref{eq2}) as follows:
\begin{equation*}
\left\{\begin{array}{l}
\Delta (u+\frac{N}{\mu})=-\lambda^2 e^{u+u_0}\left(1-\frac{N}{\lambda}\right)+4\pi\mathfrak{M},\\
\Delta \frac{N}{\lambda}=\mu (\mu+\lambda e^{u+u_0})\frac{N}{\lambda}- \mu(\lambda+\mu)e^{u+u_0}
\end{array}
\right. \mbox{ in }\Omega.
\end{equation*}
Throughout this section, we assume that $\mathfrak{M}>4$.
First of all, we are going to define the approximate solutions for (\ref{eq2}). Let  $V$ be  the radially symmetric solution of
\begin{equation}\label{5decay}
\left\{\begin{array}{l}
\Delta V+|x|^2e^{V}(1-|x|^2e^V)=0\ \textrm{in } \R^2,\\
V'(|x|)\rightarrow -\frac{2\mathfrak{M}}{|x|}+\frac{a_1(2\mathfrak{M}-4)}{|x|^{2\mathfrak{M}-3}}+O(\frac{1}{|x|^{2\mathfrak{M}-1}}), \  |x|\gg1,              \\
V(|x|)=-2\mathfrak{M}\ln |x| +I_1-\frac{a_1}{|x|^{2\mathfrak{M}-4}}+O(\frac{1}{|x|^{2\mathfrak{M}-2}}),   \ |x|\gg1,
\end{array}
\right.
\end{equation}where
$a_1$ and $I_1$ are constants (see \cite[Theorem 2.1, Lemma 2.6]{CFL} for the existence of $V$ satisfying \eqref{5decay}).
We set
\begin{equation}\label{5Uq}
U_{\lambda}(y)=\left\{\begin{array}{l}
V(\lambda|y-p_1|)+4\pi \mathfrak{M}(\gamma(y,p_1)-\gamma(p_1,p_1))(1-\theta)+C_{\lambda}, \ y\in B_{d}(p_1),\\
V(d\lambda)+4\pi \mathfrak{M}(G(y,p_1)-\gamma(p_1,p_1)+\frac{1}{2\pi}\ln d)(1-\theta)+C_{\lambda}, \ y \in \Omega /B_{d}(p_1),
\end{array}
\right.
\end{equation}
where $C_{\lambda}=2\ln\lambda+4\pi\left(\gamma(p_1,p_1)+\sum_{j=2}^Nm_jG(p_1,p_j)\right)$ and  \[\theta=\frac{1}{2\mathfrak{M}\lambda^{2\mathfrak{M}-4}}\left\{\frac{a_1(2\mathfrak{M}-4)}{d^{2\mathfrak{M}-4}}+O(\frac{1}{\lambda^2})\right\},\] which makes $U_{\lambda}\in C^1(\Omega)$, We would find a solution of \eqref{eq2} with the following form:
\begin{equation}\label{5error}
u+\frac{N}{\mu}=U_{\lambda}+\varphi\ \ \textrm{and}\ \ \frac{N}{\lambda}=  e^{U_{\lambda}+u_0+\varphi}+S,
\end{equation}
here $(\varphi,S)$ would be regard as an error term.

We note that   $e^{U_\lambda}=O(\lambda^2)$ in \eqref{5Uq}, but $e^{U_{\lambda,q}}=O(1)$  in the section \ref{sec4}.   In order to control  the difficulties  arising from  the error parts related to $\varphi^2$ term
we need to make the difference between  \eqref{error} and \eqref{5error}.

For the convenience, we also denote
\begin{equation}\label{5f}
\begin{split}
V_{\lambda}(y)&=V(\lambda|y-p_1|),\ \ h(\varphi,S)=e^{U_{\lambda}+u_0+\varphi}+S, \\
F(t)&=t^2e^{V(t)}(1-t^2e^{V(t)})\ \ \textrm{and}\ \ f(t)=t^2e^{V(t)}(1-2t^2e^{V(t)}).
\end{split}
\end{equation}
The  equation (\ref{eq2}) is reduced to a system for $(\varphi,S)$:
\begin{equation}\label{5eq_4}
\left\{\begin{array}{l}
\Delta \varphi+\lambda^2 f(\lambda|x-p_1|)\cdot1_{B_{2d}(p_1)}\varphi=h_{1,\lambda,\mu}(\varphi,S),\\
\Delta S-\mu^2S=h_{2,\lambda,\mu}(\varphi,S),
\end{array}
\right.
\end{equation}
where
\begin{equation}\begin{aligned}\label{5g1g2}
h_{1,\lambda,\mu}(\varphi,S):=&\quad-\Delta U_{\lambda} + \lambda^2f(\lambda|x-p_1|)\cdot1_{B_{2d}(p_1)}\varphi  + 4\pi\mathfrak{M}\\
&\quad -\lambda^2e^{U_{\lambda}+u_0+\varphi-\frac{\lambda}{\mu}h(\varphi,S)}(1-e^{U_{\lambda}+u_0+\varphi}-S),\\
h_{2,\lambda,\mu}(\varphi,S):=&\quad-\Delta e^{U_{\lambda}+u_0+\varphi} + \mu^2\left(1+\frac{\lambda}{\mu}e^{U_{\lambda}+u_0+\varphi-\frac{\lambda}{\mu}h(\varphi,S)}\right)\left\{e^{U_{\lambda}+u_0+\varphi}+S\right\}-\mu^2S\\
&\quad -\mu^2(1+\frac{\lambda}{\mu})e^{U_{\lambda}+u_0+\varphi-\frac{\lambda}{\mu}h(\varphi,S)}.
\end{aligned}\end{equation}
For a small constant $0<\alpha<\frac{1}{2}$, recall that
\begin{equation}
\rho(z)=(1+|z|)^{1+\frac{\alpha}{2}}, \quad \textrm{and} \quad \bar{\rho}(z)=\frac{1}{(1+|z|)(\ln(2+|z|))^{1+\frac{\alpha}{2}}}.
\end{equation}
We say that $\psi \in X_{\alpha}$ if
\begin{equation}
\|\psi\|_{X_{\alpha}}^2=\|(\Delta \tilde{\psi})\rho\|^2_{L^2(B_{2d\lambda}(0))}+\|\tilde{\psi}\bar{\rho}\|_{L^2(B_{2d\lambda}(0))}^2+\||\Delta \psi|^2+\psi^2\|_{L^1(\Omega /B_d(p_1))}<+\infty
\end{equation}
where $\tilde{\psi}(z)=\psi(\lambda^{-1}z+p_1)$, and that $\psi \in Y_{\alpha}$ if
\begin{equation*}
\|\psi\|_{Y_{\alpha}}^2= \frac{1}{\lambda^4}\|\tilde{\psi}\rho\|^2_{L^2(B_{2d\lambda}(0))}+\|\psi\|^2_{L^2(\Omega \setminus B_d(p_1))}<+\infty.
\end{equation*}
We note that $\|\cdot\|_{X_{\alpha}}$ and $\|\cdot\|_{Y_{\alpha}}$ are similar to the norms $\|\cdot\|_{X_{\alpha,q}}$ and $\|\cdot\|_{Y_{\alpha,q}}$ in section \ref{sec4},  but the scaled area is different.
We recall the preliminary results for the linear operator $L_{1}$ in \cite{LY1}, where
\begin{equation*}
L_{1}(\varphi):=\Delta \varphi+\lambda^2f(\lambda|y-p_1|)\cdot1_{B_{2d}(p_1)}\varphi.\end{equation*}

\begin{theorem}[Theorem B.1 in \cite{LY1}]\label{5theorema1}
 $ L_{1} $ is an isomorphism from $X_{\alpha}$  to $Y_{\alpha}$. Moreover, if  $w \in X_{\alpha}$ and  $h  \in Y_{\alpha}$  satisfy $L_{1} w=h$, then there is a constant $C>0$, independent of $\lambda>0$, such that
 \begin{equation}
\left\|w \right\|_{L^{\infty}(\Omega)}+\left\|w \right\|_{ X_{\alpha }} \leq C\left(\ln \lambda\right)\left\|h \right\|_{Y_{\alpha} }.
\end{equation}
\end{theorem}
Next, let us consider the corresponding nonlinear problem. We define an operator $\Psi$ by
\[
\Psi(\varphi,S) = \Big(L_{1}^{-1}(h_{1,\lambda,\mu}(\varphi,\hat{S})), \hat{S}\Big),
\]
where $\hat{S}=L^{-1}_2(h_{2,\lambda,\mu}(\varphi,S))$, and a subset $C_{\lambda,\mu}$ of $X_{\alpha}\times W^{2,2}(\Omega)$ by
\begin{equation*}\begin{aligned}
C_{\lambda,\mu}=\left\{(\varphi,S)\in X_{\alpha}\times W^{2,2}(\Omega)\ \ \Big| \ \   \|(\varphi,S)\|_*\leq (\ln\lambda)^{-3}  \right\}.
\end{aligned}\end{equation*}
where
\begin{equation*}
\|(\varphi,S)\|_*:= \|\varphi\|_{L^\infty(\Omega)}+\|\varphi\|_{X_{\alpha}}+\frac{\lambda}{\mu^2(\ln\lambda)^3}(\mu^2\|S\|_{L^2(\Omega)}+\mu\|S\|_{L^\infty(\Omega)}+\|S\|_{W^{2,2}(\Omega)}).
\end{equation*}
We note that if $(\varphi,S)\in C_{\lambda,\mu}$, then \begin{equation*} \|\varphi\|_{L^\infty(\Omega)}+\|\varphi\|_{X_{\alpha}}\leq  {(\ln\lambda)^{-3}}, \ \textrm{and}\
 \mu^2\|S\|_{L^2(\Omega)}+\mu\|S\|_{L^\infty(\Omega)}+\|S\|_{W^{2,2}(\Omega)} \leq \frac{\mu^2}{\lambda}.\end{equation*}
The following estimation would be important for the contraction argument.
\begin{lemma}\label{5deltaeUqu}
There exists a constant $C$ such that
\[
\|\Delta \{e^{U_{\lambda}+u_0+\varphi}\}\|_{L^2(\Omega)}\leq C\lambda^3(1+\|\varphi\|_{L^\infty(\Omega)}+\|\varphi\|_{X_{\alpha}}).
\]
for any $(\varphi,S)\in C_{\lambda,\mu}$.
\end{lemma}
\begin{proof} Although we have
$e^{U_{\lambda,q}}=O(1)$ from \eqref{Uq} in the section \ref{sec4},  we note that $e^{U_\lambda}=O(\lambda^2)$ from \eqref{5Uq}. Except this observation, we can follow the arguments in the proof of Lemma \ref{deltaeUqu}, and obtain Lemma \ref{5deltaeUqu}. We skip the detail.
\end{proof}

\noindent
{\bf Completion of the proof of Theorem \ref{5thm}}.
First of all, we claim that there exists a fixed point $(\bar{\varphi},\bar{S})\in C_{\lambda,\mu}$ of the operator $\Psi$.

As in the proof of Proposition \ref{contractionmap},  Lemma \ref{5deltaeUqu} and \eqref{5Uq} imply that there is a constant $C>0$ satisfying
\begin{equation}\begin{aligned}\label{5estimateg2}
\|h_{2,\lambda,\mu}\|_{L^2(\Omega)}
\leq  C \left\{\mu\lambda^4(1+\|\varphi\|_{L^\infty(\Omega)})+\lambda^3\|\varphi\|_{X_{\alpha}}+\lambda^3\|S\|_{L^2(\Omega)}(\mu+\lambda\|S\|_{L^\infty(\Omega)})\right\},
\end{aligned}\end{equation}
and\begin{equation}\label{5estimateg1}
\begin{aligned}
\|h_{1,\lambda,\mu}(\varphi,\hat{S})\|_{Y_{\alpha}}
&\leq C\left\{\frac{1}{\lambda}+\frac{\lambda^3}{\mu}(1+\|\varphi\|_{L^\infty(\Omega)})
+\|\varphi\|_{L^\infty(\Omega)}^2+\lambda\|\hat{S}\|_{L^2(\Omega)} \left(1+\frac{\lambda^2}{\mu^2}\|\hat{S}\|_{L^\infty(\Omega)}\right)\right\}.
\end{aligned}\end{equation}
We remark that the difference between (\ref{estimateg2})-\eqref{estimateg1}  and (\ref{5estimateg2})-\eqref{5estimateg1} comes from the setting of solution in \eqref{error} and \eqref{5error} in addition to Lemma \ref{deltaeUqu}  and   Lemma \ref{5deltaeUqu}.
From Theorem \ref{5theorema1} and \ref{theoremb}, the inequalities (\ref{5estimateg2}) -(\ref{5estimateg1})  and the assumption $1\ll(\ln\lambda)^5\lambda^5\ll\mu$ yield that $\Psi(\varphi,S)\in C_{\lambda,\mu}$ for any $(\varphi,S)\in C_{\lambda,\mu}$.

Similarly,  we can also get that if   $1\ll(\ln\lambda)^5\lambda^5\ll\mu$ and $(\varphi_1,S_1),  (\varphi_2,S_2) \in C_{\lambda,\mu}$, then
\begin{equation}\begin{aligned}\label{5diffg2}
&\|h_{2,\lambda,\mu}(\varphi_1,S_1)- h_{2,\lambda,\mu}(\varphi_2,S_2)\|_{L^2(\Omega)}
\\&=O(\mu\|\varphi_1-\varphi_2\|_{L^\infty(\Omega)}+\lambda^3\|\varphi_1-\varphi_2\|_{X_{\alpha}} +\lambda\mu\|S_1-S_2\|_{L^2(\Omega)}),
\end{aligned}\end{equation}and
\begin{equation}\label{5diffg1}
\begin{aligned}
&\|h_{1,\lambda,\mu}(\varphi_1,\hat{S}_1)-h_{1,\lambda,\mu}(\varphi_2,\hat{S}_2)\|_{Y_{\alpha}}
\\
&=O((\ln\lambda)^{-3}\|\varphi_1-\varphi_2\|_{L^\infty(\Omega)}
+\lambda\|\hat{S}_1-\hat{S}_2\|_{L^2(\Omega)}).
\end{aligned}\end{equation}
The estimations \eqref{5diffg2}-\eqref{5diffg1}, Theorem \ref{5theorema1}, and Theorem \ref{theoremb} imply that  if  $(\varphi_1,S_1), \ (\varphi_2,S_2) \in  C_{\lambda,\mu}$, then there exists  a constant $0<\tau<1$ satisfying \begin{equation}\begin{aligned}\label{5claims2}
\|\Psi(\varphi_1,S_1)- \Psi(\varphi_2,S_2)\|_*<\tau\|(\varphi_1,S_1)- (\varphi_2,S_2)\|_*.
\end{aligned}\end{equation}

In view of  the  contraction mapping
theorem, we get that if  $1\ll(\ln\lambda)^5\lambda^5\ll\mu$,
there exists $(\bar{\varphi}, \bar{S})\in C_{\lambda,\mu}$   satisfying
\begin{equation}\label{5final}
\left\{\begin{array}{l}
\Delta \bar{\varphi}+\lambda^2 f(\lambda|y-p_1|)\cdot1_{B_{2d}(p_1)}\bar{\varphi}=h_{1,\lambda,\mu}(\bar{\varphi},\bar{S}),\\
\Delta \bar{S}-\mu^2\bar{S}=h_{2,\lambda,\mu}(\bar{\varphi},\bar{S}).
\end{array}
\right.
\end{equation} We note that  $\left(u_{\lambda,\mu}, N_{\lambda,\mu}\right):=\left(U_{\lambda}+\bar{\varphi}-\frac{\lambda}{\mu}\left( e^{U_{\lambda}+u_0+\bar{\varphi}}+\bar{S}\right), \lambda\left( e^{U_{\lambda}+u_0+\bar{\varphi}}+\bar{S}\right)\right)$ satisfies the system \eqref{eq2}, and thus     complete the proof of Theorem \ref{5thm}.
  \qed

\bigskip

\noindent{\bf Acknowledgement}\\
   W. Ao was supported by NSFC (No. 11801421 and No. 11631011). O. Kwon was supported by  Young Researcher Program through the National Research
Foundation of Korea (NRF) (No. NRF-2016R1C1B2014942).
Y. Lee was supported by the National Research Foundation of Korea(NRF) grant funded by the Korea government(MSIT) (No. NRF-2018R1C1B6003403).


\begin{thebibliography}{1}
\bibitem{Ab} A.A. Abrikosov, On the magnetic properties of superconductors of the second
group, Sov. Phys. JETP 5, 1174-1182 (1957).
\bibitem{BCCT}
D. Bartolucci, C.-C. Chen, C.-S. Lin, G. Tarantello, Profile of blow-up solutions to mean field equations with singular data, Comm. Partial
Differential Equations 29 (2004) 1241–1265.
\bibitem{BT} D.  Bartolucci, G. Tarantello,
    Liouville type equations with singular data and their applications
    to periodic multivortices for the electroweak theory.
    Comm. Math. Phys. 229, 3-47 (2002).
\bibitem{BBH2} F. Bethuel, H. Brezis, F. Helein, Ginzburg-Landau Vortices, Birkhauser, Boston, (1994).

\bibitem{B}  E. Bogomol’nyi, The stability of classical solutions, Sov. J. Nucl. Phys. 24, 449-454  (1976).
\bibitem{BGP}  A. Boutet de Monvel-Berthier, V. Georgescu, R. Purice, A boundary value problem related to the
Ginzburg-Landau model, Comm. Math. Phys. 142,  1-23  (1991).
\bibitem{BM} H. Brezis, F. Merle,
    Uniform estimates and blow-up behavior for solutions of $-\Delta u=V(x).e^u$ in two dimensions.
    Comm. Partial Differential Equations 16, 1223-1253 (1991).
\bibitem{CaY}  L.A. Caffarelli, Y. Yang, Vortex condensation in Chern-Simons-Higgs model: an existence theorem,
Comm. Math. Phys. 168, 321-336  (1995).
\bibitem{ChC} D. Chae, M. Chae, The global existence in the Cauchy problem of the Maxwell-Chern-
Simons-Higgs system, J. Math. Phys. 43,   5470-5482 (2002).
\bibitem{CC}  D. Chae, K. Choe, Global existence in the Cauchy problem of the relativistic Chern-Simons-Higgs
theory, Nonlinearity 15,  747-758  (2002).
\bibitem{ChK}  D. Chae, N. Kim, Topological multivortex solutions of the self-dual Maxwell-Chern-Simons-Higgs
system, J. Differential Equations 134,  154-182  (1997).

\bibitem{ChK2} D. Chae, N. Kim, Vortex condensates in the relativistic self-dual Maxwell-Chern-Simons-Higgs system,
RIM-GARC preprint 97-50, Seoul National University.


\bibitem{CI}  D. Chae, Y. Imanuvilov, The existence of non-topological multivortex solutions in the relativistic
self-dual Chern-Simons theory, Comm. Math. Phys. 215,  119-142 (2000).



\bibitem{CFL}  H. Chan, C.C. Fu, C.S. Lin, Non-topological multi-vortex solutions to the self-dual
Chern-Simons-Higgs equation, Comm. Math. Phys. 231, 189-221  (2002).


\bibitem{CL} W. Chen, C. Li, Qualitative properties of solutions to some nonlinear elliptic equations in $\mathbb{R}^2$, Duke Math. J. 71,  427-439  (1993).
\bibitem{CHMY}
X. Chen, S. Hastings, J.B. McLeod, Y. Yang, A nonlinear elliptic equation arising from gauge field theory and cosmology, Proc. Roy. Soc.
Lond. A 446,   453-478 (1994).


\bibitem{C1}  K. Choe, Existence of multivortex solutions in the self-dual-Higgs theory in a background metric,
J. Math. Phys. 42,  5150-5162  (2001).

\bibitem{C2} K. Choe,  Uniqueness of the topological multivortex solution in
the selfdual Chern-Simons theory. J. Math. Phys. {46}, 012305 21pp (2005).



\bibitem{C3}  K. Choe,
    Asymptotic behavior of condensate solutions in the Chern-Simons-Higgs theory.
    J. Math. Phy. 48, 103501 (2007).


\bibitem{CK} {  K. Choe, N. Kim,  Blow-up solutions of the self-dual
Chern-Simons-Higgs vortex equation.} {Ann. Inst. H. Poincar\'{e} Anal. Non
Linaire} {25}, 313-338 (2008).





\bibitem{DJLW}  W. Ding, J. Jost, J. Li, G. Wang, An analysis of the two-vortex case in the Chern-Simons Higgs
model, Calc. Var. Partial Differential Equations 7, 87-97  (1998).
\bibitem{DJLW2}  W. Ding, J. Jost, J. Li, G. Wang, Multiplicity results for the two-sphere Chern-Simons Higgs model
on the two-sphere, Comment. Math. Helv. 74,  118-142  (1999).
\bibitem{LJLPW}  W. Ding, J. Jost, J. Li, X. Peng, G. Wang, Self-duality equations for Ginzburg-Landau and
Seiberg-Witten type functionals with 6th order potentials, Comm. Math. Phys. 217,
383-407 (2001).


\bibitem{Du} Dunne, G. Self-dual Chern-Simons theories. Lecture Notes in Physics, New series m, Monographs,
m36. Springer, New York, (1995).




\bibitem{FLL} Y.W. Fan, Y. Lee, C.S. Lin,  Mixed type   solutions of the $SU(3).$   models on a  torus,   Comm. Math. Phys.  343,  Issue 1,     233-271 (2016).





\bibitem{GT} D. Gilbarg, N.S. Trudinger,
    Elliptic Partial Differential Equations of Second Order.
    {vol. 224,} {second ed.,} Springer, Berlin, (1983).
\bibitem{H}  J. Han, Asymptotics for the vortex condensate solutions in Chern-Simons-Higgs theory, Asymptotic
Anal. 28,  31-48  (2001).
\bibitem{H2}  J. Han, Asymptotic limit for condensate solutions in the Abelian Chern-Simons Higgs model, Proc.
Amer. Math. Soc. 131,  1839-1845  (2003).
\bibitem{H3}  J. Han, Asymptotic limit for condensate solutions in the Abelian Chern-Simons Higgs model II,
Proc. Amer. Math. Soc. 131,   3827-3832 (2003).
\bibitem{H4}  J. Han, Topological solutions in the self-dual Chern-Simons-Higgs theory in a background metric,
Lett. Math. Phys. 65,  37-47 (2003).
\bibitem{HK} J. Han, N. Kim,  Nonself-dual Chern-Simons and Maxwell-Chern-Simons vortices on bounded domains. J. Funct. Anal. 221, no. 1, 167-204  (2005).
\bibitem{HJ} J. Han, J. Jang, Self-dual Chern-Simons vortices on bounded domains, Lett. Math. Phys. 64,
45-56  (2003).
\bibitem{HKP}  J. Hong, Y. Kim, P.Y. Pac, Multivortex Solutions of the Abelian Chern-Simons-Higgs Theory, Phys.
Rev. Lett. 64,   2230-2233 (1990).

\bibitem{JW}  R. Jackiw, E.J. Weinberg, Self-dual Chen-Simons vortices, Phys. Rev. Lett. 64,  2234-2237  (1990)
\bibitem{JT}  A. Jaffe, C.H. Taubes, Vortices and Monopoles, Birkhauser, Boston, (1980).
\bibitem{Kim} S. Kim, Solitons of the self-dual Chern-Simons theory on a cylinder, Lett. Math. Phys. 61,
113-122 (2002).
\bibitem{KK}  S. Kim, Y. Kim, Self-dual Chern-Simons vortices on Riemann surfaces, J. Math. Phys. 43,
2355-2362  (2002).Fno
\bibitem{Ku}  K. Kurata, Existence of nontopological solutions for a nonlinear elliptic equation from
Chern-Simons-Higgs theory in a general background metric, Differential Integral Equations 14,
 925-935 (2001).
\bibitem{La} L. Landau, E. Lifschitz, The classical theory of fields. Addison-Wesley, Cambridge
MA, (1951).
\bibitem{LLM} C. Lee, K. Lee, H. Min, Self-dual Maxwell-Chern-Simons solitons, Phys. Lett. B 252,
79-83  (1990).




\bibitem{LY1}  C.S. Lin, S. Yan,
    Bubbling solutions for relativistic abelian Chern-Simons model on a torus.
    Comm. Math. Phys. 297, 733-758 (2010).

\bibitem{LY2}  C.S. Lin, S. Yan,
    Bubbling solutions for the $SU(3)$ Chern-Simons Model on a torus.
    Comm. Pure Appl. Math. 66, 991-1027 (2013).


\bibitem{NO}  H. Nielsen, P. Olesen, Vortex-Line models for dual strings, Nucl. Phys. B 61,
45-61  (1973).

\bibitem{Ni} L. Nirenberg,  Topics in nonlinear functional analysis. With a chapter by E. Zehnder. Notes
by R. A. Artino. Lecture Notes, 1973-1974. Courant Institute of Mathematical Sciences, New
York University, New York, (1974).

\bibitem{NT0} M. Nolasco, G. Tarantello, On a sharp type inequality on two dimensional compact manifolds, Arch. Rational Mech. Anal. 145 (1998) 161–
195.
\bibitem{NT}  M. Nolasco, G. Tarantello, Double vortex condensates in the Chern-Simons-Higgs theory, Calc.
Var. Partial Differential Equations 9, 31-94  (1999).

\bibitem{PR} F. Pacard, T. Riviere, Linear and nonlinear aspects of vortices. The
Ginzburg-Landau model. Progress in Nonlinear Differential Equations and
their Applications 39 Birkhauser Boston, Inc., Boston, MA, (2000).

\bibitem{R}  T. Ricciardi, Asymptotics for Maxwell-Chern-Simons multivortices, Nonlinear Anal. 50,
1093-1106  (2002).
\bibitem{RT}  T. Ricciardi, G. Tarantello, Vortices in the Maxwell-Chern-Simons theory, Comm. Pure Appl. Math.
53, 811-851 (2000).
\bibitem{S} J. Schiff, Integrability of Chern-Simons-Higgs and Abelian Higgs vortex equations in a background
metric, J. Math. Phys. 32,  753-761 (1991).
\bibitem{SY}  J. Spruck, Y. Yang, The existence of nontopological solitons in the self-dual Chern-Simons theory,
Comm. Math. Phys. 149,  361-376  (1992).
\bibitem{SY1}  J. Spruck, Y. Yang, Topological solutions in the self-dual Chern-Simons theory, Ann. Inst. H.
Poincaré Anal. Non Lineaire 12,  75-97 (1995).
\bibitem{ST}  M. Struwe, G. Tarantello, On multivortex solutions in Chern-Simons gauge theory, Boll. Uni. Mat.
Ital. Sez. B Artic. Ric. Mat. (8). 1,  109-121  (1998).



\bibitem{T}  G. Tarantello, Multiple condensate solutions for the Chern-Simons-Higgs theory, J. Math. Phys. 37,
 3769-3796 (1996).

\bibitem{T2}  G. Tarantello, Selfdual Maxwell-Chern-Simons vortices. Milan J. Math. 72, 29-80  (2004).

\bibitem{Ta}  C. H. Taubes, Arbitrary N-vortex solutions to the first order Ginzburg-Landau equations.
Comm. Math. Phys. 72,   no. 3, 277-292 (1980).



\bibitem{'tH} G. 't Hooft,  A property of electric and magnetic flux in nonabelian gauge theories.  Nucl. Phys. B153,   141-160 (1979).

\bibitem{W}  R. Wang, The existence of Chern-Simons vortices, Comm. Math. Phys. 137,  587-597  (1991).
\bibitem{WY}  S. Wang, Y. Yang, Abrikosov’s vortices in the critical coupling, SIAM J. Math. Anal. 23,
1125-1140  (1992).
\bibitem{Y} Y. Yang, Solitons in field theory and nonlinear analysis,   Springer Monograph in Mathematics,
Springer, New York, (2001).









\end{thebibliography}
\end{document}